\documentclass[11pt]{article}
\usepackage[french, english]{babel}
\usepackage[utf8]{inputenc}
\usepackage{upgreek}
\usepackage{amsmath,amsfonts,amssymb,amsthm,mathrsfs}
\usepackage{stackengine,graphicx,scalerel}

\usepackage{minitoc}
\usepackage{bbm}
\renewcommand{\leq}{\leqslant}
\renewcommand{\geq}{\geqslant}
\usepackage[mono=false]{libertine}
\useosf
\usepackage{mathpazo}

\newcommand{\crit}{\text{\tiny{cr}}}

\usepackage{wasysym}

\usepackage[dvipsnames]{xcolor}
\usepackage[a4paper,vmargin={3.5cm,3.5cm},hmargin={2cm,2cm}]{geometry}
\usepackage[font=sf, labelfont={sf,bf}, margin=1cm]{caption}
\usepackage{graphicx,graphics}
\usepackage{epsfig}
\usepackage{latexsym}
\usepackage{stmaryrd}
\usepackage{ae,aecompl}

\usepackage[english]{babel}
 \usepackage[colorlinks=true]{hyperref}
\usepackage{pstricks}
\usepackage{enumerate}
\usepackage{tikz}
\usepackage{fontawesome}
\usetikzlibrary{arrows,automata}

\DeclareFixedFont{\beaupetit}{T1}{ftp}{b}{n}{2cm}

\newcommand{\RW}{ {\normalfont\text{\tiny{RW}}}}
\newcommand{\K}{\mathrm{K}}

\newtheorem{example}{Example}
\newtheorem{theorem}{Theorem}
\newtheorem{definition}[]{Definition}
\newtheorem{proposition}[theorem]{Proposition}
\newtheorem{lemma}[theorem]{Lemma}
\newtheorem{corollary}[theorem]{Corollary}
\theoremstyle{definition}
\newtheorem{remark}[theorem]{Remark}

\def\llbracket{[\hspace{-.10em} [ }
\def\rrbracket{ ] \hspace{-.10em}]}

\title{\textsc{Universality for catalytic equations \\ and fully parked trees}}
\author{
Alice \textsc{Contat}\thanks{Universit\'e Sorbonne Paris Nord.\hfill  \href{mailto:alice.contat@math.cnrs.fr}{\texttt{alice.contat@math.cnrs.fr}}}
\qquad\&\qquad
Nicolas \textsc{Curien}\thanks{Universit\'e Paris-Saclay.\hfill  \href{mailto:nicolas.curien@gmail.com}{\texttt{nicolas.curien@gmail.com}}}
}

\date{}
\begin{document}
\maketitle 

\abstract{ We show that critical parking trees conditioned to be fully parked  converge in the scaling limits towards the Brownian growth-fragmentation tree, a self-similar Markov tree different from Aldous' Brownian tree recently introduced and studied in \cite{bertoin2024self}. As a by-product of our study, we prove that positive non-linear polynomial equations involving a catalytic variable display a universal polynomial exponent $5/2$ at their singularity, confirming a conjecture by Chapuy, Schaeffer and  Drmota \& Hainzl. Compared to previous analytical works on the subject, our approach is probabilistic and exploits an underlying random walk hidden in the random tree model.}

\section {Introduction}

\paragraph{Universality for polynomial equation with a catalytic variable.} Starting with the pioneer works of Tutte on the enumeration of planar maps \cite{Tutte:census3,Tut62,Tut63}, the idea of introducing a ``catalytic" variable to solve an equation involving generating functions has been an extremely fruitful idea. Formally those equations are of the form 
$$ P\big(x,y, F(x,y),F_1(x),F_2(x),... , F_{\K}(x)\big)=0,$$
where $P$ is a polynomial, $F(x,y)$ is the generating function of interest, and the series $F_i(x)$, which are part of the unknowns, are typically related to the coefficient of $y^i$ in $F(x,y)$. The variable $y$ is called the \textbf{catalytic} variable since its standard fate is to be eliminated to get $[y^0]F(x,y)$. Building on the quadratic and kernel methods, Bousquet-M\'elou \& Jehanne \cite{BMJ06} developed a systematic approach  to solve those equations. This technique is proved to always work  in the case where 
  \begin{eqnarray} \label{eq:1} F(x,y) = x Q\big(y,F, \Delta^{(1)} F,\Delta^{(2)} F, ... , \Delta^{(\K)} F\big),  \end{eqnarray} where $Q$ is polynomial and $\Delta^{(i)} F=\Delta^{(i)} F(x,y)$ are the \textbf{discrete differences} of $F(x,y)$ with respect to $y$ i.e. 
  \begin{eqnarray} \label{eq:discretederivative} \Delta^{(i)} F=\Delta^{(i)} F(x,y) := \dfrac{F(x,y) -\displaystyle \sum_{j=0}^{i-1} y^j[y^j]F(x,y)}{y^{i}}, \quad \mbox{ for }i \geq 1.  \end{eqnarray} In this case, $[y^{0}]F(x,y)$ is an algebraic generating function of $x$ amenable to singularity analysis \cite{FS09} via Puiseux's expansion. In particular, when the dominant singularity is unique, the coefficients of $F(x,y)$ satisfy
 \begin{eqnarray} \label{eq:expo} [x^n][y^0]F(x,y) \underset{n \to \infty}{\sim} C \cdot x_{\crit}^{-n}\cdot n^{-\alpha}, \quad \mbox{ for some } x_{\crit}>0,\quad  C>0 \mbox{ \ \ and } \alpha \in \mathbb{Q} \backslash \{1,2,3,...\} .  \end{eqnarray}
This method has been applied with tremendous success for many combinatorial problems such as enumeration of many (many) classes of planar maps \cite{bousquet2008rational}, Tamari intervals \cite{chapuy2024scaling}, description trees \cite{CoriSchaefferDescription} or fighting fishes \cite{duchi2017fightingbis} just to name a few. In all these cases, the exponent $\alpha$ popping-up in \eqref{eq:expo} is always  either $ \frac{3}{2}$, characteristic of ``tree'' enumeration,  or $ \frac{5}{2}$, characteristic of ``planar map" enumeration.  This surprising fact has been rigorously proved for equation \eqref{eq:1} with {positive coefficients} by Drmota, Noy \& Yu \cite{drmota2022universal} \textbf{in the case ${\K =1}$} (see also \cite{schaeffer2023universal,duchi2024order}) and extended to the case $ \K=2$, modulo an extra assumption in a technical tour de force by  Drmota \& Hainzl \cite{drmotaHainzl}. One of the outcome of our work is to extend this result for arbitrary discrete differences $\K \geq1$: 
 \begin{theorem} \label{thm:5/2} Let $F$ be the solution to Equation \eqref{eq:1} with a polynomial $Q$ having positive coefficients. We suppose that $Q$ obeys our standing assumptions on aperiodicity and branching conditions, see~(*) below. Then there exists $x_\crit \in (0, \infty)$ and $C >0$ so that 
 $$ [x^n][y^0]F(x,y) \sim C\cdot (x_\crit)^{-n} \cdot n^{-5/2}, \quad \mbox{ as }n \to \infty.$$
 \end{theorem}
The conditions of branching and aperiodicity are described below in Section \ref{sec:2.1}. The branching assumption implies that the equation is non-linear, whereas the aperiodicity condition is roughly equivalent to the fact that the degree of the catalytic variable may increase and decrease and is not contained in a strict sub-lattice. Both conditions will be enforced in the rest of the paper. Contrary to the analytical approach taken in \cite{drmota2022universal,drmotaHainzl} our techniques to prove the above theorem are \textbf{probabilistic} and heavily rely on \textbf{random walks} and {random trees} as well as their scaling limits, the stable L\'evy processes and self-similar Markov trees introduced in \cite{bertoin2024self}. In particular, the variable $y$ becomes an \textbf{explicative variable} instead of a discredited \textbf{catalytic one}.\\

Let us caricature the overall idea: A positive equation of the form \eqref{eq:1} is naturally associated with an integer-type Bienaym\'e--Galton--Watson tree, the (degree of the) catalytic variable $y$ representing the (non-negative) label while the (degree of the) variable $x$ accounts for the size of such trees. Such labeled trees have been encountered many times in the literature and especially recently in connection with the parking process on random trees. The fact that \eqref{eq:1} involves only discrete differences implies that the evolution of labels along long branches of those integer-type Bienaym\'e--Galton--Watson are intimately connected to discrete random walks with i.i.d.\ increments. In particular, the scaling limits of such trees which is a self-similar Markov tree in the sense of \cite{bertoin2024self}, must be intimately connected to stable L\'evy process of exponent $\alpha-1$. Among all (distributions of) stable self-similar Markov trees this singles out the cases  $\alpha \in \{3/2, 5/2\}$. The case $\alpha = 3/2$ corresponding to the ``subcritical'' case where the underlying random tree is the Brownian Continuum Random Tree (CRT), while $\alpha = 5/2$ corresponds to the critical case $x = x_{c}$ where the underlying random tree is the Brownian Growth-Fragmentation tree that shows up inside the Brownian sphere, \cite{BBCK18,BCK18,da2025growth,le2020growth}. Let us now develop our approach in more details:

 \paragraph{Parking on trees and their scaling limits.} The first step is to interpret the function $F(x,y)$ solving \eqref{eq:1} for polynomial $Q$ with positive coefficient as the generating series of a model of \textbf{parking on trees}. The study of {parking} on trees has witnessed many developments in recent years in probability \cite{aldous2023parking,contat2020sharpness,Contat21+,ConCurParking,curien2022phase,GP19} and in combinatorics \cite{chen2021enumeration,contat2022last,duchi2024order} after the pioneer work of Lackner \& Panholzer \cite{LaP16,panholzer2020parking}. In our case, the parking model we consider is as follows: Consider  a finite rooted plane tree $ \mathrm{t}$ with two labeling functions $$\ell^{ \mathrm{car}} : \mathrm{Vertices}( \mathrm{t}) \to \{0,1,2,...\} \quad \mbox{ and }\ell^{ \mathrm{spot}} : \mathrm{Edges}( \mathrm{t}) \to \{0,1,2,...\}.$$
 The label $\ell^{ \mathrm{car}}$ is interpreted as \textbf{car arrivals on the vertices} of $ \mathrm{t}$, whereas $\ell^{ \mathrm{spot}}$ are \textbf{parking spots capacities of the edges}. We then imagine that the cars arrive on their arrival vertex and go down the tree to park on the edges (according to their parking capacities) as soon as possible, see Figure~\ref{fig:parking1} for an illustration. An Abelian property of the model shows that the resulting configuration does not depend on the order in which the cars have been parked. We say that $$ \mathbbm{t} = ( \mathrm{t}, \ell^{ \mathrm{car}},\ell^{ \mathrm{spot}})$$ is \textbf{fully parked} if there are no available spots on the edges after the parking process. We shall only consider such situations below and for each $u \in  \mathrm{Vertices}(\mathrm{t})$ we denote by $\phi_{ \mathbbm{t}}(u)$ the \textbf{flux} of cars at~$u$, i.e. the number of cars that exited the vertex $u$ towards the root of the tree (if $u = \varnothing$ is the root of the tree, this corresponds to the number of cars that did not find any parking spot). We then denote by  $$ \mathbf{t}= ( \mathrm{t}, \phi) \equiv ( \mathrm{t}, ( \phi_{ \mathbbm{t}}(u))_{u \in \mathrm{t}}),$$ the obtained $  \mathbb{Z}_{\geq 0}$-labeled tree, see Figure \ref{fig:parking1}. 
 
 \begin{figure}[!h]
  \begin{center}
  \includegraphics[width=13cm]{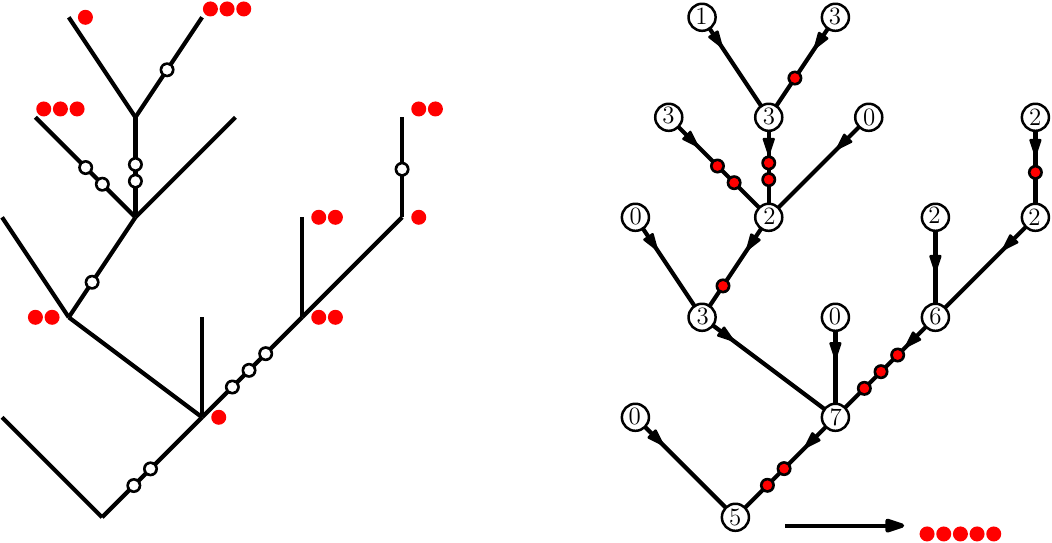}
  \caption{Left: An example of a parking tree decorated with car arrivals (in red) and parking spots on the edges (in white). Right: After the parking process, the resulting tree is fully parked (no spot is empty) and gives rise to a $ \mathbb{Z}_{\geq0}$-labeled tree where the label of a vertex is the number of cars that exited this vertex downwards. \label{fig:parking1}}
  \end{center}
  \end{figure}
  We shall denote by $ \mathrm{FPT}_p^n$ the set of all fully parked trees with outgoing flux $\phi_{ \mathbbm{t}}(\varnothing) =p$ and $n$ vertices in total. We shall now put a measure on those objects using weights. A \textbf{local weight function} is collection of non-negative numbers 
$$ w_{c,k ,(s_1, ... , s_k)} \quad \mbox{ for } c,k,s_1,... , s_k \in  \mathbb{Z}_{\geq 0}, $$ and we write $w_{c,0, \varnothing}$ when $k=0$.
This can be used to produce a sigma-finite measure $ \mathrm{w}$ on the set of all finite fully parked trees by setting 
 \begin{eqnarray} \label{def:measurew}  \mathrm{w}( \mathbbm{t})=\mathrm{w}( \mathrm{t}, \ell^{ \mathrm{car}},\ell^{ \mathrm{spot}}) = \prod_{u \in \mathrm{t}}w_{ \ell^{ \mathrm{car}}(u),k_u (\mathrm{t}) ,(\ell^{ \mathrm{spot}}(e_i(u, \mathrm{t})))_{1 \leq i \leq k_{u}(  \mathrm{t})}},  \end{eqnarray} where $k_u (\mathrm{t})$ is the number of children of $u$ in $ \mathrm{t}$ and $e_1(u, \mathrm{t}), ..., e_{k_u( \mathrm{t})}(u, \mathrm{t})$ are the edges above $u$ enumerated from left to right. For $x \geq 0$ and $p \in \{0,1,2, ...\}$, we form the \textbf{partition functions} of the measure $w$ by putting 
  \begin{eqnarray} \label{def:wpx} W_p^x = \sum_{n \geq 1}\sum_{ \mathbbm{t} \in \mathrm{FPT}_p^n} x^n \cdot   \mathrm{w}( \mathbbm{t}).	 \end{eqnarray}	
  A priori, this defines a family of formal power series $ \in \mathbb{R}\llbracket x \rrbracket$. It is then easy to see that for any polynomial $Q$ with positive coefficients, there exists a local weight function $w_{c,k,(s_1, ... , s_k)}$ so that the solution $F$ to \eqref{eq:1} satisfies 
$$ [y^p]F(x,y) = W_p^x, \quad \mbox{ for any } p \in \{0,1,2, ... \},$$ at least in terms of formal power series.   See Section \ref{sec:2.1} for details, see also \cite{duchi2024order}. 

\begin{remark}[On the case $\K=1$] Notice that the case $ { \K=1}$ corresponds to the case when there are at most one parking spot per edge. In \cite{drmota2022universal} an ingenious change of variable maps the problem to a finite positive system of polynomial equations (in probabilistic terms, a Bienaym\'e--Galton--Watson tree with finitely many types) from which they derive Theorem \ref{thm:5/2} using the Drmota--Lalley--Woods theorem, see also \cite{schaeffer2023universal}. This change of variable has   recently been understood bijectively in \cite{duchi2024order} using a redirection of the edges of the tree. Such redirections were previously  used to couple the parking process on trees to a random configuration model, see \cite{ConCurParking,Contat21+}. In particular, the critical threshold for the parking phase transition is explicitly known in the case $ { \K=1}$, see \cite{curien2022phase,contat2020sharpness,Contat21+}, showing once more that the case $\K=1$ is special compared to the general case.
\end{remark}

Coming back to the partition functions $ W_p^x$, we then show (Lemma \ref{lem:xcyc}) that there exists $x_\crit>0$ such that $W_p^x < \infty$ for all $p$'s if and only if $x \leq x_\crit$. In particular, for any $x \in (0,x_{\crit}]$, one can consider the probability measure   \begin{eqnarray} \label{def:boltz} \frac{1}{W_p^x} x^{ \# \mathrm{t}}  \mathrm{w} ( \mathrm{d} \mathbbm{t}),  \end{eqnarray} on $\cup_{n \geq 1} \mathrm{FPT}^n_p$ which is called in these pages the $x$-\textbf{Boltzmann law} (critical if $ x=x_{\crit}$) and denoted by $ \mathbb{P}_p^x$ (it is part of our standing assumption that $W_{p}^{x}>0$ for all $p \geq 0$ and all $x>0$). After a pushforward by the parking process
$$\mathbbm{t} = ( \mathrm{t}, \ell^{ \mathrm{car}},\ell^{ \mathrm{spot}}) \longmapsto  \mathbf{t} = ( \mathrm{t}, (\phi(u))_{u \in \mathrm{t}}),$$
we observe that under $ \mathbb{P}^{x}_p$, the $ \mathbb{Z}_{\geq 0}$-labeled tree $( \mathrm{t}, \phi)$ is a \textbf{multi-type Bienaym\'e--Galton--Watson tree} where types are in $ \mathbb{Z}_{\geq 0}$ and whose offspring distribution is explicitly given in terms of $(W_p^x : p \geq 0)$ and the local weight function $w$, see Proposition \ref{prop:GW}. In particular, the large scale behavior of those trees is governed by an exponent $\beta^{x}$ appearing in the asymptotic of $W^{x}_{p}$ 
  \begin{eqnarray} W_{p}^{x} \quad \underset{p \to \infty}{\sim}  \mathrm{C}^{x} (y_{\crit}^x)^{-p} \cdot p^{-\beta^{x}}, \label{def:betax}  \end{eqnarray} for some $C^x>0$ (see Lemma \ref{lem:asympW} and Proposition \ref{prop:nu_proba}). Although the existence of such expansions is straightforward using analytic combinatorics, the value of the exponent $\beta^{x}$ (as it was the case for $\alpha$ above) is  a priori \textit{unknown}\footnote{For each model its value can be algorithmically computed, but may a priori vary from one model to the other.}, and we shall pin it using the scaling limits of the above random trees, more precisely of the evolution of labels along their branches.

  \paragraph{An underlying random walk.} As we said above, the key to our approach is to use a \textbf{random walk} hidden in the parking model on trees. Such random walks have already been identified in the enumeration of planar maps, see \cite{CLGpeeling}, the inspiring works of Timothy Budd \cite{Bud15,BudOn} or the  monograph \cite[Section 5.1]{CurStFlour}. Our work can be seen as a vast generalization and unification of those results. To describe this random walk, consider the random decorated tree $ \mathbf{t}=( \mathrm{t}, \phi)$ under $ \mathbb{P}^{x}_p$ as $p \to \infty$. We prove in Section \ref{sec:rw} that, among the children of the root vertex labeled $p$, typically only one of them has a label of order $p$, whereas the other ones are of order $ \mathcal{O}_{ \mathbb{P}_{p}^{x}}(1)$ as $p \to \infty$. More precisely, we show in Proposition \ref{prop:nu_proba} that the difference between $p$ and the maximal label of the children of the root converges in distribution under $ \mathbb{P}_p^x$ as $p \to \infty$ towards a probability measure $\nu^{x}$ which, under our standing hypotheses, is supported by $\{... , -3,-2,-1,0,1,... , \K \}$  and satisfies, thanks to \eqref{def:betax} that 
$$ \nu^{x}(-k) \sim \mathrm{Cst}^{x} \cdot k^{-\beta^{x}} \mbox{ as } k \to \infty, \quad \mbox{ for some } \mathrm{Cst}^{x}>0,$$ 
where $\beta^{x}$ is the unknown exponent appearing in the asymptotic of $W_{p}^{x}$. Using an intrinsic property of the model, which can be thought of as an invariance under translation, we are able to describe the evolution of the labels along the locally largest branch of our trees as a $h$-transformation of a $\nu^{x}$-random walk where the function $h$ is merely $W_{p}$ itself. This is done in Proposition \ref{prop:llrw} and is named in these pages the ``Key formula'' because of its uttermost importance. For example, using the Key formula we are able to prove (without any calculations) that  as soon as $\nu^{x}$ has a finite first moment, \textbf{it must be centered}, so that random walks with i.i.d.\ increments of law $\nu^{x}$ converge in the scaling limit towards the spectrally negative $(\beta^{x}-1)$-stable L\'evy process. Passing then the Key formula to the scaling limit, we obtain an equation (see Eq. \eqref{eq:keyformulalimit} below) involving the exponent $\beta^{x}$ which enables us to pin-down its value:
\begin{equation}\label{eq:beta3252} \beta^{x} \in \left\{ \frac{3}{2}, \frac{5}{2}\right\}.\end{equation}
In fact, the case $ \beta^{x} = \frac{3}{2}$ corresponds to the subcritical case $x < x_{\crit}$ whereas $ \beta^{x} = \frac{5}{2}$ corresponds to the critical case $ x= x_{\crit}$, see Proposition \ref{prop:dichotomyxxc}. We insist on the fact that although the inspiration is taken from the theory of self-similar Markov trees (see below), the derivation of \eqref{eq:beta3252} is self-contained and does not require prior knowledge on continuous random trees, but still uses fine properties of stable L\'evy processes, in particular, their Lamperti representation. 
  
  Gathering-up the pieces, we prove the following universal asymptotics:

\begin{corollary}[Universality of asymptotics for partition functions] \label{cor:3/25/2} Under our standing assumptions~$(*)$, the functions $x \mapsto [y^p]F(x,y)$ for $p \geq 0$ have a common radius of convergence $x_{\crit} \in (0,\infty)$. For $ x \in (0, x_ \crit]$, if $y_{\crit}^x$ is the radius of convergence of $y\mapsto F(x_{\crit},y)$ then $y_{\crit}^x \in (0,\infty)$ and $F(x,y^x_{\crit}) < \infty$. Furthermore, for each $x \in (0, x_\crit]$ there exists some constants $C^x >0$ such that
$$ \mbox{for } x <x_{\crit}, \quad  W_p^x := [y^p] F(x,y) \sim \mathrm{C}^x \cdot (y_\crit^x)^{-p} \cdot p^{-3/2}, \quad \mbox{ as }p \to \infty,$$
$$ \mbox{for } x =x_{\crit}, \quad  W_p^{x_{\crit}} := [y^p] F(x_{\crit},y) \sim \mathrm{C}^{x_{\crit}} \cdot (y_{\crit}^{x_\crit})^{-p} \cdot p^{-5/2}, \quad \mbox{ as }p \to \infty.$$
\end{corollary}

\paragraph{Self-Similar Markov trees.} At this point, the reader may think that the exponent $\beta^{ x_{\crit}} = 5/2$ is the same as the exponent $\alpha$ we are looking for in Theorem \ref{thm:5/2}, but this is merely a coincidence! The derivation of the exponent $\alpha$ needs first to understand the typical \textbf{volume} of a tree under $ \mathbb{P}_{p}^{x}$, or in analytic terms to understand $\partial_{x} W_{p}^{x}$. We in fact prove a full scaling limit result for the labeled trees using the powerful invariance principle recently established in \cite{bertoin2024self}.
Indeed, the integer-type Galton--Watson trees discussed above exactly fall in the setting of the recent monograph \cite{bertoin2024self}, and it is natural to conjecture that they converge in the scaling limit i.e.~under $ \mathbb{P}^{x}_{p}$ and after rescaling of the distances by some power of $p$ as $p \to \infty$ to a natural self-similar Markov tree. A \textbf{self-similar Markov tree} is a random real rooted tree $ ( \mathcal{T}, \rho)$ given with a decoration $ g : \mathcal{T} \to \mathbb{R}_+$ which starts from $1$ at the root $\rho$ and is positive on its skeleton. Such trees come equipped with a natural mass measure $\mu_\mathcal{T}$ which is usually supported by the leaves of $\mathcal{T}$. Those random decorated trees  enjoy both a Markov and a self-similar property and  are believed to be all the possible scaling limits of multi-type Galton--Watson tree with type in $ \mathbb{Z}_{\geq 0}$. A special family of such random trees $ \boldsymbol{\mathcal{T}}_{\gamma} = ( \mathcal{T}_\gamma, g_\gamma, \mu_{\mathcal{T}_\gamma})$ was exhibited in \cite[Section 5.5]{bertoin2024self} in relation with spectrally negative stable processes of index $\gamma \in (0,2)$. In particular, $ \boldsymbol{\mathcal{T}}_{1/2}$ is nothing but a decorated version of the famous Brownian CRT of Aldous, and $ \boldsymbol{ \mathcal{T}}_{3/2}$ is the Brownian growth-fragmentation tree which already appeared inside the Brownian sphere/disk and was conjectured to be the scaling limits of the parking trees appearing in our previous works \cite{ConCurParking,aldous2023parking}. Those two trees are in fact distinguished members of the family $ \boldsymbol{\mathcal{T}}_{\gamma}$ since they are the only ones for which the evolution along ``typical'' branches is exactly given by a version of the $\gamma$-stable L\'evy process. This striking property was first observed in a different context by Miller \& Sheffield \cite{MS15}, later made more explicit by Bertoin, Budd, Curien, Kortchemski in \cite[Section 5]{BBCK18} and finally explained in the context of self-similar Markov trees in  \cite[Section 5.5]{bertoin2024self} to which we refer for details. Our Theorem~\ref{thm:magic} where $\beta^{x}$ is proved to belong to $\{3/2,5/2\}$ is directly inspired by those works.
 
Building on the fine properties of $W^{x}_{p}$ established in Corollary \ref{cor:3/25/2}, we are able to apply the results of \cite[Chapter VI]{bertoin2024self} and show that indeed, the above multi-type Galton--Watson tree $ \mathbf{t}$ under the laws $ \mathbb{P}_{p}^{x}$ admit $ \boldsymbol{ \mathcal{T}}_{1/2}$ or $ \boldsymbol{ \mathcal{T}}_{3/2}$ as scaling limits. This can be seen as a universal \textbf{invariance principle for random fully-parked trees} where all distributions have bounded support. We write $\mu_{ \mathbf{t}}$ for the counting measure on the vertices of $ \mathbf{t}$.

\begin{theorem}[Universal self-similar limits for the fully parked trees]  \label{thm:scaling} Suppose $(*)$ then we have
\begin{itemize}
\item When $x < x_{\crit}$ there exists some constant $ s^{x},v^{x}>0$ such that 
$$  \left(    s^{x} \cdot \frac{ \mathrm{t}}{ p^{1/2}}, \frac{\phi}{p}, v^{x} \cdot \frac{\mu_ \mathbf{t}}{p}\right) \mbox{ under } \mathbb{P}_p^{x} \xrightarrow[p\to\infty]{(d)}  \boldsymbol{ \mathcal{T}}_{1/2},$$
\item When $x=x_{\crit}$
$$  \left(   s^{x_{\crit}} \cdot \frac{ \mathrm{t}}{ p^{3/2}}, \frac{\phi}{p}, v^{x_{\crit}}\cdot \frac{\mu_ \mathbf{t}}{ p^{2}}\right) \mbox{ under } \mathbb{P}_p^{x} \xrightarrow[p\to\infty]{(d)}  \boldsymbol{ \mathcal{T}}_{3/2},$$
\end{itemize} 
the above convergence holds for the Gromov--Hausdorff--Prokhorov hypograph convergence developed in \cite[Chapter 1]{bertoin2024self} to which we refer for details.
\end{theorem}

The above convergence  result, although very satisfactory, may frighten the reader which has little acquaintance with random real trees. Luckily, the proof of  Theorem \ref{thm:5/2} only uses the convergence of the third coordinate in the case $x= x_{\crit}$, more precisely only the fact that the size of $ \mathrm{t}$ under $ \mathbb{P}_{p}^{x_{\crit}}$ is roughly of order $p^{2}$. This fact can also be proved using our random walk estimates without requiring the full technology of \cite{bertoin2024self}. 
From this, the tail exponent $\alpha$ appearing in \eqref{eq:expo} is related to the tail exponent of the survival probability of the $\nu$-random walk along a typical branch of $ \mathbf{t}$, which in our case is deduced from  classical results on random walks in the domain of attraction of the spectrally negative $3/2$-stable L\'evy process. \bigskip

 \noindent \textbf{Acknowledgments:} This work started during a conference on self-similar Markov trees organized at CIRM in July 2024. We thank the CIRM for hospitality, and all the participants, in particular Timothy Budd and Thomas Budzinski for stimulating discussions. N.C. is indebted to Armand  Riera and Jean Bertoin and for the many conversations they  had about \cite{bertoin2024self}. 
 This work was supported by SuPerGRandMa, the ERC Consolidator Grant No.\ 101087572.

\tableofcontents

\clearpage \section*{Table of notation}
To help the reader navigate through these pages, we gather here the most important notation used.

\newcommand*{\threesim}{%
  \mathrel{\vcenter{\offinterlineskip
  \hbox{$\sim$}\vskip-.25ex\hbox{$\approx$}}}}

\noindent \begin{tabular}{ll}
$ \mathbb{Z}_{>0}$ & $=\{1,2,3,\cdots\}$\\
$ \mathbb{Z}_{\geq0}$ & $=\{0,1,2,3,\cdots\}$\\
$ \mathbb{Z}_{<0}$ & $=\{\cdots,-3,-2,-1\}$\\
$ \mathbb{Z}_{\leq0}$ & $=\{\cdots,-3,-2,-1,0\}$\\
\end{tabular}\\ 
\noindent For positive sequences $(a_{n})_{n \geq 0}, (b_{n})_{n \geq 0}$ we write \\
\noindent \begin{tabular}{ll}
$a_{n} \sim b_{n}$ & if $ \frac{a_{n}}{b_{n}} \to 1$ as $n \to \infty$\\
$a_{n} \approx b_{n}$ & if $ \frac{a_{n}}{b_{n}} \to C$ for some $C > 0$ as $n \to \infty$\\
$a_{n} \lesssim b_{n}$ & if for all $n \geq 0$ we have $ \frac{a_{n}}{b_{n}} < C$ for some $C > 0$\\
$a_{n} \threesim b_{n}$ & if for all $n \geq 0$, $ c< \frac{a_{n}}{b_{n}} <C$ for some $c,C > 0$  $n \to \infty$\\
\end{tabular}\\ 
\noindent In particular, \textbf{we shall use $ \approx$ a lot below since we shall not care about positive constants}. \\ 
\noindent \begin{tabular}{ll}
$Q$ & positive polynomial in \eqref{eq:1}\\
$F(x,y)$ & generating series solution to  \eqref{eq:1}\\
$w$ & local weight function\\
$ \mathrm{w}$ & push forward of the local weight function on fully parked trees\\
$ \mathrm{t}$ & plane rooted tree \\
$ \mathbbm{t}$ & $ = ( \mathrm{t}, \ell^{ \mathrm{car}}, \ell^{ \mathrm{spot}} )$ tree with car arrivals on vertices and parking spots on edges \\
$\mathbf{t}$ & $= ( \mathrm{t}, \phi) $ our underlying $ \mathbb{Z}_{ \geq 0}$-labeled  tree  \\
$\K$ & the absolute constant bounding the degree, car arrivals, and parking spot capacities\\

$ \mathrm{FPT}_p^n$ & set of the fully parked tree with $n$ vertices and outgoing flux $p$ \\
$W_p^x$&  $= [y^{p}] F(x,y)$ partition function of labeled trees starting with label $p$\\
$x_{\crit}$ & radius of convergence of $x \mapsto W_{p}^{x}$\\
$ y_{\crit}^{x}$ & radius of convergence of  $y \mapsto F(x,y)$\\
$\widetilde{W}_p^x $ & $ (y^{x}_{\crit})^{p} W_{p}^{x}$, i.e. polynomial part in the partition function\\
$\beta^x, \beta$ & polynomial tail in the asymptotic of $W_{p}^{x}$\\
$ \mathbb{P}_{p}^{x}$& $x$-Boltzmann law\\
$ \hat{\nu}, \nu$ & see Definition \ref{def:nuhat}\\
$(S, \eta)$ & decoration and reproduction process\\
$ \mathcal{L}_{t}$ & $(S, \eta)$ is a locally largest exploration up to time $t$ with no ties, see \eqref{def:Lt}\\
$\tau$ & absorption time for the decoration process\\
$\tau_{x}$ & $= \inf \{i \geq 0 : S_{i} \leq x\}$\\
$\tau_{\{x\}}$ & $= \inf \{i \geq 0 : S_{i} =x\}$\\
$\tau_{\geq x}$ & $= \inf \{i \geq 0 : S_{i} \geq x\}$\\
$ \mathbb{P}_{p}^{\RW}$& law under which $(S, \eta)$ is made of iid increments.\\
$ \mathbf{t}^ \circ$ & vertices of $ \mathrm{t}$ with label $\K$ \\
$  \mathrm{Vol}^ \circ$ & $= \# \mathbf{t}^{\circ}$\\
$  \mathrm{Vol}^ \bullet$ & $= \# \mathrm{t}$\\
$W_p^{\circ,x}$ & $ = \mathbb{E}_{p}^{x}[ \mathrm{Vol}^{\circ}]$  \\
$W_p^{\bullet,x}$ & $ = \mathbb{E}_{p}^{x}[ \mathrm{Vol}^{\bullet}]$  \\
$ \mathbb{P}_{p}^{\circ,x}$ & pointed $x$-Boltzmann law, see \eqref{eq:deflawp}\\
\end{tabular}

\section{Catalytic equation as parking models} \label{sec:2.1}
In this section we make explicit the connection between catalytic equations of the form \eqref{eq:1} with a polynomial $Q$ having positive coefficients, and the enumeration of weighted fully-parked trees. Interpreting \eqref{eq:1} as an \emph{infinite} positive system of polynomial equations we deduce the first properties on the partition functions (Lemma \ref{lem:asympW}). This will enable us to introduce the Boltzmann measure on fully-parked trees which will be the key random variables studied in the rest of this work.
\subsection{Catalytic equation, parking models}
In this work, we follow Neveu's formalism \cite{Nev86} and see planar trees $ \mathrm{t}$ as subsets of the Ulam tree $ \mathcal{U} =  \bigcup_{n \geq 0}  (\mathbb{Z}_{> 0})^{n}$. In particular, the tree $ \mathrm{t}$ is itself the set of its vertices, $ \varnothing$ is the root vertex of $ \mathrm{t}$ and for $u \in \mathrm{t}$, we write $k_{u}( \mathrm{t})$ for the number of children of $u$ in $ \mathrm{t}$.  We write $ u \preceq v$ for the genealogical (partial) order on $  \mathcal{U}$.\\
Recall our notation $ \mathbbm{t} = ( \mathrm{t}, \ell^{ \mathrm{car}}, \ell^{ \mathrm{spot}})$ for a plane tree with car arrivals $\ell^{ \mathrm{car}}$ on the vertices and parking spot assignations $\ell^{ \mathrm{spot}}$ on the edges as well as $ \mathbf{t}= ( \mathrm{t}, \phi)$ for the underlying plane tree labeled by the outgoing flux of cars at each vertex. We always suppose that after parking, the tree is fully-parked (no available spots anymore) and recall that $ \mathrm{FPT}_{p}^{n}$ is the set of all parking trees with $n$ vertices with outgoing flux $\phi(\varnothing) =p$. Recall also the definition of the measure $ \mathrm{w}$ from the local weight function $w_{c,k,(s_{1}, ... , s_{k})}$ in \eqref{def:measurew}, recall  from  \eqref{def:wpx} the definition of $W_{p}^{x}$ and form the (formal) generating series $$F(x,y) = \sum_{p \geq 0} y^{p} W_{p} ^{x}.$$ Recall also from \eqref{eq:discretederivative} the definition of $\Delta^{{(i)}} F$. The link between catalytic equations of the form \eqref{eq:1} and parking model on trees is as follows:
\begin{proposition}[See also \cite{duchi2024order}]\label{prop:tutte} The generating series $F$ is the unique solution  (as formal power series) to Equation \eqref{eq:1} where $Q$ is the polynomial with  positive coefficients specified by the local weight function as: for all $ c, a_0, \dots, a_\K \geq 0$, writing $ k = \sum_{i = 0}^{\K} a_i$,
 \begin{eqnarray} \label{eq:lienwQ}  [y^c f_{}^{a_0} f_1^{a_1}... f_\K^{a_\K}]Q(y,f,f_1, ... , f_\K) =  \sum_{\substack{(s_1, \dots s_k)\\ \forall i, \#\{ j: s_j = i\} = a_i   }}w_{c,k, (s_1,\dots, s_k) }.  \end{eqnarray}
\end{proposition}
Before the proof, remark that given a local weight function, we can always find a polynomial $Q$ satisfying \eqref{eq:lienwQ} but there are multiple choices of local weight function corresponding to a given $Q$. This issue is fixed below by ensuring exchangeability of $w$ in our standing assumptions $(*)$.
\begin{proof} 
The proposition is a reformulation of the well-known \textbf{Tutte equations} we obtain by decomposing a fully parked tree with respect to the splitting at its root vertex, and considering $y$ as a catalytic variable. Specifically, given a fully parked tree with a root having weight $w_{c,k,(s_1, \ldots, s_k)}$, then for $1 \leq i \leq k$, the subtree above the $i$-th children is a fully parked tree with a flux of at least $s_i$, call it $p_i$. Moreover, the total flux emanating from the root is simply the sum of $c$  cars arriving at the root plus the flux coming from all its children minus the cars that park on the edges between the root and its children (all spots have to be occupied since we consider fully parked trees), see Figure \ref{fig:parking1}. The constraint, which we shall call the ``parking constraints" are then 
  \begin{eqnarray} \label{eq:parkingconstraint} \sum_{i=1}^k (p_i-s_i) +c =p, \quad \mbox{ and }\quad   p_i \geq s_i, \ \ \forall\ 1 \leq i \leq k.  
 \end{eqnarray}
In terms of the generating function, this gives 
\begin{eqnarray}\label{eq:tutte} F(x,y) &=& x \sum_{c,k} \sum_{(s_1, \dots ,s_k)} w_{c,k,(s_1, \ldots, s_k)} y^c \prod_{i = 1}^{k} \sum_{p_i \geq s_i} y^{p_i - s_i} W_{p_i}^x  \\
& \underset{\eqref{eq:discretederivative}}{=}& x \sum_{c,k} \sum_{(s_1, \dots ,s_k)} w_{c,k,(s_1, \ldots, s_k)} y^c \prod_{i = 1}^{k} \Delta^{(s_i)}F(x,y). \nonumber   \end{eqnarray}
Introducing $a_0, ... , a_\K$ such that $\forall i, \{ j : s_j = i\} = a_i $ and recalling \eqref{eq:lienwQ} we can write the above as:
  \begin{eqnarray*}
F(x,y)&=& x \sum_{c,k}  \sum_{\substack{(s_0, s_1, \dots s_{k})\\ \forall i, \{ j, s_j = i\} = a_i   }} w_{c,k,(s_1, \ldots, s_k)} y^c  \prod_{i = 0}^{\K} \left(\Delta^{(i)}F \right)^{a_i} \nonumber \\
&\underset{\eqref{eq:lienwQ}}{=}& x \sum_{a_0, \dots, a_\K} \sum_{c} [y^c f^{a_0} f_1^{a_1}... f_\K^{a_\K}]Q(y,f,f_1, ... , f_\K) y^c  \prod_{i = 0}^{\K} \left(\Delta^{(i)}F \right)^{a_i} \nonumber \\
&=&x Q(y, F,\Delta^{(1)}F, \dots, \Delta^{(\K)}F).  \nonumber 
\end{eqnarray*}
In particular, it shows that the coefficients of $[x^n] F(x,y)$ can be deduced from that of $ ([x^{k}] F(x,y) : 0 \leq k \leq n-1)$ for $n \geq 1$, and the function $F$ is unique as a formal power series. This proves the proposition. \end{proof}

\subsection{Positive system and standing assumptions}
In all this work we shall suppose that $Q$ and the local weight function are tighten as in the previous proposition.   Equation \eqref{eq:1} can also be interpreted as an infinite system of equations for the partitions functions 
$$ W_p^x = x Q_p(x, W_0^x,W_1^x, ... ), \quad \mbox{ for } p \geq 0,$$
where $Q_p$ is a series with positive coefficients. Such systems of equations are very-well understood when there are finitely many of them, see \cite{banderier2015formulae,lalley2004algebraic,drmotaHainzl} or the beautiful thesis of Hainzl \cite{hainzl2023singularity} for more details. The salvation in our case will ultimately comes from a ``translation invariance'' property inherited from the parking model. In this section we prove that all partitions functions $W_p^x$ have the same behavior at their singularity, yielding our first rough estimates.
\paragraph{Dependency graph.}  We introduce the \textbf{dependency graph}  $ \mathcal{G}_{Q}$ of the system of functional equations on the $(W_i^x)$ induced by Equation \eqref{eq:1}, which is the oriented graph on the integers $\{0,1,2, ... \}$ constructed by putting an oriented edge going from $i$ to $j$ if in $W_i^x = [y^i]\left(x Q ( y, F, \Delta^{(1)} F, \ldots,  \Delta^{( \K)} F ) \right)$ one can find a non-trivial term involving $W_j^x$.  See Figure \ref{fig:grapheGQ} for an illustration. In other words, the graph $ \mathcal{G}_Q$ contains an oriented edge from $i$ to $j$ if
$$ \frac{ \partial }{ \partial W_j^x} [y^i]\left(x Q ( y, F, \Delta^{(1)} F, \ldots,  \Delta^{( \K)} F ) \right) \neq 0.$$
In particular this implies that $W_i^x \geq x^{d_{i \to j}} C_{i \to j} W_j^{x}$ for some positive constant $C_{i \to j}$ and some degree $d_{i \to j} \geq 1$. 
Notice then an important property of our parking model (and thus of the system of functional equation), which can be thought of \textbf{translation invariance} and will eventually yield to the connection with random walks: \medskip 

\begin{center}
\fbox{
 (trans.) $\qquad $ 
\begin{minipage}{11cm}
if there is a factor involving $W^x_{j}$ in $[y^i] x Q (y, F, \Delta^{(1)} F, \ldots,  \Delta^{( \K)} F )$ for $i,j \geq 0$, then we can find the  exact same factor by replacing $W^x_{j}$ with $W_{j+1}^x$  in $[y^{i+1}] x Q ( y, F, \Delta^{(1)} F, \ldots,  \Delta^{( \K)} F )$. 
\end{minipage} 
}\end{center}
\medskip 

This is plain since in a fully parked tree, adding one to the flux of a descendant of a vertex automatically add $1$ to the flux of this vertex. In particular, if the edge $i \to j$ in present in $ \mathcal{G}_Q$, then all edges $i +q \to j+q$ are also present for $q \geq 0$. As an example, if there is an oriented edge from $0$ to some $p \geq 1$, then for all $q \geq 0$ and $k \geq 0$, then there is an oriented path from $q$ to $ q+ kp$ where this second quantity can be made as large as we wish by taking $k$ large enough.\bigskip

With this graphical construction at hands, one can now properly state our standing assumptions the local weight function denoted by (*) in the introduction: \\

\noindent \fbox{\begin{minipage}{15cm} 
\textbf{Standing assumptions (*)}:
\begin{enumerate}
\item\label{ass:bound} \emph{(Boundedness)} We assume that the offspring, the car arrival and parking spot numbers are bounded. More precisely, we assume that there exists a constant denoted by $$ \K \geq 0$$ in the rest of the paper, such that if $k >  \K$ or $ c>   \K$ or if  there exists $ 1 \leq i \leq k$ such that $ s_i>  \K$, then $ w_{c,k,(s_1, \ldots, s_k)} = 0$. 
\item\label{ass:ex} \emph{(Exchangeability)} We assume that the local weight function is exchangeable in $(s_1, \dots , s_k)$. This is granted up to replacing the $ w_{c,k,(s_1, \ldots, s_k)}$ by 
\begin{eqnarray*}  \widetilde{w}_{c,k,(s_1, \ldots, s_k)} :=  \frac{1}{k!} \sum_{ \sigma \in \mathfrak{S}_k}w_{c,k,(s_ {\sigma(1)}, \ldots, s_ {\sigma(k)})},
\end{eqnarray*}
where $\mathfrak{S}_k$ denotes the set of permutations of  $ \{ 1, \dots, k\} $.
\item\label{ass:branch} \emph{(Branching)} We assume that there exist $k \geq 2$, $c \geq 0$ and $s_1, \ldots, s_k \geq 0$ such that $$w_{k,c,(s_1, \ldots, s_k)}>0.$$
In particular the polynomial $Q$ is non-linear.
\item\label{ass:ap} \emph{(Aperiodicity for the flux)} We assume that the flux is aperiodic, 
which means that 
 there exists $p \geq 0$ such that 
the vertex set $ \{r : r \geq p\}$ is strongly connected in the dependency graph $ \mathcal{G}_Q$  (i.e. $q,r \geq p$ are connected by an oriented path going from $q$ to $r$ and from $r$ to $q$ while staying above $p$). We write $p_0$ the minimal value of $p$ such that the system is strongly connected from rank $p$. See Figure \ref{fig:grapheGQ}.
\item\label{ass:connect}{(Connectivity)} We assume that $0$ is connected to a vertex $p \geq1$ in $ \mathcal{G}_{Q}$ and that $W_{p}^{x}>0$ for all $p \geq 0$ and all $x>0$.
\item\label{ass:apvertex} \emph{(Aperiodicity for the vertices)} 
We assume that $ x \mapsto W_0^x$ has a unique dominant singularity on the boundary of its disk of convergence.
\end{enumerate} 
\end{minipage}}

\medskip

Let us quickly comment on those assumptions. The boundedness assumption will ensure that all equations considered below are polynomial equations. Of course we expect that this condition can be slightly relaxed, but in the general case new universality classes should appear as in \cite{chen2021enumeration} see \cite[Chap. 8.3]{bertoin2024self} and Section \ref{sec:comments}. The branching assumption assures that the underlying tree is not a line segment and so \eqref{eq:1} is not linear. Finally, the aperiodicity assumption, appart from avoiding sub lattice issues, ensures that the flux $\phi$ can both increase and decrease along the branches of our labeled trees. A stronger version of those assumptions is already present in \cite{drmota2022universal} and \cite{drmota2023universal}. The aperiodicity assumption for the vertices is here to avoid sub-lattice issues for the degree of $x$.

\begin{figure}[!h]
 \begin{center}
 \includegraphics[width=13cm]{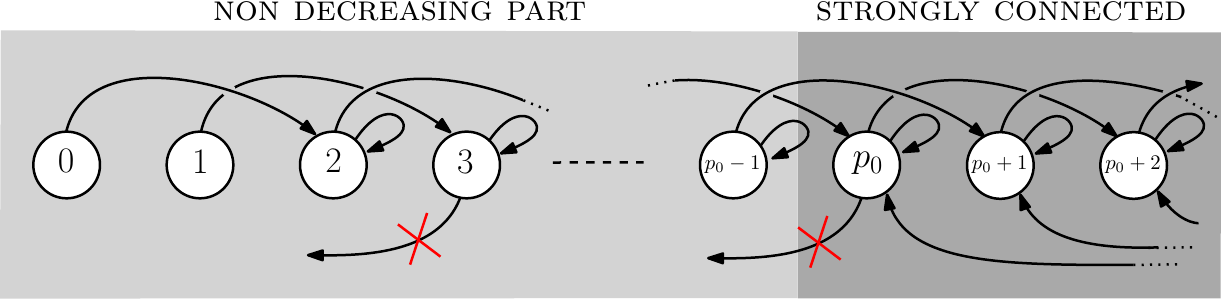}
 \caption{Illustration of the aperiodicity conditions: the graph $ \mathcal{G}_Q$ can be decomposed in a non decreasing and a strongly connected part from $p_0$ on. If an edge $i \to j$ is present in the graph, then so is $i +q \to j+q$ with $q \geq 0$. \label{fig:grapheGQ}}
 \end{center}
 \end{figure}
 
 \paragraph{Link between $p_0$ and $\K$.} Recall from above that $p_0$ is the minimal value from which the graph $ \mathcal{G}_{Q}$ is strongly connected and that $0$ is connected to some $ p \geq 1$, see Figure \ref{fig:grapheGQ}. We first claim that $p_0$ is not directly linked to any $p < p_0$, for otherwise, by translation invariance, the graph $ \mathcal{G}_{Q}$ would by strongly connected on $\{p, p+1, ..., p_{0}, ... \}$. The graph $ \mathcal{G}_Q$ can thus be decomposed into the part $\{0,1, ... , p_0-1\}$ from which there are only edges pointed towards vertices with larger labels (the non decreasing part) and the part $\{p_0, p_0+1, ... \}$ which is strongly connected.

 Our last remark on the assumptions above is that $p_0 \leq \K$. Indeed, let $p > p_0 $ such that there exists an edge from $p$ to $p_0$ in $ \mathcal{G}_Q$ (such $p$ exists  since by hypothesis, the subset $ \{ p_0, p_0+1, \dots\}$ is strongly connected in $ \mathcal{G}_Q$). By definition, there exists $k \geq 1$, $ c, \geq 0$, $s_1 \leq p_0$ and $  p_i \geq s_i \geq 0$ for $2 \leq i \leq k$ such that  the term 
 $$ x w_{c,k,(s_1, \dots, s_k)} W_{p_0}^x \prod_{i=2}^{k} W_{p_i}^x$$ appears in $[y^p]\left(x Q ( x,y, F, \Delta^{(1)} F, \ldots,  \Delta^{( \K)} F ) \right)$. Note that we must have $ \sum_{i=2}^{k} (p_i -s_i) + p_0 - s_1 +c = p$. But in this case, the term
  $$ x w_{c,k,(s_1, \dots, s_k)} W_{s_1}^x \prod_{i=2}^{k} W_{p_i}^x$$ appears in $[y^{p-(p_0 -s_1)}]\left(x Q ( x,y, F, \Delta^{(1)} F, \ldots,  \Delta^{(\K)} F ) \right)$ and $p-(p_0 - s_1) > s_1$ since $ p> p_0$, which implies that there exists a decreasing edge from $p-(p_0 - s_1)$ to $s_1$ in $ \mathcal{G}_Q$. In particular, we have $p_0 \leq s_1 \leq \K$.  \bigskip

Let us give some natural examples of weight functions that satisfies our standing assumption, see also the many examples given in \cite{hainzl2023singularity}.
 \begin{example}[Parking on BGW trees] The most natural example is to come back to the classical model of parking on Bienaym\'e--Galton--Watson trees. Let $ \mu$ and $ \xi$ be two probability distributions on the non negative integers $ \{ 0, 1, \dots \}$ which are respectively thought of as the offspring and car arrival distribution. For all $k,c \geq 0$, we set  $$w_{c,k, \underbrace{ \text{\scriptsize{(1,\dots , 1)}}}_{k \text{ times}}} = \mu (k) \cdot \xi (c)$$ and $w_{c,k,(s_1, \dots, s_k)} = 0$ otherwise. Then this choice of weights corresponds to choosing a tree with weights proportional to its Bienaym\'e--Galton--Watson measure and with i.i.d.\ car arrivals on its vertices and with exactly one spot per edge (and restricting to those fully parked configurations). It corresponds to the classical model of parking cars on Bienaym\'e--Galton--Watson  (see e.g.\ \cite{ConCurParking,Contat21+,contat2020sharpness,GP19,curien2022phase,aldous2023parking}), except that in the literature parking spots are usually on the vertices and not on the edges. The only difference is that the parking spot of the root vertex disappear when the parking spots are slide from each vertex to the edge below (in the direction of the root). 
In the case of a geometric offspring distribution $ \nu$,  our main Equation \eqref{eq:1} has already been considered by Chen \cite{chen2021enumeration}. He already showed  that in the ``generic case", then 
$$[x^n] W_0^x \sim \mathrm{cst}\cdot n^{-5/2} \rho^{-n}$$
for some constant $\mathrm{cst}>0$ and radius $ \rho>0$.
 \end{example}

\begin{example}[Planar maps] A rooted planar map is a connected planar graph (loops and multiple edges are allowed) embedded into the sphere, and with one distinguished oriented edge called the root edge. The face that is left to the root edge is often called the exterior face and its length (number of edges adjacent to the face) is called the perimeter of the map. We denote by $M(x,y)$ the generating function of rooted planar maps where $x$ counts the number of edges and $y$ counts the perimeter. Tutte's equation writes as 
$$ M(x,y) = 1 + x y^2 M(x,y)^2 + x y \frac{ y M(x,y) - M(x,1)}{ y-1}.$$
By defining $ N(x,y) = M(x,y+1)$, we get
\begin{eqnarray*}
 N(x,y) &=& 1 + x (y+1)^2 N(x,y)^2 + x (y+1) \frac{ (y+1) N(x,y) - N(x,0)}{ y} \\
 &=& 1 + x (y+1)^2 N(x,y)^2 + x (y+1)  \left( N(x,y) + \Delta^{(1)}N(x,y)\right),
 \end{eqnarray*}
We obtain an Equation of the form \eqref{eq:1} for $N(x,y)-1$. It is easy to see that the system is strongly connected and satisfies all our standing assumptions. 
Applying the so-called \emph{kernel method} \cite{BMJ06}, we can  solve explicitly the system and show that the number of rooted planar maps is asymptotically
$$ [x^n] N(x,0) = [x^n] M(x,1) \sim \frac{2}{ \sqrt{\pi}} n^{-5/2} 12^n \quad \mbox{ as } n \to \infty.$$
 \end{example}

\subsection{First properties of the generating function $F$ using analytic combinatorics}
We now gather a few basic properties of the partition functions $W_{p}^{x}$. Roughly speaking, we use standard analytic combinatorics arguments to prove that those functions behaves as $C \cdot \rho^n\cdot n^{-\beta}$ for $\beta \in \mathbb{Q}_{>0} \backslash \{1,2,3,...\}$.  The next sections will actually identify $\beta$ as being $ \frac{3}{2}$ or $ \frac{5}{2}$ using a connection with an underlying random walk.\medskip

Recall the definition of $p_0$ in the aperiodicity Assumption \eqref{ass:ap} part of our standing assumptions $(*)$.
We start by proving the modest result below using our standing assumptions on the local weight function:
\begin{lemma}[Introducing $x_{\crit}, y_{\crit}$] \label{lem:xcyc}We have the following:
\begin{itemize} \item The partition functions $(W_p^x : p \geq 0)$, as  functions of $x$, are positive functions, and have a common finite radius of convergence $x_{\crit} \in (0, \infty)$. Furthermore, $W_{p}^{x_{\crit}} < \infty$ for all $p \geq 0$.
\item For all $x \in (0,x_{\crit}]$, the function $y \mapsto F(x,y) = \sum_{p\geq 0} y^{p} W_{p}^{x}$ has a finite and positive radius of convergence $y_{\crit}^x \in (0, \infty)$ and $F(x,y_{\crit}^x) < \infty$ (even for $x = x_{\crit}$). 
\end{itemize}
\end{lemma} 
\begin{proof}
We start by considering the function $x \mapsto W_p^x $ for $ p \geq p_0$ where we recall that the system is strongly connected for $ p \geq p_0$. It implies that we can find a direct path $i_0:= p_0+1, i_1, \dots, i_{k_1}:=p_0$ from $p_0+1$ to $p_0$,  such that for all $1 \leq j \leq k_1$, there exists $c_j>0$ and $\ell_j\geq 0$ such that $ W_{i_{j-1}} \geq c_j x^{\ell_j} W_{i_j}$. Writing $C_1 := \prod_{j = 1}^{k_1} c_j$  as well as $ L_1 := \sum_{j =1}^{k_1} \ell_j$  we deduce that $C_1 x^{L_1} W^x_{p_0+1} \leq  W^x_{p_0}$ and performing the same reasoning with a path linking $p_0$ to $p_{0}+1$ yields the further existence of $C_2,L_2 \geq 0$ such that 
\begin{equation*}C_1 x^{L_1} W^x_{p_0+1} \leq  W^x_{p_0} \leq  C_2 x^{L_2} W^x_{p_0+1}.\end{equation*}
In fact, our system is not only strongly connected, but it is also ``translation invariant".  Indeed, if there is a factor involving $W^x_{p}$ in $[y^r] x Q (y, F, \Delta^{(1)} F, \ldots,  \Delta^{(\K)} F )$ for $p,r \geq p_0$, then the same factor involving $W_{p+1}^x$ appears  in $[y^{r+1}] x Q ( y, F, \Delta^{(1)} F, \ldots,  \Delta^{(\K)} F )$.
In particular, for the same $C_1,C_2, L_1,L_2$ as above, we have for all $ p \geq p_0$ and all $x \geq 0$, then 
\begin{equation}\label{eq:aperiodic} C_1 x^{L_1} W^x_{p+1} \leq W^x_{p} \leq C_2 x^{L_2}W^x_{p+1}.\end{equation}
This shows that the functions $( x \mapsto W^x_{p} : p \geq p_0)$ have the same (possibly infinite) radius of convergence $x_{\crit} \in [0, \infty]$  and the same behaviour at $x = x_{\crit}$. Moreover, for $0 \leq p < p_0$, since they are no edge linking $p$ to a smaller value in $ \mathcal{G}_Q$,  the function $W_p^x = [y^p] F(x,y) = [y^p]x Q ( y, F, \Delta^{(1)} F, \ldots,  \Delta^{(\K)} F )$ can be written as a finite combination of the $( W^x_{q} : q \geq p)$ with positive coefficients. Furthermore, the only term involving $ W_p^x$  which might appear in \linebreak $[y^p]x Q ( y, F, \Delta^{(1)} F, \ldots,  \Delta^{(\K)} F )$ are linear in $ W_p^x$ (with no other $(W_q : q \geq p)$). Indeed, if there is a term divisible by $ W_p^x W_q^x$ in $[y^p]x Q ( y, F, \Delta^{(1)} F, \ldots,  \Delta^{(\K)} F )$, then the same term with replacing $W_q^x$ by $W_{q+1}^x$ appears in $[y^{p+1}]x Q ( y, F, \Delta^{(1)} F, \ldots,  \Delta^{(\K)} F )$ which implies that there is an edge in the dependency graph $ \mathcal{G}_Q$ from $p+1$ to $p$, contradicting the fact that $ p < p_0$. This means that there exists $w \in [0,  \infty) $ such that $ (1 - w x) W_p$ is a finite combination of the $( W^x_{q} : q > p)$ with positive coefficients. Thus, the radius of convergence of $ W_p^x$ is the minimum between the radius of convergence of the  $( W^x_{q} : q > p)$'s and $1/w$. But recall that our system is translation invariant, meaning that if $w x W_p^x$ appears in $[y^p]x Q ( y, F, \Delta^{(1)} F, \ldots,  \Delta^{(\K)} F )$, then for all $ q > p$, there exists $w' \geq w$ such that the term  $w' x W_{q}^x$ appears in $[y^q]x Q ( y, F, \Delta^{(1)} F, \ldots,  \Delta^{(\K)} F )$ and the radius of convergence of $ W_q$ is smaller than $ 1/w' \leq 1/w$. As the consequence, we obtain that the radius of convergence of $W_p$ is the minimum of the radius of convergence of the  $( W^x_{q} : q > p)$'s and a straightforward induction shows that this minimum is $ x_{\crit}$,  and the whole family as the same radius of convergence.

Let us now prove that $ x_c > 0$. The first observation is that since the car arrivals, the number of children of a vertex and the number of spots on an edge are bounded by $\K$, then there are at most $ (\K+1)^{\K+2}$ weights of the form $w_{c,k,(s_1, \dots, s_k)}$ which are positive. In particular, the number of trees with $n$ vertices which have a positive weight is at most $ \mathrm{Cat}(n)  (\K+1)^{n (\K+2)} \leq 4^{n}(\K+1)^{n (\K+2)} $. Let us write 
$$M := \max_{\substack{c \geq0\\ k \geq 0}} \max_{s_1, \dots, s_k \geq 0} w_{c,k,(s_1, \dots, s_k)} < \infty.$$
Then for all $x \geq 0$ and all $ p \geq 0$, we have 
\begin{eqnarray*}  
W_p^x &=&\sum_{n \geq 0} x^n \sum_{ \mathbbm{t} \in \mathrm{FPT}_n^p} \mathrm{w}( \mathbbm{t}) \\
&\leq&\sum_{n \geq 0} x^n \sum_{ \mathbbm{t} \in \mathrm{FPT}_n^p} M^n \cdot \mathbb{1}_{ \mathrm{w}( \mathbbm{t})> 0} \\ 
&\leq&\sum_{n \geq 0} x^n(4(\K+1)^{\K+2}M)^n.
\end{eqnarray*}
In particular, the right-hand side is finite when $x < 1/(4 M (\K+1)^{\K+2})$, which proves that $x_c \geq 1/(4 M (\K+1)^{\K+2}) >0$.

To prove that $x_\crit$ is finite, we use the branching assumption: We know that there exists $k \geq 2$ and $ (c,k,(s_1, \ldots, s_k))$ such that $w_{c,k,(s_1, \ldots, s_k)} >0.$  Let $p \geq \max(c-s_1, p_0)$.  Then by \eqref{eq:tutte} we have \begin{eqnarray*} W_p^x \geq x  w_{(c,k,(s_1, \ldots, s_k))}\cdot W_{p+ s_1 - c}^x \cdot \prod_{i = 2}^{k} W_{s_i}^x 
\end{eqnarray*}
Moreover, since $p+ s_1 - c$ and $ s_2$ are connected via an oriented path to $p$ in $ \mathcal{G}_Q$, we can find an integer $n$ and a constant $C$ such that for all $x >0$ we have
\begin{eqnarray*}W_p^x  \geq C x^n (W_p^x)^2.
\end{eqnarray*}
In particular, for positive $x$ and when $W_p^x$ is finite we have 
\begin{eqnarray*}W_p^x  \leq \frac{1}{C} x^{-n}.
\end{eqnarray*}
Moreover, the function $ x \mapsto W_p^x $ is a non-zero series with positive coefficients. Thus there exists $C'> 0$ and $m \geq 1$ such that $$ W_p^x \geq C' x^m.$$
This implies that $ F(x,0)$ can be finite only for the $x$ such that $ C' x^m \leq C x^{-n/(k-1)}.$ 
In particular, the series $ x \mapsto W_p^x$ must have a finite radius of convergence $ x_{\crit} < \infty$  and also that $ \lim_{x \to x_{\crit}^-} W_p^x < + \infty$ so that $ W_p^{x_{\crit}} < + \infty.$ 
This proves the first item of the Lemma. 

Let us now move to the second item and use similar ideas.  Let us fix $x \in (0, x_{\crit}]$. 
Again, on the one hand, we have $ W_{p}^{x} >0$ for all $p> 0$, thus $F(x,y) \geq y^{p } W_{p}^{x}$. On the other hand, using again our branching assumption with the same notation as above, we have for all $p$ large enough, 
$$ W_{p+c}^{x} \geq x w_{c,k,(s_1, \ldots, s_k)}\cdot \sum_{q = 0 }^{p} \left( W_{q+s_1}^{x} W^x_{p-q+s_2}\right)\cdot \prod_{i = 3}^{k} W_{s_i}^x, $$
Summing over all possibilities for $p \geq 0$ and recalling $s_i \leq \K$ we deduce that 
$$ y^ \K \Delta^{(\K)} F(x,y) + \sum_{i=0}^{\K-1} y^i W_i^x = F(x,y)   \geq \mathrm{Cst}^x y^{d} \left(\Delta^{(\K)} F(x,y)\right)^2.$$
Note that $\Delta^{(\K)} F(x,y) \geq W_ \K^x$. Thus, if $\Delta^{( \K)} F(x,y) < \infty$, dividing the previous inequality by $ y^ \K \Delta^{(\K)} F(x,y)$, we obtain
$$ 1 + \frac{1}{W_ \K^x} \sum_{i=0}^{\K-1} y^{i-\K} W_i^x   \geq \mathrm{Cst}^x y^{d- \K} \Delta^{(\K)} F(x,y).$$

Here the left-hand side tends to $1$ as $y$ tends to infinity, whereas the right-hand side is larger than any polynomial in $y$. We conclude as for the series in $x$ that $ y \mapsto F(x,y)$ must have a finite radius of convergence $ y_{\crit}^x$ and must be finite at $y_{\crit}^x$.
 \end{proof}

 Our next step (which is also folklore) is to use algebraicity and standard analytic combinatorics techniques to get a finer asymptotic on $W_{p}^{x}$:
\begin{lemma}\label{lem:asympW} For $0 < x \leq x_{\crit}$, there exist $b^{x} \in \{1,2, ...\} $, constants $C^{x}_j > 0$ for $ j \in \{0, 1, \dots, b^x-1\}$ together with $y_{\crit}^{x}>0$ and an exponent $ \beta^{x}  \in (1, \infty) \backslash \{2,3,... \}$ such that 
$$ W_{p}^x \underset{\substack{p \to + \infty\\ p \equiv j\, \mathrm{mod}\, b^x}}{\sim} C^{x}_j \cdot  p^{- \beta^{x}} \cdot  (y_{\crit}^{x}) ^{-p}.$$
\end{lemma}
We will show in the next section that actually $b^{x}=1$ so that $C_{j}^{x}$ are constant in $j$. To simplify notation, in the sequel we shall denote by $\widetilde{W}_{p}^{x}$ the function obtained after correcting  the exponential  part, i.e.
  \begin{eqnarray} \label{eq:polyalways} \forall p \geq 0, \forall 0 < x \leq x_{\crit}, \quad \widetilde{W}_{p}^{x} := (y_{\crit}^{x})^{p}\cdot W_{p}^{x},  \end{eqnarray} so that the latter always decay in a polynomial fashion like $ p^{-\beta^{x}}$. 

\begin{proof} Recall that by Proposition \ref{prop:tutte}, the bivariate series $ F(x,y)$ satisfies \eqref{eq:1}. Thus, if we \textbf{fix} $ 0< x \leq x_{\crit}$, multiplying both sides of \eqref{eq:1} by $y^\K$, we deduce that the series $ y \mapsto F(x,y)$ satisfies a polynomial equation $P(y, F(x,y)) =0$ and is for this reason algebraic. Moreover, by the second item of Lemma \ref{lem:xcyc}, this series has a finite radius of convergence $ y_{\crit}^x \in (0, \infty)$.  As a consequence, we can use \cite[Theorem VII.8 p 501]{FS09}, which implies that there exists $b^x \in \mathbb{Z}_{ \geq 1}$,  constants $C^{x}_j > 0$ and exponents $ \beta^{x}_j \in \mathbb{R}$ for $ j \in \{0 \leq j \leq b^x-1\}$ such that 
$$ W_{p}^x \underset{\substack{p \to + \infty\\ p \equiv j\, \mathrm{mod}\, b^x}}{\sim} C^{x}_j \cdot  p^{- \beta^{x}_j} \cdot  (y_{\crit}^{x}) ^{-p}.$$
Moreover, since our system is aperiodic (Assumption \ref{ass:ap}), by \eqref{eq:aperiodic} it implies that all exponents $ \beta^x_j$ are equal and let us denote by $ \beta^x$ their common value. By the second item of Lemma \ref{lem:xcyc}, we must have $ \beta^x > 1$ since $ F(x,y_{\crit}^x) < + \infty$ and furthermore $\beta^{x} \notin \{2,3, ...\}$ by \cite[Theorem D, p.293]{flajolet1987analytic} 
\end{proof}

\subsection{A multi-type Galton--Watson tree}

We end this section by introducing  the random tree model underlying  \eqref{eq:1}. Since we know that $ W_{p}^{x}$ are (non zero) finite numbers for $x \leq x_{\crit}$, we can define the associated Boltzmann distribution on fully-parked trees (plane trees) of $\cup_{n \geq 1} \mathrm{FPT}_{p}^{n}$ by normalizing the measure $ \mathrm{w}( \mathrm{d} \mathbbm{t}) x^{\#  \mathrm{t}}$ as in \eqref{def:boltz}.  We shall write $ \mathbb{P}^x_{p}$  the obtained probability measure on the canonical space of $ \mathbb{Z}_{\geq 0}$-labeled trees $  \mathbf{t} = ( \mathrm{t}, (\phi(u))_{u \in \mathrm{t}})$ and denote by $ \mathbb{E}_{p}^x$ its associated expectation. We insist on the fact that $ \mathbf{t}$ are in particular \textit{plane} trees, see Figure \ref{fig:parking1}. 
\begin{proposition}[Fully parked trees as integer-type Bienaym\'e--Galton--Watson trees] \label{prop:GW}For any $x \in (0, x_\crit]$, under $ \mathbb{P}^x_{p}$ the random tree $(  \mathrm{t}, \phi)$ is a multitype Bienaymé--Galton--Watson tree with offspring distribution prescribed by 
$$ \pi_{p}^x(p_{1}, ... ,p_{k}) := x \sum_{ \substack{ c \geq 0 \\ \forall i , 0 \leq s_i \leq p_i \\ \sum_{i = 1}^{k} (p_i - s_i) +c =p}} w_{c, k, (s_1, \ldots, s_k)} \cdot  \frac{ W_{p_1}^x \ldots W_{p_k}^x}{ W_p^x},$$ where $\pi_{p}^x(p_{1}, ... ,p_{k})$ is the probability that a particle of type $p$ gives rise to $k$ particles of label $p_1, ... , p_k$ from left to right.
\end{proposition}
\begin{proof} This proposition is a straightforward probabilistic reformulation of Tutte's decomposition at the root proved in Proposition \ref{prop:tutte}.
\end{proof}

\paragraph{A word on canonical spaces.} In this work it will be convenient to deal with random variables of the same type (e.g.~random trees, random discrete processes, random c\`adl\`ag processes...) but having slightly different laws. Instead of working on an abstract probability space $ (\Omega, \mathbb{P},  \mathcal{F})$ and defining different random variables $X_{1},X_{2} ... : \Omega \to E$ having the different laws, we shall rather work on the canonical space $\Omega= E$ (e.g. $E=  \{ \mbox{all finite plane labeled trees}\}$)  endowed with its obvious Borel sigma-field and endow this canonical space with different probability measures $ \mathbb{P}_{1}, \mathbb{P}_{2}, ...$ under which the canonical random variable $ \mathrm{Id} :\Omega \to \Omega$ has the law of $X_{1}, X_{2}, ...$. This will make many of our arguments more transparent to the cost of paying attention to the probability symbol first!

\begin{center} \hrulefill \textit{\ In the sequel $ 0 < x \leq x_{\crit}$ is fixed and is dropped from notation for simplicity.} \hrulefill  \end{center}

\section{An underlying random walk}

\label{sec:rw}
In this section we make precise the connection with a random walk. We start by introducing the asymptotic law $\nu$ of the increments in the locally largest branch along our labeled trees under $ \mathbb{P}^x_{p}$ as $p \to \infty$. We then show that the locally largest exploration is obtained by a  simple $h$-transformation of the $\nu$-random walk with independent increments.  This connection will enable us to import results from  fluctuation and scaling limit theory of random walk and ultimately will pin down the values of the exponents $\beta^{x}$ to $\{ \frac{3}{2}, \frac{5}{2}\}$ in Theorem \ref{thm:magic}. 

Given a labeled tree $ \mathbf{t}= ( \mathrm{t}, \phi)$ where $\phi : \mathrm{t} \to \{0,1,2, ... \}$, we call the \textbf{locally largest branch} the set of vertices $\varnothing = u_0 \preceq u_1 \preceq u_2 \preceq ... \preceq u_\tau$  such that for every $0 \leq i \leq \tau-1$ the vertex $u_{i+1}$ is the children of $u_{i}$ having the largest label, and $ \tau$ is the first time $i$ such that $u_{i}$ is a leaf (and has thus no children). In case of tie, then we take the left-most of such children, see Figure \ref{fig:ll} below.

\begin{figure}[!h]
 \begin{center}
 \includegraphics[width=13cm]{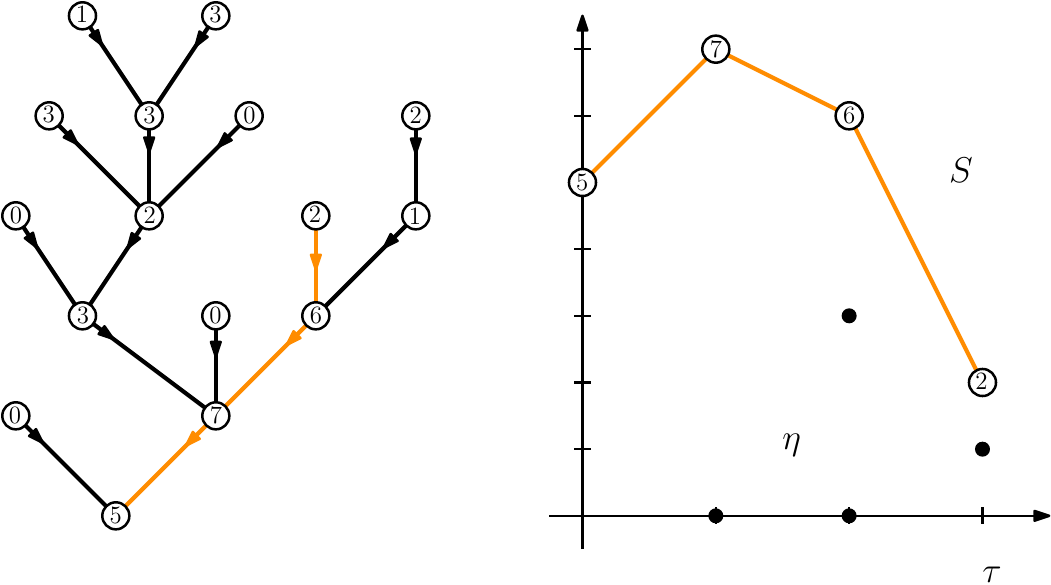}
 \caption{ Example of a locally largest exploration in a labeled tree: the locally largest branch is in orange. The decoration process $S$ is in orange on the right, whereas the reproduction measure $\eta$ is displayed with black dots on the right figure. \label{fig:ll}}
 \end{center}
 \end{figure}  
 We shall write below $S_i = \phi(u_i)$ for $0 \leq i \leq \tau$ for the labels along the locally largest branch and $Y^{i}_{j}$ for $1 \leq i \leq \tau$ and $1 \leq j \leq k_{u_{i-1}}(  \mathrm{t})-1$ for the labels $ \phi(v)$ of the siblings of $u_{i}$, i.e.~those vertices $v\ne u_{i}$ having $u_{i-1}$ for parent, ranked from left-to-right. Notice that this collection may be empty and has always less than $K-1$ numbers by our standing assumption. In the terminology of \cite[Section 6]{bertoin2024self} the process $S= (S_{i})_{0 \leq i \leq \tau}$ is called the \textbf{decoration process}, whereas the point measure 
 $$ \eta = \sum_{i \geq 1} \sum_{j \geq 1} \delta_{i , Y^{{i}}_{j}}$$ is called the \textbf{reproduction process} (it is here a random measure). Each atom  of $\eta$ is interpreted as a new particle created on the side of the locally largest exploration: It should be clear  from the multitype branching structure (see Proposition \ref{prop:GW}), that under $ \mathbb{P}_p$ and conditionally on $ (S,\eta)$,  the discrete labeled trees dangling from the locally largest branch are independent labeled trees of law $ \mathbb{P}_{Y^{i}_{j}}$. See \cite[Chap. 6.1]{bertoin2024self} for details.
 
 Recall for further use that if $(S)$ is the decoration process of a locally largest exploration, since the degree, car arrivals and parking spots capacities are bounded by $ \K$ by our standing assumptions, using the parking constraints \eqref{eq:parkingconstraint} we deduce that we must have 
  \begin{eqnarray} \label{eq:sautll} S_{i+1} \geq \frac{S_i- \K}{\K}, \mbox{ for all } 0 \leq i < \tau,  \end{eqnarray}
and even if $(S)$ is not the locally largest exploration (that is, regardless of branch followed) we have 
  \begin{eqnarray} -\K \leq \big(S_{i+1}-S_{i}\big) + \sum_{j \geq 1} Y_{j}^{i+1} \leq \K^{2} \quad \mbox{ and  }\quad \label{eq:sommedelta} S_i + \K \geq S_{i+1},  \qquad \mbox{ for all } 0 \leq i < \tau.   \end{eqnarray}

\subsection{Splitting large flux and measures $ \hat{\nu}$ and  $\nu$}
Recall that $x \in (0, x_\crit]$ has been fixed once and for all in this section, and is often dropped from notation. In particular, the coming measures $\hat{\nu}$ and $\nu$ implicitly depend on $x$. 
\begin{definition}[Introducing $\hat{\nu}$ and $\nu$]  \label{def:nuhat} We define a measure $\hat{\nu}$ on $ \mathbb{Z}_{\geq 0} \times \mathbb{Z}_{\geq 0} \times \bigcup_{k=0}^{\K-1}\big(( \mathbb{Z}_{\geq 0}) ^{k+1} \times ( \mathbb{Z}_{\geq 0}) ^{k}\big)$ by 
\begin{eqnarray*} 
\hat{\nu} \big(q ; c, (s_0, ... , s_k), (p_1, ..., p_k )\big) &:=&  x\cdot y_{\crit}^{-q} \cdot  (k+1)  w_{ c,k+1, (s_0, s_1, \dots, s_k)} \prod_{j= 1}^{k} {W}_{p_i}  \cdot \mathbbm{1}_{ \mathscr{C}}\\ 
&=&  x\cdot y_{\crit}^{-\sum_{i=0}^k s_i +c} \cdot  (k+1)  w_{ c,k+1, (s_0, s_1, \dots, s_k)} \prod_{j= 1}^{k} {\widetilde{W}}_{p_i}  \cdot  \mathbbm{1}_{ \mathscr{C}}
,
\end{eqnarray*}  
where the parking constraints $ \mathscr{C}$ are
 \begin{eqnarray} \label{eq:parkingconstraintnu} \sum_{i=1}^k p_i- \sum_{i=0}^{k}s_i +c = -q, \quad \mbox{ and }\quad\forall\ 1 \leq i \leq k,\ p_i \geq s_i  \end{eqnarray}
When dealing with labeled trees only (as opposed to parking trees), we will often consider the push-forward of this measure $\hat{\nu}$ on several coordinates, which will then be denoted by $*$ when forgotten. For example, the measure $\hat{\nu}(\cdot ; *,(*), (\cdot))$ is the induced measure on $(q ; (p_{1}, ... , p_{k}))$ in the above notation.  The measure $\hat{\nu}( \cdot ; *,(*),(*))$ supported $\{... , -1,0,1,2, ... , \K\}$ will simply be denoted by $\nu$. 
\end{definition}

The probabilistic interpretation of this measure is the following: 
\begin{proposition} \label{prop:nu_proba} The law of $(S_1-S_0 ; Y^1_1, Y^1_2,... , Y^1_k)$ under $ \mathbb{P}_p$ converges as $p \to \infty$ towards $\hat{\nu}(\cdot; *,(*), (\cdot))$. In particular, $\hat{\nu}$ is a probability measure. Furthermore, the different constants $C_j^x$ appearing in Lemma~\ref{lem:asympW} so that we have 
 \begin{eqnarray} \label{eq:tailnu} \widetilde{W}_p \underset{p \to \infty}{\approx}  p^{-\beta} \quad \mbox{ and } \quad \nu(-j) \underset{j \to \infty}{\approx}j^{-\beta},  \end{eqnarray}
for  $\beta \in (1, \infty) \backslash \{2,3, ... \}$ where $\approx$ means that the RHS is asymptotic to the LHS times a positive constant.  Furthermore  for all $ \ell \geq 1$, as $ A \to \infty$ we have 
 \begin{eqnarray} \label{eq:notwojumps} \hat\nu\left( \left \{(q ; *,(*), (p_1, ..., p_k)) : \exists  \mbox{ distinct }1 \leq i_1, \dots, i_{\ell} \leq k \mbox{ s.t. }\ p_{i_1}, \dots, p_{i_{\ell}} \geq A \right\} \right)  =  \mathcal{O}( A^{ \ell (1-\beta)}) .  \end{eqnarray}
\end{proposition}
This fact has already been observed (using explicit computations) in the context of random planar maps, see \cite[Section 4.1]{CurStFlour}. It is based on a ``big-jump" principle for polynomial tail random variables which has been used in many contexts see \cite[Lemma 2.5]{AS03}, \cite[Lemma 4.3]{CKdissections}, or \cite[Lemma 19]{ConCurParking}.
\begin{proof} Recall that by Proposition \ref{prop:GW}, under $ \mathbb{P}_p$ the tree $ ( \mathrm{t}, \phi)$ is a Bienaym\'e--Galton--Watson tree with offspring distribution given by 
\begin{eqnarray*}
 \pi_{p}(p_{1}, ... ,p_{k}) &=&  \sum_{ \substack{ c \geq 0 \\ \forall i , 0 \leq s_i \leq p_i \\ \sum_{i = 1}^{k} (p_i - s_i) +c =p}} w_{c, k, (s_1, \ldots, s_k)} \cdot x \cdot \frac{ W_{p_1}^x \ldots W_{p_k}^x}{ W_p^x}\\
&=&  \frac{1}{ \widetilde{W}_p^x}\sum_{ \substack{ c \geq 0 \\ \forall i , 0 \leq s_i \leq p_i \\ \sum_{i = 1}^{k} (p_i - s_i) +c =p}} (y_{\crit})^{c} w_{c, k, (s_1, \ldots, s_k)} \cdot x \cdot \prod_{i =1}^{k} \widetilde{W}_{p_i}^x y_{\crit}^{- s_i}
 \end{eqnarray*}

Recall from \eqref{eq:sautll} that one of the $p_i$'s must be larger than $(p-\K)/\K$.   
Actually, when $p$ is large we will show that the most likely scenario is that only one the $p_i$'s is of order $p$ and the others are $ \mathcal{O}(1)$. We start by evaluating the probability that two children are large. We make the change of variable $k \to k+1$ and write ${p+q}, p_1, ..., p_{k}$ for the labels of the children of the particle label $p$.

Fix $A >0$ large.  Let us compute the contribution of the configurations such that $p+q \geq (p-\K)/\K$ and $p_1 \geq A$. 
Recall from Lemma \ref{lem:asympW} that   \begin{eqnarray} \label{eq:asymptWlb} \widetilde{W}_{p} \threesim  p^{-\beta} \quad \mbox{ as }p \to \infty, \end{eqnarray} for  $ \beta > 1$, where $\threesim$ means that both sides are comparable up to positive multiplicative constants.  
Then, recall that  $ F(x, y_{\crit}) < \infty$ so that  for $0 \leq k \leq \K-1$ and for $(s_{0}, ... , s_{k}) \in \{0,... , \K \}^{k}$ fixed we have 
\begin{eqnarray*} \sum_{\substack{(q, p_1, \ldots, p_k)\\ \forall i \geq 1, p_i \geq s_i \ \ {p+q\geq s_{0}}\\ \sum_{i=1}^{k} (p_i - s_i)  + c =s_0 -q \\ (p-\K)/\K \leq p+q}} \widetilde{W}_{p+q}\prod_{i=1}^k \widetilde{W}_{p_i}
&\leq& \underbrace{\left( \sup_{q \geq (p-\K)/\K- p } \widetilde{W}_{p+q}\right)}_{\lesssim \left(\frac{p-\K}{\K} \right)^{ - \beta}}
\underbrace{\sum_{p_1 \geq A} \widetilde{W}_{p_1}}_{\lesssim \frac{(A-1)^{1-\beta}}{\beta-1}} \underbrace{\prod_{i = 2}^{k} \left( \sum_{p_i = 0}^{ + \infty} \widetilde{W}_{p_i}\right)}_{ = F(x, y_ \crit) \lesssim 1} \quad \lesssim \ \ 
 p^{- \beta} A^{- \beta+1},
\end{eqnarray*}
Recall that by our exchangeability assumption in $(*)$, the weights are invariant under permutation of the $ s_i$, and so the roles of the $p_i$ are also invariant under permutation. Since the local weight are bounded above by $\K$,  after dividing by $ \widetilde{W}_p$ and using \eqref{eq:asymptWlb}, we deduce that 
\begin{eqnarray} \label{eq:tightness} \forall p \ge 0, \qquad \mathbb{P}_p\left( \exists j \mbox{ s.t. } Y_j^1 > A \right) \lesssim A^{- \beta +1}.\end{eqnarray} Now, if $p_1, ... , p_k \geq 0$ and  $q \in \mathbb{Z}$ are fixed, for large $p$'s we have

\begin{eqnarray*}
&& \mathbb{P}_p (S_1 - S_0 = q ;  Y_1^1=p_1, Y^1_2=p_2, ... , Y_{k}^1=p_k)\\ 
& \underset{ \mathrm{exchg.}}{=}& (k+1)
\pi_p ( p+q, p_1, \dots , p_k)  \\
&=& \frac{ \widetilde{W}_{p+q}}{ \widetilde{W}_p} y_{\crit}^{-q} x  (k+1) \sum_{ c \geq 0} \sum_{\substack{(s_0, \ldots, s_{k})\\ 
{\forall i \geq 1, p_{i}\geq s_{i}}\\ 
\sum_{i=1}^{k} (p_i - s_i) + c = s_0 - q }} w_{c,k+1, (s_0, \ldots, s_k)} \prod_{i=1}^{k} {W}_{p_i}.
\end{eqnarray*}
In the view of the form of the asymptotic given by Lemma \ref{lem:asympW}, we first fix $ j \in \{0, \ldots, b-1\}$ and concentrate on those 
$p \to + \infty$ with $p \equiv j\ \mathrm{mod}\ b$. The last display then converges along that sequence towards 
$$\frac{ C_{j+q \mathrm{\ mod\ }b}}{ C_j}\cdot  \hat{\nu} (q; *, (*), (p_1, p_2, ... ,p_k)). $$
This establishes the weak convergence of the law of $(S_1 - S_0; (Y_j^1)_{j\geq 1})$ under $ \mathbb{P}_p$ towards the above distribution along $p \to \infty$ with $p \equiv j\ \mathrm{mod}\ b$. In particular, by Fatou's lemma, the last display is a sub-probability measure. However, it follows from \eqref{eq:tightness} {and \eqref{eq:sommedelta}} that the family of laws of $(S_1 - S_0; (Y_j^1)_{j\geq 1})$ under $ \mathbb{P}_p$  is tight for  $p \geq 0$. In particular, the previous display defines a probability distribution. Focusing on the first coordinate, we thus have 
$$\sum_{q \in \mathbb{Z}} \frac{ C_{j+q \mathrm{\ mod\ }b}}{ C_j} \nu (q) = 1 \quad \mbox{ that is } \quad  C_j = \sum_{-j \leq q \leq b -1- j} C_{j+q} \nu(b \mathbb{Z} + (j+q)).$$ 
In particular, if $J \in \mathrm{argmin} \{C_j : 0 \leq j \leq b-1\}$, then the inequality
$$C_J =\sum_{-J\leq q \leq b -1- J} C_{J+q} \nu(b \mathbb{Z} + (J+q)) \geq C_J \sum_{-J\leq q \leq b -1- J} \nu(b \mathbb{Z} + (J+q)) = C_J$$
must be an equality since $ \nu$ is a probability measure, implying (by the maximum principle) that all the $(C_j : 0 \leq j \leq b-1)$ must be equal to $C_J$, as desired. 
The estimates \eqref{eq:tailnu} and \eqref{eq:notwojumps} are then easy: notice that conditionally on $c,k+1,(s_0, ... , s_k)$ the variables $p_1, ... , p_k$ under $\hat{\nu}$ are independent random variables with polynomial tails of order $j^{-\beta}$ by Lemma \ref{lem:asympW}. It is easy to see that the sum of finitely many such variables keep the same polynomial tail. Eq.~\eqref{eq:notwojumps} is straightforward from the above remark. \end{proof}

\subsection{The locally largest exploration as an $h$-transform}
For $p \geq 0$, we shall introduce a new probability measure $ \mathbb{P}_{p}^{ \RW}$ (with associated expectation $ \mathbb{E}_{p}^{ \RW}$) under which the canonical process $(S_{i})_{i \geq 0}$  and its decoration process $\eta$ is given by i.i.d.~increments of law $\hat{\nu}$ (recall that $x \in (0, x_{\crit}]$ is implicit in the notion). More precisely, if $(X_{i}, ( Y^{i}_{j} : j \geq 1))_{i \geq 1}$ are i.i.d.\ according to $\hat{\nu}(\cdot ; *,(*), \cdot) $, under $ \mathbb{P}_{p}^{ \RW}$ we put 
$$ S_{0} = p, \quad S_{k} = S_{0} + \sum_{i=1}^{k}  X_{i},$$
$$ \eta = \sum_{i \geq 1} \sum_{j \geq 1} \delta_{i , Y^{{i}}_{j}}.$$
In particular, under $ \mathbb{P}_{p}^{ \RW}$ the process $(S)$ is a vanilla $\nu$-random walk started from $p$. {It does not generally correspond to a locally largest exploration, but \eqref{eq:sommedelta} remains true. } The more complex locally largest exploration process $(S, \eta)$ under $ \mathbb{P}_p$ can be recovered by $h$-transforming the latter: If $(S)$ is a $ \mathbb{Z}$-valued process and $\eta$ an atomic measure on $ \mathbb{Z}_{\geq 0} \times \mathbb{Z}_{ \geq 0}$ (when several atoms are confounded, the weight of the point belongs to $\{2,3, ...\}$), for $t \geq 0$ we denote by $ \mathcal{L}_{t}$ the event on which $(S,\eta)$ may be the \textbf{locally largest exploration} in a tree up to height $t$  \textbf{without any ties} i.e.~if $\eta = \sum_{i \geq 1} \sum_{j \geq 1} \delta_{i, Y^{i}_{j}}$ we have
  \begin{eqnarray} \label{def:Lt} S_{i} > Y^{i}_{j}, \quad \forall j \geq 1, \quad \mbox{ for all }i \leq t.   \end{eqnarray}
 Avoiding ties is the above definition is only a technical simplification (see Lemma \ref{lem:notiesll}). 
Below we use the shorthand notation $S|_{[0,t]} \equiv (S_i : 0 \leq i \leq t)$ and $\eta|_{[0,t]} \equiv \eta|_{[0,t] \times \mathbb{ Z}_{\geq 0}}$ for the restriction of the processes over the time intervalle $[0,t]$. The link between locally largest decoration-reproduction processes under $ \mathbb{P}_p$ and $ \mathbb{P}_p^\RW$ is given by the following: 
\begin{proposition}[Key formula] \label{prop:llrw}For any positive functional $f$ we have 
 \begin{eqnarray*} \mathbb{E}_{p}\left[ f\Big(S|_{[0,t]}, \eta|_{[0,t]}\Big) \mathbbm{1}_{ \mathcal{L}_{t}}\mathbbm{1}_{ S|_{[0,t]} \geq \K} \right] &=& \mathbb{E}_{p}^{ \RW}\left[ f\Big((S|_{[0,t]}, \eta|_{[0,t]}  \Big)    \mathbbm{1}_{ \mathcal{L}_{t}} \mathbbm{1}_{ S|_{[0,t]} \geq \K} \cdot  \frac{ \widetilde{W}_{S_{t}}}{ \widetilde{W}_{ S_{0}}}     \right].  \end{eqnarray*}
 \end{proposition}
\begin{proof} By the Markov property of both processes and the form of the above display, it is sufficient to prove the equality for $t=1$. 
In that direction, fix $p \geq \K$, $ q \geq \K - p$ and  fix $p +q > p_1, \dots, p_k \geq 0$  (so that the event $ \mathcal{L}_1$ is satisfied) and let us compute $ \mathbb{P}_p ( S_1 = p +q  \mbox{ and } Y^{1}_{j} = p_j, \ j \geq 1)$ using Proposition \ref{prop:GW}. Recall that the increments are exchangeable by our standing assumption so that we can exchange the locally only largest label with the first one without ambiguity. It gives
 \begin{eqnarray*} &&\hspace{-1.5cm} \mathbb{P}_p ( S_1 = p+q \mbox{ and } Y^{1}_{j} = p_j, \ j \geq 1) \\
 &=& (k+1) \pi_{p} (p +q , p_1, \ldots, p_k)  \\
 &=& (k+1) \sum_{c \geq 0 }  \sum_{ \substack{ (s_0, \dots, s_k) \\ {\forall i \geq 1 , s_i \leq p_i \ \ p+q \geq s_{0}} \\ \sum_{i = 1}^{k} (p_i - s_i) +c= s_0 - q}} w_{c, k+1, (s_0, \ldots, s_k)} \cdot x \cdot \frac{ W_{p+q} \cdot W_{p_1} \ldots W_{p_k}}{ W_p}.  \end{eqnarray*}
 Notice that since $p +q \geq \K$ the inequality $p + q \geq s_0$ is always satisfied so that setting $p+q$ aside we can write the last display as
 \begin{eqnarray*}
   \frac{W_{p+q}y_{\crit}^{p +q} }{W_p y_{\crit}^p} \cdot  \hat{\nu} \big(q ; *, (*), (p_1, p_2, ... , p_k)\big) = 
   \frac{ \widetilde{W}_{p+q} }{\widetilde{W}_p }  \cdot  \mathbb{P}_p^\RW ( S_1 = p+q \mbox{ and } Y^{1}_{j} = p_j, \ j \geq 1), 
 \end{eqnarray*}
 as desired. \end{proof}
 Before diving into the properties of the laws $\hat\nu,\nu$ let us explain why we named the above proposition the ``Key formula''. It obviously links the locally largest exploration $(S,\eta)$ under $ \mathbb{P}_{p}$ to a much simpler process with i.i.d.\ increments under $ \mathbb{P}_{p}^{\RW}$, thus justifying the fact that a random walk is hidden in the model. More importantly, notice that the partition functions $ \widetilde{W}_{p}$ appear twice in the right-hand side of the key formula: 
\begin{itemize}
\item first implicitly in the law $\nu$,  in particular through the negative tail in $p^{-\beta}$ of $\nu$ which comes from the asymptotic of $\widetilde{W}_{p}$ in \eqref{eq:tailnu},
\item  and second, it appears in the Radon--Nikodym derivative $ \displaystyle \frac{ \widetilde{W}_{S_{t}}}{ \widetilde{W}_{ S_{0}}}$.
\end{itemize} This double apparition will ultimately give an equation on $\beta$ which will pin down its value in $  \{3/2, 5/2\}$. For this we need to pass the key formula to the scaling limit and perform a calculation using stable L\'evy process. But in order to understand the scaling limits of $\nu$-random walk, we need to understand the tails of $\nu$ (as done in Proposition \ref{prop:nu_proba}) as well as its expectation: this is the goal of the next subsection. But before, let us tackle the problem of ties. Indeed, we show that under $ \mathbb{P}_{p}$ for large $p$'s, then the probability of ties is negligible. For $x \subset \mathbb{R}$, we denote by $$ \tau_{x} := \inf \{ n \geq 0 : S_n \leq x\} \quad \mbox{ and } \quad \tau_{\geq x} = \inf \{ n \geq 0 : S_n \geq x\}.$$ 
\begin{lemma}[No ties in the locally largest exploration.]\label{lem:notiesll}For all $p \geq 1$,
$$\mathbb{P}_p \left( \exists \mbox{ tie at step $1$} \right)  \lesssim p^{-\beta}. $$
In particular, for any $ \delta, A >0$, then we have 
$$ \mathbb{P}_{p}\big( \mathcal{L}_{ \tau_{\delta p} \wedge A p^{\beta-1}}\big) \xrightarrow[p\to\infty]{} 1.$$
\end{lemma}

\begin{proof} By definition, we have 
 \begin{eqnarray*} \mathbb{P}_{p}( \exists \mbox{ tie at step $1$}) &=& \mathbb{P}_{p}( \exists j \geq 1 \mbox{ s.t. } Y_j^1 = S_1)  \\
 &\underset{ \mathrm{exch.}}{\lesssim}&  \sum_{\substack{p_{1} =p_{2} \geq p_{3}, ... , p_{k}\\ \sum_{i = 3}^{k} p_i \in [p- 2p_1 - \K, p-2p_1 + \K^2]}} \pi_{p}(p_{1},p_{2}, p_{3}, ... ,p_{k})\\
 & \lesssim & \sum_{\substack{p_{1} =p_{2} \geq p_{3}, ... , p_{k}\\ \sum_{i = 3}^{k} p_i \in [p- 2p_1 - \K, p-2p_1 + \K^2]}}  \frac{\widetilde{W}_{p_{1}} \cdots \widetilde{W}_{p_{k}}}{\widetilde{W}_{p}}  \end{eqnarray*} 
where the  inequalities on $p_{i}$'s use the parking constraints \eqref{eq:sommedelta}. Recall from \eqref{eq:sautll}  that we must have  $(p-\K)/\K  \leq p_1\leq (p+ \K)/2$, and from the proof of Proposition \ref{prop:nu_proba} the fact that if $ \sum_{i=3}^{k} p_i$ is large, the most probable scenario is that only one of the $p_i$ is large and almost equal to the sum. We can then use the asymptotics on the partition function $\widetilde{W}_p$ given in \eqref{eq:tailnu} and  deduce that for large $p$'s the above display is bounded above by
 $$ \lesssim p^{\beta} \sum_{p_{1}=p_{2} \geq (p- \K)/\K} p_{1}^{-\beta} p_{2}^{-\beta} (p-2p_1)^{- \beta}\ \lesssim \ p^{-\beta} \sum_{i =0}^{p} i^{- \beta} \lesssim p^{- \beta}.$$
 The second point follows from a union bound on all times $0 \leq i \leq A p^{\beta-1}$.
\end{proof}

\subsection{Expectation of $\nu$}
\label{sec:expectationnnu}
 In this section, we use the key formula (Proposition \ref{prop:llrw}) to prove a dichotomy: either $\nu$ has an infinite (negative) drift or it has zero drift. We will see later (Proposition \ref{prop:dichotomyxxc}) that the first case appears when $x<x_{\crit}$ whereas the second happens when $x=x_{\crit}$. Combined with the tail estimates given in Proposition \ref{prop:nu_proba}, this will tell us that $\nu$ is in the domain of attraction of the spectrally negative $(\beta-1)$-stable random variable.

Recall from Definition \ref{def:nuhat} that $\nu$ has support in $ (- \infty, \K]$, and in particular its expectation denoted by $ \mathbb{E}[\nu]$ is well-defined and belongs to  $ [- \infty , \K]$. Recall that $x \in (0, x_\crit]$ has been fixed and is often dropped from notation.
 
 \begin{proposition}[Dichotomy (I), with drifts] \label{prop:nodrift} In the notation of Proposition \ref{prop:nu_proba}, we have the following dichotomy: 
 \begin{center}
 either $\beta  \in (2, \infty) \backslash\{3,4, ... \}$ and $\mathbb{E}[\nu] = 0$ or $\beta \in (1,2)$ and $\mathbb{E}[\nu] = -\infty$.
  \end{center}
\end{proposition}
\begin{proof} Recall the notation  $\tau_{ p} $ and $\tau_{ \geq p}$ for the hitting times of $(-\infty,p]$ and $[p, \infty)$ respectively  by the canonical process $(S)$. Suppose that $ \mathbb{E}[\nu]$ is finite (which implies that $ \beta > 2$ by \eqref{eq:tailnu}) and non-zero and consider the stopping time 
$$ \theta_{p} = \Big(\tau_{ \geq 3p/2}\Big) \wedge \Big(\tau_{p/2}\Big) \wedge \Big( \frac{p}{| \mathbb{E} [\nu]|}\Big).$$
The easy claim is that by the law of large numbers, under $ \mathbb{P}^{\RW}_{p}$ the stopping time $\theta_{p}$ coincides with high probability either with $\tau_{p/2}$ if  $ \mathbb{E}[\nu] <0$ or with $\tau_{ \geq 3p/2}$ if $ \mathbb{E}[\nu] >0$, that is we have
  \begin{eqnarray}  \mathbb{P}_p^{\RW} \left(\Big(\tau_{ \geq 3p/2}\Big) \wedge \Big(\tau_{ p/2}\Big) \geq \frac{p}{ |\mathbb{E}[\nu]| } \right) \xrightarrow[p\to\infty]{} 0.   \label{eq:lundesdeux}\end{eqnarray}
Furthermore, we claim that the event $ \mathcal{L}_{ \theta_p }$ (corresponds to a locally largest exploration with no ties up to time $\theta_p$) is satisfied with high probability {under $\mathbb{P}_p^{\RW}$}. Indeed
by Proposition \ref{prop:nu_proba}, we have
$$ \mathbb{P}_p^{\RW} \big( \sup_{1 \leq j \leq \K} {Y_{j}^1} \geq  \frac{p}{3 \K}\big) \lesssim p^{- \beta+1}.$$
Hence,  
$$ \mathbb{P}_p^{\RW} \left( \exists i \leq  \theta_{p}, 1 \leq j \leq K :  Y_{i}^j \geq \frac{p}{3 \K}\right) \lesssim p^{- \beta+2} \xrightarrow[p\to\infty]{} 0,$$ because $\beta >2$. Notice now that thanks to \eqref{eq:sommedelta} when $p$ is large enough, as long as $S_{i-1} \geq p/2$, if $Y_{i}^j \leq \frac{p}{3 \K}$, we must have $S_i > \max Y_j^i$ and thus corresponds to a locally largest exploration with no ties. Thus, 
\begin{align*} 
\mathbb{P}_p^{ \RW} \left(  \mathcal{L}_{\theta_p} \right) \ \geq\  \mathbb{P}_p^{ \RW} \left(  \forall i  \leq \theta_{p}, 1 \leq j \leq \K \mbox{ we have }   Y_{i}^j \leq  \frac{p}{3 \K}\right) \  \xrightarrow[p\to\infty]{} 1.
\end{align*}
Using \eqref{eq:sommedelta} again, notice that $S_{\tau_{ \geq3p/2}} \leq  \frac{3p}{2} + \K$ if $\tau_{ \geq 3p/2}< \infty$ and by \eqref{eq:sautll} on the event $ \mathcal{L}_{\theta_{p}}$ we have $S_{\tau_{p/2}} \geq  \frac{p/2-\K}{\K}$. Hence by \eqref{eq:tailnu}, the variable ${ \widetilde{W}_{ \theta_p}}/{ \widetilde{W}_p}$ is bounded and converges to $(3/2)^{-\beta}$ when $\theta_p = \tau_{ \geq 3p/2}$ or is asymptotically larger than $2^\beta$ when $\theta_p = \tau_{p/2}$. We deduce from the above discussion that  
\begin{eqnarray*}  
1 \xleftarrow[p \to \infty]{ \mathrm{Lem.\ } \ref{lem:notiesll}} \quad \mathbb{P}_{p}( \mathcal{L}_{\theta_{p}}) \underset{ \mathrm{Prop.\ } \ref{prop:llrw}}{=}\mathbb{E}_p^{ \RW} \left[ \mathbbm{1}_{(S, \eta) \in \mathcal{L}_{\theta_p}} \frac{ \widetilde{W}_{ \theta_p}}{ \widetilde{W}_p} \right ]  \quad \xrightarrow[p\to\infty]{ \eqref{eq:lundesdeux}} \begin{cases}
	\geq 2^{\beta} &\text{ if $ \mathbb{E} [\nu] < 0$}\\
	\left(\frac{3}{2}\right)^{-\beta}  &\text{ if $ \mathbb{E} [\nu]> 0$.}
\end{cases}
\end{eqnarray*}
This is a contradiction, hence $ \mathbb{E}[\nu]=0$ or $ \mathbb{E}[\nu] = -\infty$. \end{proof}

\section{$\beta$ belongs to $\{3/2, 5/2\}$ via Lamperti's calculus} 
\label{sec:lampertistable}
In this section we prove:
\begin{proposition}[Dichotomy (II), with exponents] \label{dicho2} The exponent $\beta \in (1,\infty) \backslash \{2,3, ... \}$ appearing Proposition \ref{prop:nu_proba}  must belong to $\{3/2, 5/2\}$, that is $ \beta = \frac{3}{2}$ if $ \mathbb{E}[\nu]= \infty$ and $\beta=  \frac{5}{2}$ if $\mathbb{E}[\nu]=0$.
\end{proposition} 
  \begin{center} \hrulefill \textit{\ In the sequel we shall thus speak of the case $\beta = 3/2$ or $\beta=5/2$.} \hrulefill  \end{center}
  
  The proof of the previous result is done by passing the proof of Proposition \ref{prop:nodrift} (which is based on the Key formula) to the scaling limit and performing a calculation with stable L\'evy process (using the Lamperti representation of stable L\'evy processes killed when entering the negative half-line). {This section is the only one where continuous objects and stochastic calculus is used.} Let us start with some background on L\'evy processes and their domain of attraction. 

\subsection{Background on stable L\'evy processes}

Recall from Proposition \ref{prop:nu_proba} that the law $\nu$ has a support bounded from above, satisfies {$\nu(-j) \approx \ j^{-\beta}$} as $j \to \infty$ for some $\beta \in (1, \infty) \backslash \{2,3, 4, ... \}$ and  is centered whenever $\beta \in (2, \infty)$ by Proposition~\ref{prop:nodrift}. This implies in particular that the $\nu$-random walk is in the domain of attraction of the $((\beta-1) \wedge 2)$-stable L\'evy process. Let us explain this in more details. Introduce the shorthand notation  $$\gamma = (\beta-1) \wedge 2 \quad \in (0,1)\cup(1,2],$$ and for $x \in \mathbb{R}$, under the law $ \mathbf{P}_{x}$ let $(\xi_{t} : t \geq 0)$ be a stable spectrally negative L\'evy process, starting from $x$, with L\'evy--Khintchine exponent $  \mathbf{E}_x[\exp(  -\lambda  (\xi_{t}-x))] = \exp( t \lambda^{\gamma})$ for $\lambda \leq 0$, see \cite{Ber96,kyprianou2022stable} for background. In particular, when $\gamma \in (0,2) \backslash \{1\}$ it has L\'evy measure
 \begin{eqnarray} \label{eq:levymeasure} \frac{\gamma(\gamma-1)}{\Gamma(2-\gamma)} \  \frac{\mathrm{d}r}{|r|^{\gamma+1}}  \mathbbm{1}_{r <0}, \end{eqnarray}
and no Brownian part. When $\gamma \in (0,1)$ the process is the opposite of a pure-jump subordinator, while for $ \gamma \in (1,2)$   it may increase but has only negative jumps. We use the notation $\Delta \xi_t = \xi_t - \xi_{t-}$ for the jump of $\xi$ at time $t$. For $\gamma = 2$, it is equal to $ \sqrt{2}$ times a standard linear Brownian motion issued from $x$. Notice that the degenerate $\gamma =1$ where we have $\xi_{t} = x-t$ under $ \mathbf{P}_{x}$ has been excluded since we already know from analytic combinatorics that $\beta \ne 2$ (see Lemma \ref{lem:asympW}).

Recall also that under $ \mathbb{P}^{\RW}_{p}$ the process $(S)$ is a $\nu$-random walk starting from $p$. Since\footnote{The cases $\gamma=1$ or $\gamma=2$  are technically challenging where delicate centering or variance estimates are needed to establish convergence towards a stable law, whereas when $\gamma \in (0,1)\cup(1,2)$ we only need tail estimates and centering (when $\gamma >1$).} we have $\beta \ne \{2,3\}$, it follows from \cite[Theorem 8.3.1]{BGT89} that $\nu$ is in the domain of attraction (without centering) of the law $\xi_{1}$ under $ \mathbf{E}_0$. In turns, by \cite[Theorem 15.17]{Kal07} or \cite[Chapter VII]{JS03}, this implies the stronger convergence of processes in the Skorokhod sense:
$$ \left(\frac{S_{\lfloor n^{\gamma}t \rfloor }}{ n}\right)_{t \geq 0}  \mbox{\  under } \mathbb{P}^{\RW}_{0} \quad \xrightarrow[n\to\infty]{(d)} \quad  \left(  \mathfrak{c} \cdot \xi_{t}\right)_{t \geq 0}  \mbox{ \ under } \mathbf{P}_{0}, \qquad \mbox{ where } \mathfrak{c} >0.$$
{The constant $ \mathfrak{c}$ is related to the tail of $\nu$, so is unknown in general.} The preceding convergence is easily extended to get a scaling limit of the decoration-reproduction processes under $ \mathbb{P}_p^{ \RW}$:

\begin{lemma}[Scaling limit of the decoration-reproduction processes]  \label{lem:scalinglimit} Consider the rescaled processes $ S^{(p)}_{t} = p^{-1} S_{[p^{\gamma}t]}$ and the point measure $\eta^{(p)} = \sum_{i \geq 1}\sum_{j \geq 1} \delta_{i/p^{\gamma}, Y^{i}_{j}/p}$. Then we have 
$$  \left( S^{(p)} , \eta^{(p)}\right) \mbox{ \ under } \mathbb{P}_{p}^{ \RW} \quad \xrightarrow[p\to\infty]{(d)} \quad  \Big( (\mathfrak{c}\cdot  \xi^{}_{t})_{t \geq 0} , \sum_{\begin{subarray}{c}t \geq 0 \\    \Delta \xi^{}_{t}>0 \end{subarray}} \delta_{t, \mathfrak{c}\cdot \Delta \xi^{}_{t}} \Big) \mbox{ \ under } \mathbf{P}_{1},$$  where the convergence holds for the product topology with the Skorokhod topology for the first coordinate and  the vague convergence of measures on $ \mathbb{R}_{+} \times \mathbb{R}_{+}^{*}$ for the second one.
\end{lemma}
\begin{proof} The convergence of the first coordinate is granted by the above discussion. We thus need to show that the atoms of the measure are prescribed by the jumps of the process. 
Recalling \eqref{eq:sommedelta} the proof reduces to showing that a large negative jump  $S_i-S_{i-1}$ corresponds to a single large tree of label $Y^i_j$ (as opposed to several large trees whose total sum would roughly be equal to $S_i-S_{i-1}$). It thus suffices to show that for any $t \geq0$
$$ \sup_{ 0 \leq s \leq t\,p^{\gamma}} \frac{\big(\sum_{j \geq 1} Y^{s}_{j}\big) - \max_{j \geq 1} Y^{s}_{j}}{p} \quad \xrightarrow[p\to\infty]{ ( \mathbb{P}^{{\tiny{\RW}}}_{p})}\quad  0.$$ 
But recall from  Proposition \ref{prop:nu_proba} that
$$ \sup_{r \geq 0 }\mathbb{P}^{\RW}_{r}(\exists i \ne j : Y^{1}_{i} \mbox{ and } Y^{1}_{j} \geq  \varepsilon p) \lesssim (\varepsilon p)^{2(1-\beta)}.$$ The penultimate display easily follows after a union bound since $p^\gamma \times p^{2(1-\beta)} \leq p^{\beta-1+2-2\beta} = p^{1-\beta} \to 0$ as $p \to \infty$. This completes the proof.\end{proof}

\subsection{Identifying $\beta$ using Lamperti transformation}
Let us now pass the Key formula to the scaling limit. Fix a small $\delta>0$, a large $B >0$ and consider the stopping time 
$$\theta_p = \big(\tau_{\delta p}\big) \wedge B p^{\gamma}.$$ Recall that by Lemma \ref{lem:notiesll} there are no ties under $ \mathbb{P}_{p}$ until time $\theta_p$ with high probability, that is  $ \mathbb{P}_{p}( \mathcal{L}_{\theta_p}) \to 1$. Since by \eqref{eq:sautll} on the event $ \mathcal{L}_{\theta_p}$ we have $S_{\theta_p} \geq  (\delta p- \K)/\K$, by the asymptotic \eqref{eq:tailnu} on the partition function $ \widetilde{W}$, the function ${ \widetilde{W}_{S_{\theta_p}}}/{ \widetilde{W}_{p}}$ is bounded and is asymptotically of the same order as $(S_{\theta_p}/p)^{-\beta}$. One can then apply the Key formula (Proposition \ref{prop:llrw}) for  the function $f=1$ at the bounded stopping time $\theta_p$ and deduce that 
$$ 1 \xleftarrow[p \to \infty]{ \mathrm{Lem. }\ \ref{lem:notiesll}}  \mathbb{P}_{p}\big( \mathcal{L}_{\theta_p} \big) =\mathbb{E}_{p}^{ \RW}\left[  \frac{ \widetilde{W}_{S_{\theta_p}}}{ \widetilde{W}_{p}} \cdot  \mathbbm{1}_{ \mathcal{L}_{ \theta_p}}\right].$$
On the continuous side, since we have only binary splitting in the limit, being a locally largest exploration is the fact that $|\Delta \xi_s| \leq  \frac{\xi_{s-}}{2}$ for all $s$'s, and by standard properties of L\'evy process, the analog of ties happens with probability zero, that is under $ \mathbf{P}_{1}$ we have $|\Delta \xi_{s}| \ne \frac{\xi_{s-}}{2}$  for all $s \geq 0$, almost surely. We deduce from Lemma \ref{lem:scalinglimit} and the continuous mapping theorem that 
  \begin{eqnarray} \label{eq:passagelimite} 1 = \lim_{p \to \infty} \mathbb{E}_{p}^{ \RW}\left[  \frac{ \widetilde{W}_{S_{\theta_p}}}{ \widetilde{W}_{p}} \cdot  \mathbbm{1}_{ \mathcal{L}_{ \theta_p}}\right] \underset{ \mathrm{Lem.\ } \ref{lem:scalinglimit}}{=} \mathbf{E}_{1}\left[ \frac{(\xi_{\vartheta}^{})^{-\beta}}{1} \cdot \mathbbm{1}\left\{ |\Delta \xi_s| <  \frac{\xi_{s-}}{2} : \forall s \leq \vartheta \right\}\right],
 \end{eqnarray} where in the continuous setting we have set analogously $\vartheta = \tau_{\delta} \wedge B$ with the (abuse of) notation $\tau_\delta = \inf\{s \geq 0 : \xi_s \leq \delta\}$. Since $\xi_\vartheta > \delta/2$ on that event, and that under $ \mathbf{P}_{1}$ we almost surely have $\tau_{\delta} < \infty$, one can let $B \to \infty$ and deduce by dominated convergence that 
  \begin{eqnarray} \label{eq:keyformulalimit} \forall \delta \in (0,1), \qquad 1= \mathbf{E}_{1}\left[(\xi_{\tau_{ \delta}}^{})^{-\beta} \cdot \mathbbm{1}\left\{ |\Delta \xi_s| <  \frac{\xi_{s-}}{2} : \forall s \leq \tau_{\delta} \right\}\right].  \end{eqnarray}
The previous display is an implicit equation on $\beta$ which forces it to be equal to $3/2$ or $5/2$:
 \begin{theorem}[Lamperti's magic] \label{thm:magic}Let $\beta \in (1,\infty) \backslash \{2,3, ... \}$ and recall that $\gamma = (\beta-1) \wedge 2$. Let $\xi$ be a $\gamma$-stable spectrally negative L\'evy process started from $1$ under $ \mathbf{P}_{1}$. For $\delta \in (0,1)$ denote $\tau_{ \delta} = \inf\{ s \geq 0 : \xi_{s} \leq \delta\}$ and suppose that \eqref{eq:keyformulalimit} holds. Then necessarily we have  $ \beta = 3/2$ or $\beta = 5/2$.
 \end{theorem}
 
 Before starting the proof let us explain how we will compute the RHS of \eqref{eq:keyformulalimit} using  standard tools from the theory of L\'evy processes and especially the Lamperti representation of stable L\'evy processes killed when becoming negative. The key idea is to notice that \eqref{eq:keyformulalimit} involves a local constraint which is scale invariant and only depends upon the process $\xi$ killed at time $\tau_{ 0}$ when it reaches negative values (since then the indicator function cannot be satisfied anymore).  In particular,  by the seminal work of Lamperti \cite{Lam72}, under $ \mathbf{P}_{x}$ for $x >0$ the process $X_{s} = \xi_{s} \mathbbm{1}_{s \leq \tau_{0}}$ obtained from $\xi$ started from $x$ and killed at time $\tau_{ 0}$ is a $\gamma$-positive self-similar Markov process (pssMp), that is 
 $$ \Big(x\cdot X_{t / x^{\gamma}}\Big)_{0 \leq t \leq \tau_{ 0} \cdot x^{\gamma}} \mbox{ under } \mathbf{P}_{1} \quad \overset{(d)}{=} \big(X_{t}\big)_{0 \leq t \leq \tau_{ 0}} \mbox{ under } \mathbf{P}_{x}.$$ This process can thus be seen, after time change, as the exponential of another L\'evy process called $\zeta$. Specifically, we can introduce the process $(\zeta_{s})_{s \geq 0}$ defined by 
   \begin{eqnarray} \label{eq:lamperti}	 \xi_t \mathbbm{1}_{t \leq \tau_0} = X_t =  \mathrm{e}^{\zeta_{\kappa(t)}}, \quad \mbox{ where} \quad \int_{0}^{\kappa(t)} \mathrm{e}^{\gamma \zeta_{s}} \mathrm{d}s = t, \end{eqnarray} as long as $t \leq \int_{0}^{\infty} \mathrm{e}^{\gamma \zeta_{s}} \mathrm{d}s = \tau_{ 0}$. The crucial observation of Lamperti is that under $ \mathbf{P}_x$ the process $\zeta$ is a L\'evy process started from $\log x$ with explicit characteristics (see below). We refer to \cite{kyprianou2022stable} for background on this transformation. We can then translate \eqref{eq:keyformulalimit} using $\zeta$ via the Lamperti transformation. Indeed, if $\varsigma_x = \inf\{ s \geq 0 : \zeta_s \leq x\}$ is the hitting time of $(-\infty, x]$ by $\zeta$ (notice the different notation not to confuse with $\tau_x$ which applies to $\xi$), then  in particular we have $$\xi_{\tau_{ \delta}} = \exp( \zeta_{\varsigma_{ \log \delta}})$$  and the event $\{ |\Delta \xi_s| <  \frac{\xi_{s-}}{2} : \forall s \leq \tau_{\delta} \}$ can be translated in the Lamperti transformation into 
   $$ \left\{ |\Delta \xi_s| <  \frac{\xi_{s-}}{2} : \forall s \leq \tau_{\delta} \right\} = \left\{\Delta \zeta_{s} > -\log(2) : \forall s \leq \varsigma_{ \log \delta}\right\}.$$ Recalling that under $ \mathbf{P}_1$ the L\'evy process $\zeta$ starts from $\log 1 = 0$, we get from \eqref{eq:keyformulalimit} that for any $x <0$ we have 
 \begin{eqnarray} \mathbf{E}_{1}\left[    \exp( - \beta \zeta_{\varsigma_{ x}}) \mathbbm{1}_{\{\Delta \zeta_{s} \geq -\log(2) : \forall s \leq \varsigma_{ x}\}} \right] =1,   \label{eq:magiclamperti}\end{eqnarray}
 and the proof will use this equation.

 \begin{proof}[Proof of Theorem \ref{thm:magic}]  For clarity, we split the calculation depending on whether $\beta \in (1,2)$, where both $\xi$ and $\zeta$  are pure jump, or in the case $\beta \in (2,3)$, where compensation is involved. We start with the easy case $ \beta > 3$ where $\xi$ is a Brownian motion.
 
 \noindent \textsc{Case without jump $\beta >3$.} In this case, since Brownian motion has continuous sample paths, for any $\delta >0$ the indicator function in \eqref{eq:keyformulalimit} is plainly verified and $\xi_{\tau_{ \delta}} =\delta$ so that  we have 
 $$ 1 = \mathbf{E}_{1}\left[{(\xi_{\tau_{ \delta}}^{})^{-\beta}} \cdot \mathbbm{1}\left\{ |\Delta \xi_s| <  \frac{\xi_{s-}}{2} : \forall s \leq \tau_{ \delta} \right\} \right] =  (\delta)^{-\beta},$$
 which is absurd since $\delta \in (0,1)$ and $\beta >1$. 
 \medskip

In the general case $\beta \in (1,3)$, we will first use forthcoming Lemma \ref{lem:martingale}  to argue that \eqref{eq:magiclamperti} will imply a similar equation for fixed time: for any $t >0$ we have   \begin{eqnarray} \label{eq:exp1}\mathbf{E}_{1}\left[ \exp( - \beta \zeta_{t}) \cdot \mathbbm{1}_{\left\{\Delta \zeta_{s} \geq -\log(2) : \forall s \leq t\right\}}\right]=1.  \end{eqnarray} We shall then compute the expectation on the LHS using the exact characteristics of the process $\zeta$ in each case $\beta \in (1,2)$ and $\beta \in (2,3)$ and find only one possible solution in each interval:\medskip

 \noindent \textsc{Pure jump case.} When $\beta \in (1,2)$, the process $\xi$ is the opposite of a subordinator which is pure jump with L\'evy measure proportional to $ \mathrm{d}x\, x^{-\beta} \mathbbm{1}_{ x < 0}$ by \eqref{eq:levymeasure}.  By \cite[Theorem 5.10 and 5.15]{kyprianou2022stable}, the L\'evy process $ \zeta$ in the Lamperti representation of $\xi_{s} \mathbbm{1}_{s \leq \tau_{0}}$ is again the opposite of a pure-jump subordinator, started from $0$ with a killing rate $ \mathrm{k}^{\dagger}_{\beta}= \frac{1}{\Gamma(2-\beta)}$ and whose L\'evy measure is, after push-forward by $ x \mapsto \mathrm{e}^{x}$,  given by 
$$ \pi^{\dagger}_{\beta}( \mathrm{d}x) = \frac{-1}{\Gamma(1-\beta)}\cdot  \frac{ \mathrm{d}x}{(1-x)^{\beta}} \mathbbm{1}_{x \in [0,1]}.$$
With this at hands, the LHS of \eqref{eq:exp1}  is easily computed using the exponential formula for Poisson measure \cite[Chap. 0.5]{Ber96}, and is equal to $ \exp(t \psi( \beta))$ for the function $\psi(\beta)$ given by 
  \begin{eqnarray} \psi(\beta) & = &\int_{0}^{1} \underbrace{\frac{-1}{\Gamma(1-\beta)} \frac{ \mathrm{d}x}{(1-x)^{\beta}}}_{ \mbox{L\'evy measure}} \left(x^{-\beta} \mathbbm{1}_{x> \frac{1}{2}}-1\right) \underbrace{- \frac{1}{\Gamma(2- \beta)}}_{ \mbox{ killing }}\label{eq:calculpoisson}\\ &=& - 2^{\beta-1}\left( 2^{\beta}\cos((\beta-1)\pi) \Gamma(\beta- \frac{1}{2}) \pi^{-1/2} - {}_{2}F_{1}(-\beta+1,\beta,2-\beta, 1/2)\right),\nonumber   \end{eqnarray}
  where ${}_{2}F_{1}$ is the regularized hypergeometric function. Using a formal calculus software, the above equation is easily seen to have $\beta = \frac{3}{2}$ as its unique root, for example, one can check that the function is increasing over $(1,2)$ (its derivative is negative) and $\beta=3/2$ is a root.

    \begin{figure}[!h]
   \begin{center}
   \includegraphics[width=8cm]{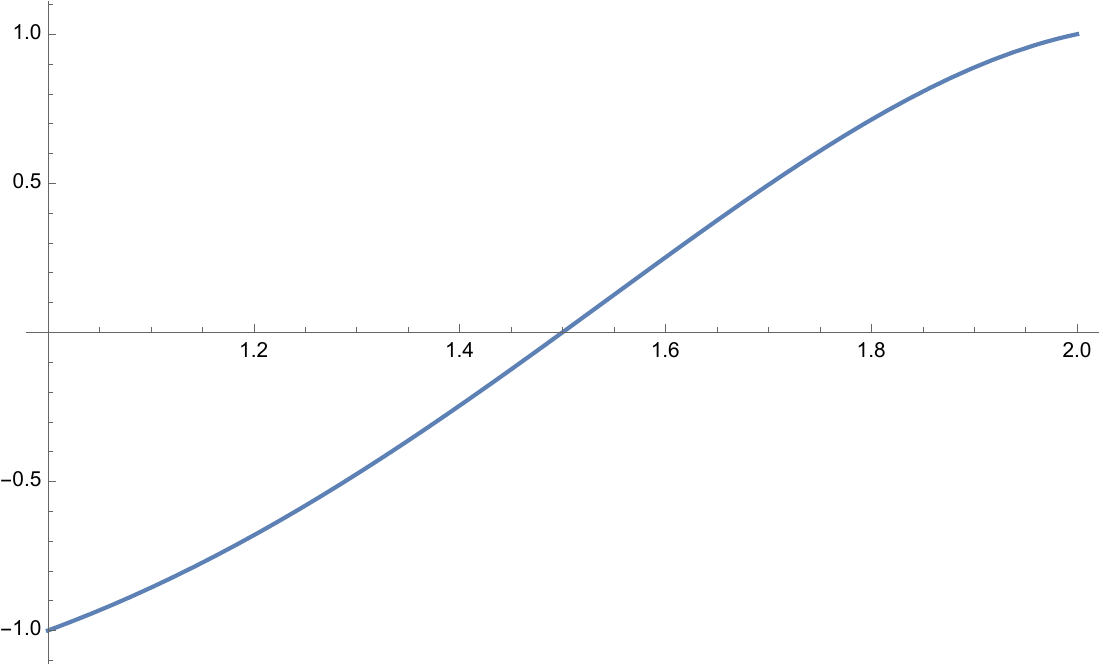}
   \caption{A plot of the function appearing in  \eqref{eq:calculpoisson} with the only root at $\beta=3/2$.}
   \end{center}
   \end{figure}

   \noindent \textsc{Compensation case.} When $\beta \in (2,3)$, we can still perform the Lamperti transformation \eqref{eq:lamperti} on the  $\xi_{s} \mathbbm{1}_{s \leq \tau_{0}}$ and this time, by \cite[Theorem 5.10 and 5.15]{kyprianou2022stable}, the process $\zeta$ is a L\'evy process with no Brownian part, killing rate $  \mathrm{k}_\beta^\dagger=\frac{\beta-2}{\Gamma(3-\beta)}$, and L\'evy measure given, after a push-forward by $ x \mapsto \mathrm{e}^{x}$,   by 
$$ \pi^{\dagger}_{\beta}( \mathrm{d}x) = \frac{-\Gamma(\beta)\sin((\beta-1) \pi)}{\pi} \frac{ \mathrm{d}x}{(1-x)^{\beta}}.$$ 
 To finish specifying the characteristic of $\zeta$ under $ \mathbf{P}_{1}$, one should give its drift which is the coefficient in front of $z$ in  the following normalization of its L\'evy Khtinchine exponent ($\gamma$ stands for Euler's Gamma constant)
  \begin{eqnarray*} \mathbf{E}_{1}[ \mathrm{e}^{z \zeta_{t}}] & =& \exp( t \psi^{\dagger}(z))\\
\mbox{ where }  \psi^{\dagger}(z)&=& \int_{0}^{1} 
\pi^{\dagger}_{\beta}( \mathrm{d}x) \big( x^{z}-1-z\log x\big) -  \mathrm{k}_\beta^\dagger - z \left( \gamma + \frac{\Gamma'(2-\beta)}{\big(\Gamma(2- \beta)\big)^2} \right) \\
&=&  \frac{1}{\pi}\left(\Gamma(\beta-1 - z) \Gamma(1 + z) \sin(\pi (\beta-1-z))\right) = \frac{\Gamma(1+z)}{\Gamma(2-\beta+z)}. \end{eqnarray*} 
The analog of \eqref{eq:calculpoisson} is now obtained using compensation in the exponential formula for Poisson measures and this yields to 
$ \mathbf{E}_{1}[\exp( -\beta \zeta_{t}) \mathbbm{1}_{\{\Delta \zeta_{s} \geq -\log(2) : \forall s \leq t\}}] = \exp(t \psi(\beta))$ where 
\begin{eqnarray} \label{eq:calculcompensation} \psi(\beta) & = &  \int_{0}^{1} 
\pi^{\dagger}_{\beta}( \mathrm{d}x) \big( x^{-\beta}\mathbbm{1}_{x>1/2}-1+ \beta \log x\big) -  \mathrm{k}_\beta^\dagger + \beta  \left( \gamma + \frac{\Gamma'(2-\beta)}{\big(\Gamma(2- \beta)\big)^2} \right).  \end{eqnarray}

This again has an exact expression in terms of Incomplete Beta and Hypergeometric functions (which we skip to spare the reader), and one can check (with the help of  a formal calculus software) that the above function admits a unique root in 
$(2,3)$ at  $\beta = \frac{5}{2}$. 
\end{proof}
    \begin{figure}[!h]
   \begin{center}
   \includegraphics[width=8cm]{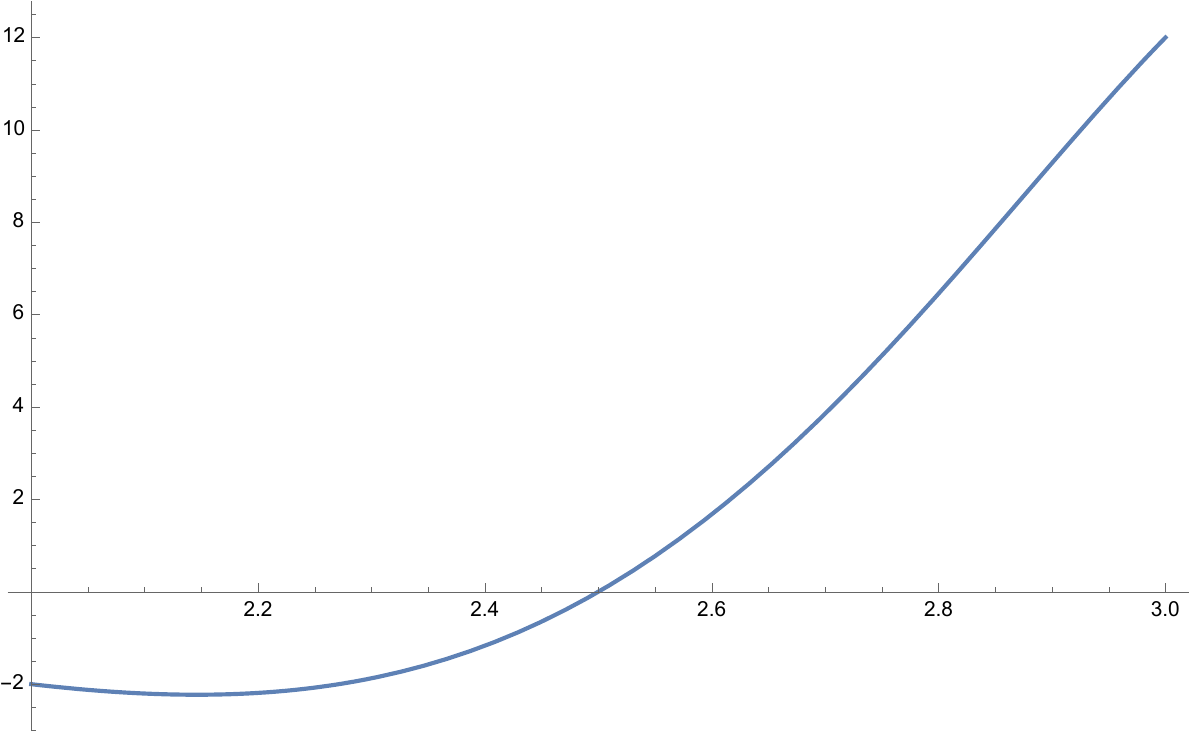}
   \caption{A plot of the function appearing in  \eqref{eq:calculcompensation} with the only root at $\beta=5/2$.}
   \end{center}
   \end{figure}

  \begin{lemma} \label{lem:martingale} Under the assumption \eqref{eq:magiclamperti}, under $ \mathbf{E}_1$  the process $t \mapsto \Big( \exp( - \beta \zeta_{t}) \mathbbm{1}_{\{\Delta \zeta_{s} \geq -\log(2) : \forall s \leq t\}}\Big)$ is a martingale for its canonical filtration and in particular has constant expectation $1$.
  \end{lemma}
  \begin{proof} By the Markov property of the process $\zeta$ and the multiplicative form of the process considered, the martingale property is granted once we have proved it has constant expectation. Fix $t >0$ and take $x<0$ negative. By the Markov property applied at time $t$ we have 
 \begin{eqnarray*} 1 &\underset{\eqref{eq:magiclamperti}}{=}& \mathbf{E}_{1} \left[  \exp( - \beta \zeta_{\varsigma_{ x}}) \mathbbm{1}_{\{\Delta \zeta_{s} \geq -\log(2) : \forall s \leq \varsigma_{ x}\}} \right]\\
 &=& \mathbf{E}_{1} \left[   \mathbbm{1}_{t \geq \varsigma_{ x}}\exp( - \beta \zeta_{\varsigma_{ x}}) \mathbbm{1}_{\{\Delta \zeta_{s} \geq -\log(2) : \forall s \leq \varsigma_{ x}\}} \right] \\ 
 && +  \mathbf{E}_{1} \Big[   \mathbbm{1}_{t < \varsigma_{ x}} \underbrace{\mathbb{E}\left[ \exp( - \beta \zeta_{\varsigma_{ x}}) \mathbbm{1}_{\{\Delta \zeta_{s} \geq -\log(2) : \forall s \leq \varsigma_{ x}\}} \mid \mathcal{F}_{t}\right]}_{ = \exp( - \beta \zeta_{t}) \mathbbm{1}_{\{\Delta \zeta_{s} \geq -\log(2) : \forall s \leq t\}}} \Big].  \end{eqnarray*}
 We can then let $x \to -\infty$ and get by monotone convergence that the second term tends to the desired expectation. To finish the proof it remains to show that the first term tends to $0$ as $x \to -\infty$ for fixed $t$. For this we bound it as follows 
 $$\mathbf{E}_{1} \left[   \mathbbm{1}_{t \geq \varsigma_{ x}}\exp( - \beta \zeta_{\varsigma_{ x}}) \mathbbm{1}_{\{\Delta \zeta_{s} \geq -\log(2) : \forall s \leq \varsigma_{ x}\}} \right] \leq \mathrm{e}^{-\beta (x- \log 2)} \cdot  \mathbf{P}_{1}\big( \varsigma_{ x} \leq t \mbox{ and  no jumps  }  \leq- \log 2\big).$$
 We claim that for any L\'evy process $ \mathfrak{z}$ started from $0$ with jumps bounded from below (here by $-\log 2$), the probability to reach a very negative level $x$ by the fixed time $t$ decreases faster than any exponential in $|x|$, so that the above display tends to $0$ as $x \to -\infty$. To see this, by applying  the Markov property at the first time the process $ \mathfrak{z}$ drops below $x$, we see that the desired probability is upper bounded by a constant  times the probability that such a L\'evy process is below $x/2$ (or even $x-1$) at time $t$. Using  Markov inequality this is itself upper bounded for any $a >0$ by 
 $$  \mathbb{P}( \mathfrak{z}_{t} \leq \frac{x}{2}) \leq \mathbb{E}[ \exp(-  a \mathfrak{z}_{t})] \cdot \mathrm{e}^{a  x/2}.$$
 By the L\'evy--Khintichine formula, any L\'evy process $ \mathfrak{z}$ with L\'evy measure supported on $[- \mathrm{cst}, \infty)$ for some $ \mathrm{cst}>0$ satisfies $ \mathbb{E}[ \exp(-  a \mathfrak{z}_{t})] < \infty$ for any $ a >0$, so the previous display indeed decreases faster than any exponential as $x \to -\infty$. This proves the desired claim. 
  \end{proof}

  \subsection{Scaling limit of locally largest exploration}
  We end this section by a technical result that we shall need in the proof of Theorem \ref{thm:scaling}. It formally links the scaling limit of the decoration-reproduction process under $ \mathbb{P}_{p}$ that is implicitly derived above to the formal decoration-reproduction process used to construct the $\gamma$-ssMt as in \cite{bertoin2024self}. This section can be skipped at first reading. \\ 
  
  We start by recalling the construction of the locally largest exploration decoration-reproduction process from  \cite[Example 3.6]{bertoin2024self} in the case $\beta=3/2$ and \cite[Example 3.9]{bertoin2024self} (and the remark following it)  in the case $\beta=5/2$. First, in the case $\beta = 3/2$, under (yet another) probability measure $ \mathrm{P}$, suppose that $\zeta$ is a L\'evy process starting from $0$, with no Brownian part, and being pure jump with only negative jump with L\'evy measure ${ \Lambda}_{\mathrm{Bro}}$ given by 
    \begin{eqnarray} \label{eq:GLMbrownianfrag} \int_{[-\log 2,0]} F\big(  \mathrm{e}^{{y}}\big)   \  { \Lambda}_{\mathrm{Bro}}(\mathrm{d} y) \quad  := \quad   \sqrt{\frac{2}{\pi}} \int_{1/2}^{1} F(x) \frac{ \mathrm{d}x}{(x(1-x))^{3/2}}.  \end{eqnarray}
    In particular we have  $  \mathrm{E}[ \mathrm{e}^{ z \zeta_{t}}] = \exp( t \psi(z))$ with 
    \begin{eqnarray} \label{eq:LK32} \psi(z) &=&  \int_{ \mathbb{R}^*}   \left( \mathrm{e}^{z y} -1 \right) \Lambda_{\mathrm{Bro}} (  \mathrm{d} y).  \end{eqnarray}
In the case $\beta = 5/2$, we suppose  that $\zeta$ is a L\'evy process starting from $0$, with no Brownian part,  L\'evy measure ${ \Lambda}_{\mathrm{BroGF}}$ given by 
  \begin{eqnarray*} \label{eq:GLMgrowthBrownian} \int_{[-\log 2,0]} F\big(  \mathrm{e}^{{y}}\big)  \  { \Lambda}_{\mathrm{BroGF}}(\mathrm{d} y)  \quad  := \quad    \frac{3}{ 4 \sqrt{\pi}} \int_{1/2}^{1} F\big(x\big) \frac{ \mathrm{d}x}{(x(1-x))^{5/2}},  \end{eqnarray*}
  and whose L\'evy--Khintchine exponent $\psi(z)$ specified by $  \mathrm{E}[ \mathrm{e}^{ z \zeta_{t}}] = \exp( t \psi(z))$ is given by
    \begin{eqnarray*} \psi(z) &=&  \frac{4(7-3\pi)}{3\sqrt{\pi}} z  + \int_{ \mathbb{R}^*}   \left( \mathrm{e}^{z y} -1-  z y \mathbbm{1}_{|y|\leq 1} \right) \Lambda_{\mathrm{BroGF}} (  \mathrm{d} y),  \end{eqnarray*}
  see \cite[Example 3.9]{bertoin2024self} from where the formula is taken. We then construct under $ \mathrm{P}$ the positive self-similar Markov process $(X_{t})_{0 \leq t \leq \tau_{ 0}}$ associated with $ \zeta$ via \eqref{eq:lamperti} and for the self-similarity exponent $\gamma = \beta-1  \in \{ 1/2, 3/2\}$. In particular, since the L\'evy measure of $\zeta$ is supported by $[-\log 2, 0]$, the process $X$ never makes negative jumps larger than half of its size.
  
  \begin{lemma}[Scaling limit of the locally largest exploration]  \label{lem:scalingLLbertoin}There exists a constant $ \mathfrak{c}>0$ such that for any $ \delta>0$, with the same notation as in Lemma \ref{lem:scalinglimit} we have
  $$\left( S^{(p)} , \eta^{(p)}\right)|_{[0, p^{-\gamma} \tau_{ \delta p}]} \mbox{ \ under } \mathbb{P}_{p} \quad \xrightarrow[p\to\infty]{(d)} \quad  \Big( (\mathfrak{c}\cdot X_{t})_{0 \leq t \leq \tau_{ \delta} } , \sum_{\begin{subarray}{c}0 \leq t \leq \tau_{ \delta} \\    \Delta X_{t}>0 \end{subarray}} \delta_{t, \mathfrak{c}\cdot \Delta X_{t}} \Big) \mbox{ \ under } \mathbf{P}_1,$$ 
  for the product topology where for the first coordinate we use the Skorokhod topology  and for the second one the vague convergence of measures on $ \mathbb{R}_{+} \times \mathbb{R}_{+}^{*}$. 
  \end{lemma}
  \begin{proof} The arguments yielding to \eqref{eq:keyformulalimit} can be reproduced mutatis mutandis to show the statement of the lemma but where the process $(X_{t})_{0 \leq \tau_{ \delta}}$ is replaced by the process $ \tilde{X}$ whose law is prescribed by  
  $$ \mathrm{E}[F( \tilde{X}_{t} : 0 \leq t \leq \tau_{ \delta})] =  \mathbf{E}_{1}\left[ F(\xi_{t} : 0 \leq t \leq \tau_{\delta}) \cdot (\xi_{\tau_{ \delta}}^{})^{-\beta} \cdot \mathbbm{1}\left\{ |\Delta \xi_s| <  \frac{\xi_{s-}}{2} : \forall s \leq \tau_{\delta} \right\}\right],$$
  where we recall that under  $\mathbf{E}_1$ the process $\xi$ is a spectrally negative $\gamma$-stable L\'evy process. Using \eqref{eq:lamperti} it is easy to see that the above display defines a (family of laws of) positive self-similar Markov process with index $\gamma$ whose associated L\'evy process $\zeta$ associated via the Lamperti transformation  is obtained  using a similar $h$-transformation.  Namely, with some obvious notation, for any $x <0$ we have
  \begin{eqnarray} \label{eq:charaklevy} \mathrm{E}[F( \zeta_{t} : 0 \leq t \leq \varsigma_{ \log x})] =  \mathbf{E}_{1}\left[ F(\zeta_{t} : 0 \leq t \leq \tau_{ x}) \cdot  \mathrm{exp}\big({-\beta \zeta_{\tau_{ \delta}}}\big) \cdot \mathbbm{1}_{\left\{ \Delta \zeta^{}_{s} \geq -\log 2  : \forall s \leq \tau_{ x} \right\}}\right].  \end{eqnarray}
    Using Lemma \ref{lem:martingale}, it follows that the law of $\zeta$ under $ \mathrm{P}$ it just given by the law of $\zeta$ under $ \mathbf{P}_{1}$ biaised by the martingale $ t \mapsto \mathrm{exp}\big({-\beta \zeta_{t}}\big) \cdot \mathbbm{1}_{\left\{ \Delta \zeta^{}_{s} \geq -\log 2  : \forall s \leq t \right\}}$. Those exponential martingale transformations of L\'evy process are classical (sometimes called Esscher transforms) and in particular yields to L\'evy processes, see \cite[Section 2.8]{kyprianou2022stable}. One then has to check that the law of this L\'evy process matches with the one described just before the lemma. This is done by computing the L\'evy-Khintchine exponent using the same arguments that yielded to \eqref{eq:calculpoisson} in the case $\beta=3/2$ or  to \eqref{eq:calculcompensation} in the case $\beta=5/2$: for $z \in \mathbb{R}$ we have 
  \begin{eqnarray*} \mathrm{E}[ \mathrm{e}^{z \zeta_{t}}] &\underset{ \eqref{eq:charaklevy}}{=}&  \mathbf{E}_{1}\left[  \mathrm{exp}\big((z-\beta) \zeta_{t}\big) \cdot \mathbbm{1}_{\left\{ \Delta \zeta^{}_{s} \geq -\log 2  : \forall s \leq t \right\}}\right].  \end{eqnarray*}
  In the case $\beta=3/2$ the RHS is equal to 
    \begin{eqnarray*}
 & \underset{ \mathrm{see\ } \eqref{eq:calculpoisson}}{=}&    \int_{0}^{1} {\frac{-1}{\Gamma(1-\beta)} \frac{ \mathrm{d}x}{(1-x)^{\beta}}} \left(x^{(z-\beta)} \mathbbm{1}_{x> \frac{1}{2}}-1\right) - \frac{1}{\Gamma(2- \beta)} \\ 
 &=&   \frac{-1}{ \Gamma(1-\beta)} \int_{1/2}^{1} \frac{\mathrm{d}x}{(x(1-x))^{3/2}} \big( x^{z}-1) + \underbrace{ \frac{-1}{ \Gamma(1-\beta)} \int_{0}^{1} \frac{\mathrm{d}x}{(x(1-x))^{3/2}} \big( \mathbbm{1}_{x >1/2}-x^{3/2})- \frac{1}{\Gamma(2- \beta)}}_{=0},  \end{eqnarray*}
 so that we recover up to a multiplicative factor $2$, the L\'evy--Khintchine formula in \eqref{eq:LK32}. This proves indeed that the law of $\tilde{X}$ is the same as that of a multiple of $X$ as desired. The case $\beta=5/2$ is similar, using the analog of \eqref{eq:calculcompensation}  and performing the same manipulations. We leave the details of the calculation to the fearless reader.  
\end{proof}

\section{Scaling limits of the trees and their volume}
In this section, we focus on the \textbf{volume} of the random tree $ \mathbf{t}$ under $ \mathbb{P}_p^x$ thus making the link between the catalytic variable $y$ and the size variable $x$. For this, we shall first make precise various notions of size, or equivalently of \textbf{pointed} structures. 
Let us begin with the most natural one, denote $ \mathrm{Vol}^{\bullet}  = \# \mathrm{t}$ for the number of vertices in the underlying tree $ \mathrm{t}$. We can then introduce the associated pointed partition functions 
$$ W_{p}^{\bullet, x} = \sum_{n \geq 1} \sum_{ \begin{subarray}{c}\mathbbm{t} \in \mathrm{FPT}_{p}^{n} \\ v \in \mathrm{t} \end{subarray}}  \mathrm{w}( \mathbbm{t}) \cdot x^{n} = \sum_{n \geq 1} \sum_{ \mathbbm{t} \in \mathrm{FPT}_{p}^{n}} \mathrm{w}( \mathbbm{t}) \cdot n \cdot x^{n} =  [y^{p}] x\frac{ \partial}{ \partial x} F(x,y).$$
In particular, we have $ W_{p}^{ \bullet,x} = W_{p}^{x}  \cdot \mathbb{E}^{x}_{p}[ \mathrm{Vol}^{\bullet}]$ for all $ x \in (0, x_{\crit}]$ and all $p \geq 0$. We shall also consider the following variation on the notion of size which is slightly more practical to approach from a probabilistic point of view (both notions will be related in Lemma \ref{lem:volcompa}). 
Recall from Assumption $(*)$ the value $ \K \geq 1$ which in particular bounds the maximal number of parking spots on a given edge. If $  \mathbf{t} = (\mathrm{t}, \phi)$ is a labeled tree, we denote by $ \mathbf{t}^{\circ}$ the set of its vertices $v$ of label $\K$ (we chose this particular value for commodity). In particular, we may very well have $ \mathbf{t}^{\circ}=\varnothing$. We then write $ \mathrm{Vol}^{\circ} = \# \mathbf{t}^{\circ}$ for the number of such vertices. Note that trivially $ \# \mathbf{t}^{\circ} = \mathrm{Vol}^ \circ \leq \mathrm{Vol}^ \bullet = \# \mathrm{t}$. 
We can then consider the pointed partition function obtained by putting for $x>0$ 
  \begin{eqnarray} \label{def:Wcircp}W_{p}^{\circ, x} = \sum_{n \geq 1} \sum_{\mathbbm{t} \in \mathrm{FPT}_{p}^{n}} \#\mathbf{t}^\circ \cdot \mathrm{w}( \mathbbm{t}) \cdot x^{n}.  \end{eqnarray} 
Even if we may have $ \mathbf{t}^{\circ}=\varnothing$ for a specific $ \mathbf{t}$, we have $ W_p^{\circ,x} >0$ as soon as $x>0$ by the connectivity assumption  in $(*)$.  Although $W_{p}^{\circ, x}$ cannot simply be obtained by differentiating $F(x,y)$ with respect to $x$, we similarly have $ W_{p}^{ \circ,x} = W_{p}^{x} \cdot \mathbb{E}_{p}^{x}[ \mathrm{Vol}^{\circ}]$. We shall also use the same notation as in \eqref{eq:polyalways} and put   \begin{eqnarray} \label{eq:polyalwaysbis} \widetilde{W}_{p}^{\circ,x} = {W}_{p}^{\circ,x} \cdot (y_{\crit}^{x})^{p} \quad  \mbox{and} \quad \widetilde{W}_{p}^{\bullet,x} = {W}_{p}^{\bullet,x} \cdot (y_{\crit}^{x})^{p}.  \end{eqnarray}
When $x < x_{\crit}$, since $[y^{p}]F(\cdot,y)$ is analytic around $x$ by Lemma \ref{lem:xcyc}, the functions $W_{p}^{\bullet, x}$ are well-defined and so is $W_{p}^{\circ, x}$  since $W_{p}^{ \circ,x} \leq W_{p}^{ \bullet,x}$. However, at this point it is not clear whether those functions are finite for $x=x_{\crit}$.

We shall prove below that $ p \mapsto  \widetilde{W}_p^{\circ,x}$ is ``almost'' a harmonic function for the $\nu$-random walk killed when dropping below $\K$. This will enable us to understand quite precisely the behaviors of those functions in the cases $\beta \in \{3/2, 5/2\}$ as identified above. For this, we recall some background on fluctuation theory for random walk in Section \ref{sec:fluctuations}.  Before that, we present a rough comparison of volumes which only relies on aperiodicity and the multitype branching structure. As in the preceding sections the value  $ x \in (0, x_{\crit}]$ is implicit in the notation for simplicity. 

 \subsection{Comparison of volumes}
This section contains a lemma enabling us  transfer back to $ \mathrm{Vol}^\bullet$ and $ \widetilde{W}^\bullet_p$ the knowledge we will acquire on $ \mathrm{Vol}^\circ$ and $ \widetilde{W}^\circ_p$ via random walk estimates. The take-away message being that $\mathrm{Vol}^\bullet$ and $ \mathrm{Vol}^{\circ}$ are roughly proportional when large.
 
 \begin{lemma}[$\mathrm{Vol}^\bullet \preceq \mathrm{Vol}^\circ$] \label{lem:volcompa}There exists a constant $ \delta >0$ such that for $n\geq 1$ large  enough we have
 $$\forall p \geq 0, \qquad  \mathbb{P}_{p}( \mathrm{Vol}^{\circ} \geq n) \geq \mathbb{P}_{p}( \mathrm{Vol}^{\bullet}\geq \delta n)- \mathrm{e}^{-\delta n}.$$
 \end{lemma}
We deduce from this lemma that the moments of $ \mathrm{Vol}^{\circ}$ and $ \mathrm{Vol}^{\bullet}$ are comparable up to constant, in particular they are simultaneously finite or infinite.
 
 \begin{proof} We imagine that under $ \mathbb{P}_p$ we explore the underlying labeled tree $ \mathbf{t}$ vertex-by-vertex, say in the depth first order. Since the tree is a multitype Bienaym\'e--Galton--Watson tree, this can be encoded in a Markov process $ ( \mathbb{X}_k : k \geq 0)$ with canonical filtration $ \mathcal{F}_k$ where the current state is given by the finite list $\in \mathbb{Z}_{\geq 0}^{( \mathbb{Z}_\geq 0)}$ of the labels of the vertices yet to be explored. The starting state is then $ \mathbb{X}_0 = \{p\}$ under $ \mathbb{P}_p$ and when exploring a vertex, we remove the first element of the list $ \mathbb{X}_k$ and replace it with the list of the labels of its children (possibly empty) to get $ \mathbb{X}_{k+1}$. The process stops at the first time the list $ \mathbb{X}_k$ becomes empty and since we are exploring one vertex at a time, this happens precisely at time $k= \mathrm{Vol}^\bullet$. We then claim that under our standing assumptions $(*)$, there exists $c>0$ and an integer $N$ such that for any time $k \geq 0$,
  \begin{eqnarray}\mathbb{P}_p\left( \begin{array}{c}\mbox{discovering a vertex of label }\K\\ \mbox{ in the time intervalle } [k,k+N) \end{array}  \Big| \mathcal{F}_k \ \& \  \mathrm{Vol}^\bullet > k\right) \geq c.   \label{eq:filtration}\end{eqnarray}
This is easy to see if the vertex to be explored is of large enough label, since then by Proposition~\ref{prop:nu_proba}, there is a probability bounded away from $0$ that the the left-most child of this vertex has label precisely $\K$ (that is we take $N=1$). Otherwise, if the label of the vertex to explore is small, we claim that by aperiodicity consideration, there is a probability bounded away from $0$ that when exploring in depth-first search we reach a vertex with large enough label within $  N$ steps so that the preceding reasoning applies. 

To deduce the lemma and overcome the problematic correlation between $ \mathrm{Vol}^ \bullet$ and the discovery of vertices of label $\K$, we decompose the exploration into time intervals $[kN,(k+1)N)$ of length $N$, even after the absorption time, and for each $k \geq 0$, if $kN \leq  \mathrm{Vol}^ \bullet$ we denote by $ \epsilon_{k} = 1$ if we discover a vertex of label $\K$ in this interval, and $0$ otherwise. When $  \mathrm{Vol}^ \bullet  < kN$, we complete this sequence by independent Bernoulli variables with parameter $c$. Using \eqref{eq:filtration}, it is easy to see that the sequence $ ( \epsilon_{k})_{k \geq 0}$ constructed this way always dominates a sequence of independent Bernoulli variables with parameter $c$. In particular, by large deviations estimates we have 
$$ \mathbb{P}_{p}\left( \sum_{i=0}^{n} \epsilon_{i} \leq \frac{c}{2}n\right) \leq \mathrm{e}^{-\delta n},$$ for some $\delta>0$ (independently of $p$). We can thus write
  \begin{eqnarray*} \mathbb{P}_{p}( \mathrm{Vol}^{\circ} \geq \frac{c}{2}n) &\geq& \mathbb{P}_{p}(  \mathrm{Vol}^ \bullet \geq n N) - \mathbb{P}_{p}\left( \sum_{i=0}^{n}\epsilon_{i} \leq \frac{c}{2}n\right)\\
  &\geq& \mathbb{P}_{p}(  \mathrm{Vol}^ \bullet \geq nN) - \mathrm{e}^{-\delta n}  \end{eqnarray*} as announced.
   \end{proof}

\subsection{Fluctuations theory for random walks} \label{sec:fluctuations}
In this section we gather a few (classical and more recent) estimates for random walk with i.i.d.~increments of law $\nu$. In particular recall that $\nu$ has support bounded from above, satisfies $\nu(-j) \approx j^{-\beta}$ for $\beta=3/2$ or $\beta=5/2$ and is centered in the last case. We saw in Section \ref{sec:levybeta} that  the step distribution $\nu$ is in the strict domain of attraction of  the $(\beta-1)$-spectrally negative stable law. If as above $(\xi_{t})_{t \geq 0}$ is the associated L\'evy process, its positivy parameter is defined by
$$ \varrho = \mathbf{P}_0( \xi_{1}>0) = \left\{ \begin{array}{cl} 0 &\mbox{ if } \beta = 3/2,\\   \frac{2}{3} & \mbox{ if }\beta=5/2.\end{array} \right.$$

\begin{figure}[!h]
 \begin{center}
 \includegraphics[width=6cm]{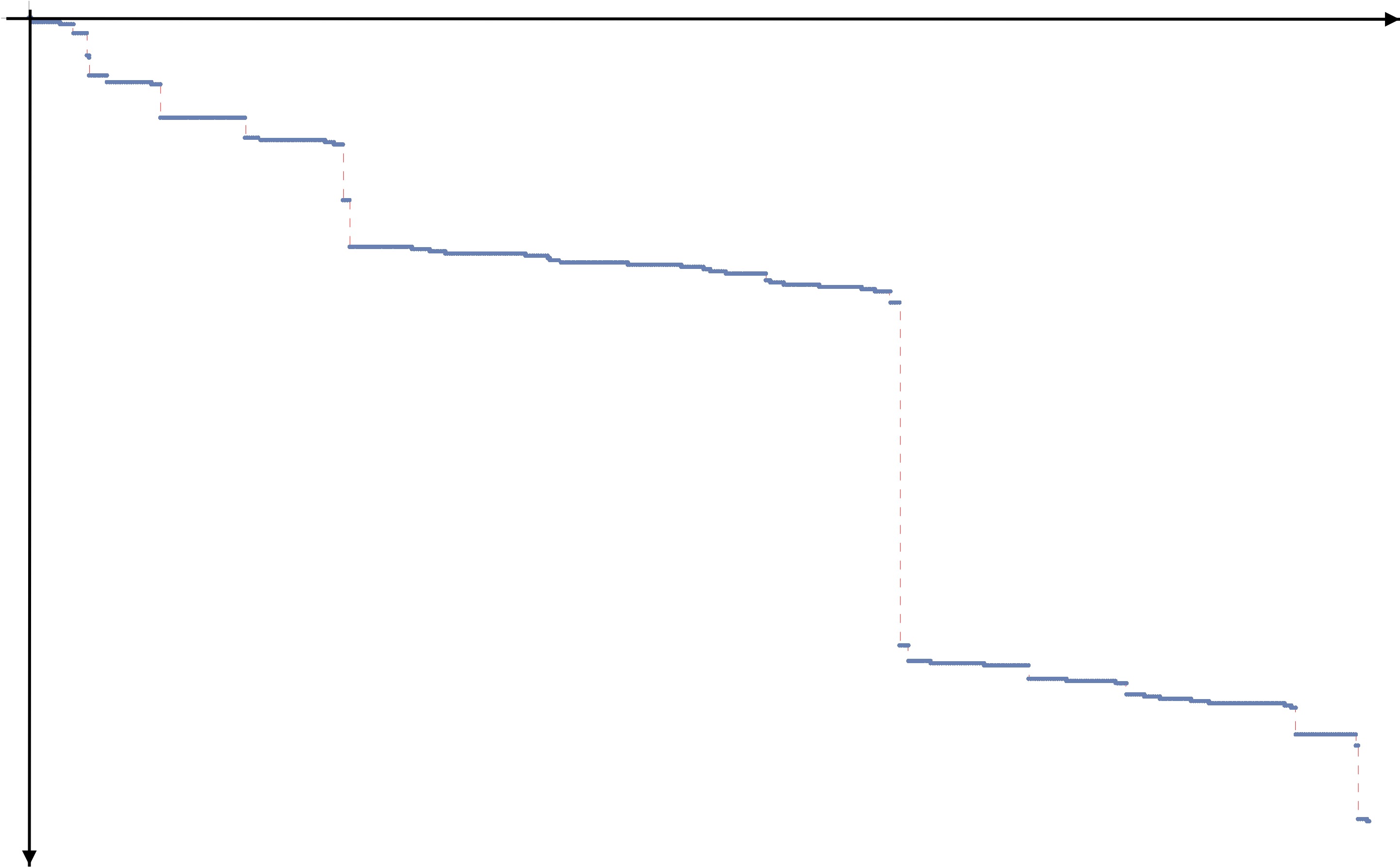} \hspace{2cm} \includegraphics[width=6cm]{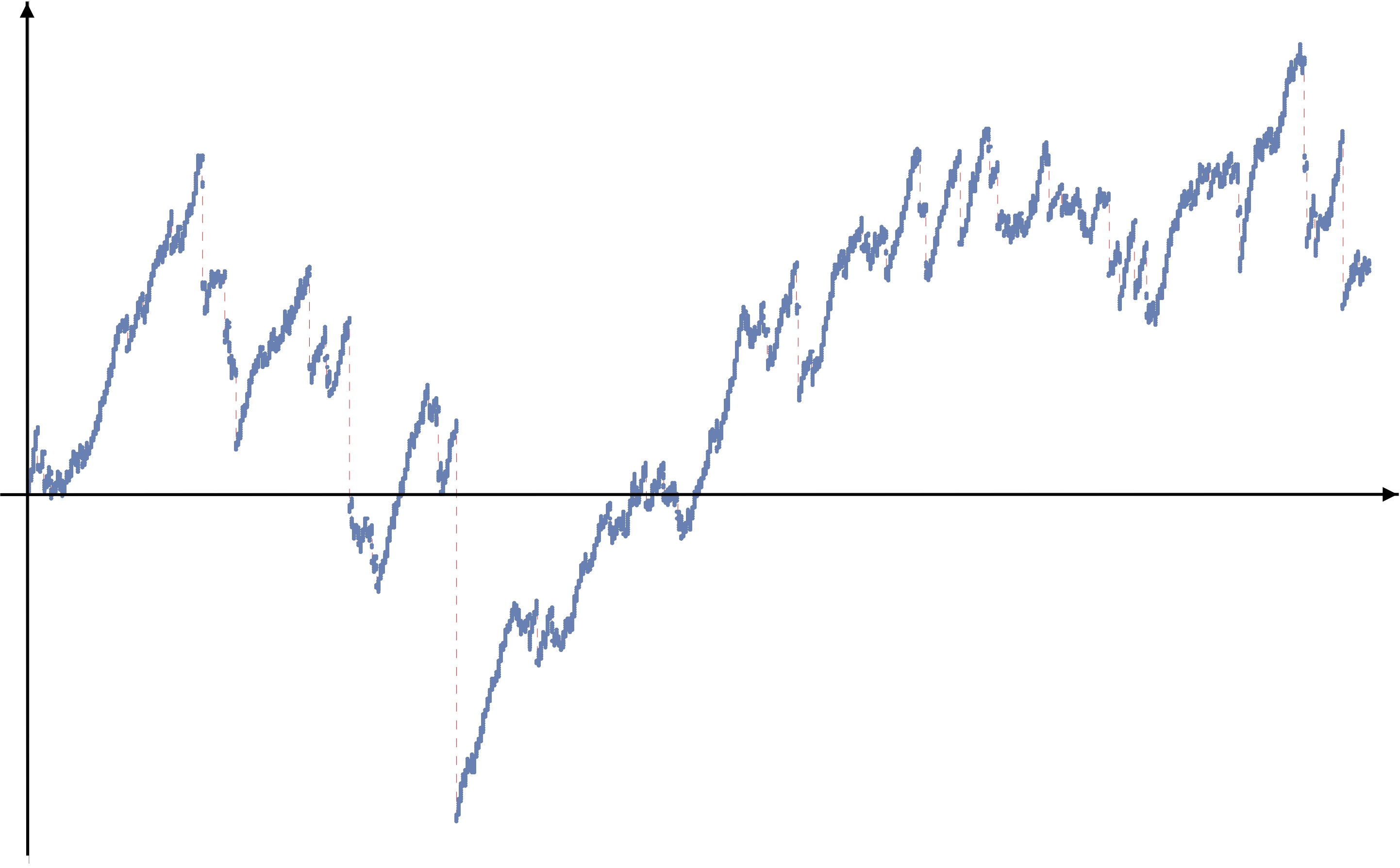}
 \caption{A simulation of the $\beta-1$-stable L\'evy processes, Left $\beta= 3/2$, Right $\beta=5/2$.}
 \end{center}
 \end{figure}

Recall that under $ \mathbb{P}^{\RW}_{0}$ the process $(S)$ is a $\nu$-random walk started from $0$ and consider the decreasing ladder heights $0=H_{0}> H_{1}> H_{2} > \cdots$ and epochs $0=T_{0}< T_{1}< \cdots$ obtained by setting $H_{0}=T_{0}=0$ and for $i \geq 1$,
$$ T_{i}= \inf \{ k > T_{i-1} : S_{k}< H_{i-1}\}, \quad \mbox{ and } \quad H_{i} = S_{T_{i}}.$$
In particular, in our notation we have $T_{1} = \tau_{-1}$. Since $\liminf_{k \to \infty} S_{k} = -\infty$,  by the strong Markov property the variables $(H_{i+1}-H_{i} , T_{i+1}-T_{i})_{ i \geq 0}$ are i.i.d. These ladder variables are used to defined two functions, the \textbf{pre-renewal} function 
 \begin{eqnarray} \label{def:hpre} \mbox{ for } p\geq 0, \quad h_{ \mathrm{pre}}(p) := \mathbb{P}_{0}^{\RW}( \exists i\geq 0 : H_{i} = -p) = \mathbb{P}_{0}^{\RW}( S_{\tau_{-p}} = - p),  \end{eqnarray} in particular $h_{ \mathrm{pre}}(0)=1$ and $ h_{ \mathrm{pre}}(-p) =0$ for $p \geq 1$; 
and its primitive the \textbf{renewal} function $$H_{ \mathrm{ren}}(p) = h_{ \mathrm{pre}}(0) + ... + h_{ \mathrm{pre}}(p) =  \mathbb{E}_{0}^{\RW}\big[ \# \{\mbox{minimal records  } \in [-p,0] \}\big].$$ By classical results that go back to at least Doney \& Greenwood \cite{doney1993joint}, it is known that  $H_{1}$ is in the domain of attraction of the  $ ( \beta-1) (1- \varrho) = 1/2$ stable  spectrally negative law (notice that $( \beta-1) (1- \varrho) = 1/2$ regardless of the value of $\beta \in \{3/2, 5/2\}$), more precisely:
\begin{proposition} In both cases $\beta \in \{3/2, 5/2\}$ we have 
  \begin{eqnarray} \label{eq:CarChau} \mathbb{P}_0^{\RW}(H_{1}< -x) \approx x^{-1/2}, \mbox{ as } x \to \infty.  \end{eqnarray} 
  In particular, the pre-renewal function satisfies 
  \begin{eqnarray} \label{eq:hpre} h_{ \mathrm{pre}}(p) \approx  p^{-1/2}, \mbox{ as } p \to \infty.  \end{eqnarray} 
Furthermore when $\beta = \frac{5}{2}$ we furthermore have
 \begin{eqnarray} \quad \label{eq:VatVach} \mathbb{P}_0^{\RW}(T_{1}> n) \approx n^{-1/3}
 \mbox{ as } n \to \infty.  \end{eqnarray}
 \end{proposition}
 \begin{proof} Let us first suppose  that we are in the finite mean case $\beta=5/2$. Then the estimate on $ \mathbb{P}_{0}^{\RW}(T_{1}>n)$ follows directly from Doney \cite[Theorem 2]{Do82} which restricts to the integrable spectrally one-sided case and where there are no slowly varying function in our case. 
 Still in the case $\beta=5/2$, the estimate on $\mathbb{P}_0^{\RW}(H_{1}< -x)$ almost follows directly from  Caravenna \& Chaumont \cite{CC08}:   by \cite[Lemma 2.1]{CC08} the slowly varying function that may appear in the asymptotic of $\mathbb{P}_0^{\RW}(H_{1}< -x)$ is the inverse of that appearing in the asymptotic of $H_{ \mathrm{ren}}(p)$ as $p \to \infty$. In turn, by \cite[Lemma 2.1]{CC08}, the slowly varying function appearing in the asymptotic of $H_{ \mathrm{ren}}(p)$  is linked to the one present in $ \mathbb{P}_{0}^{\RW}(T_{1}>n)$ as $n \to \infty$. But by the preceding result, those functions are eventually constant.\\
   To deduce the tail estimate on $ h_{ \mathrm{pre}}$ from the estimate on the tail of $H_{1}$ we use renewal arguments: Introduce $ \mathcal{R} = \{H_{i}  : i \geq 0\}$ our task is equivalent to proving that 
  \begin{eqnarray} \label{eq:goalrenewalstrong} h_{ \mathrm{pre}}(p) = \mathbb{P}_p^{ \RW} \Big( S_{\tau_{0}} = 0 \Big) = \mathbb{P}_0^{\RW}(- p\in \mathcal{R}) \quad \underset{p \to \infty}{\approx} \quad  \frac{ 1}{ \sqrt{p}}. \end{eqnarray} In this setup, it is only recently that necessary and sufficient conditions on the law of the inter-arrivals  have been found so that  the strong renewal theorem holds: by the wonderful \cite[Theorem 1.5]{caravenna2019local} whose assumptions are granted in our context, we deduce \eqref{eq:goalrenewalstrong} thus completing the proof of the lemma in the case $\beta=5/2$.
  In the case $\beta = \frac{3}{2}$, the walk is transient and makes the analysis much simpler. In particular, we shall see below that $ \mathbb{P}_{0}^{\RW}(T_{1}>n)$ decreases stretched exponentially fast.  To get the estimate on $\mathbb{P}_{0}^{\RW}(H_{1}<-x)$, we still use \cite[Lemma 2.1]{CC08} which relates it to the tail of  $H_{ \mathrm{ren}}$. However, in the transient setup it is easy to see that 
  $$ H_{  \mathrm{ren}}(p) \approx \mathbb{E}_{0}^{\RW} \left[ \sum_{i=0}^{\infty} \mathbbm{1}_{-p \leq S_{i} \leq 0}\right],$$ and the RHS is easily seen to by $\approx \sqrt{p}$ be standard renewal arguments, see e.g.  \cite[Eq.(1.9)]{caravenna2019local}.   \end{proof}

\subsection{Pointed partition functions and pre-renewal functions}

Our goal is now to compare  $  \widetilde{W}^\circ_p$ with the pre-renewal function $ h_{ \mathrm{pre}}$  and deduce:

\begin{lemma} \label{lem:wtildepoint} Regardless of $\beta \in \{3/2, 5/2\}$, the pointed partition functions $\widetilde{W}_{p}^{\circ}$ are finite and  we have 
\begin{eqnarray*} \widetilde{W}_{p}^{\circ} \ \ \threesim \ \  p^{-1/2}, \quad \mbox{ as } p \to \infty,  \end{eqnarray*} where we recall that $a_{n}\threesim b_{n}$ means that $(a_{n}/b_{n})$ is bounded away from $0$ and $\infty$ as $n \to \infty$.
\end{lemma}

\begin{proof}
To gain intuition, let us first suppose that  $ W^{\circ}_{p} < \infty$ for all $p \geq 0$ and let us write Tutte's equation for them. By definition, see \eqref{def:Wcircp},  for $p \geq  0$  we have
\begin{eqnarray*} W_{p}^ \circ &= & W_{\K} \mathbbm{1}_{p={\K}} + x\sum_{k \geq 0} \sum_{i = 1}^{k}\sum_{ c \geq 0} \sum_{s_1, \ldots, s_k} w_{k,c, (s_1, \ldots, s_k)} \sum_{\substack{(p_1, \ldots, p_k)\\ \forall j, p_j \geq s_j \\ \sum_{j=1}^{k} (p_j - s_j) + c = p}} W_{p_i}^ \circ \prod_{j \neq i} W_{p_j} \\
&\underset{ \mathrm{exch.}}{=} & W_{\K} \mathbbm{1}_{p=\K} + x \sum_{k \geq 0} k \sum_{ c \geq 0} \sum_{s_1, \ldots, s_k} w_{k,c, (s_1, \ldots, s_k)} \sum_{\substack{(p_1, \ldots, p_k)\\ \forall j, p_j \geq s_j \\ \sum_{j=1}^{k} (p_j - s_j) + c = p}} W_{p_1}^ \circ \prod_{i =2}^k W_{p_i}
\end{eqnarray*}
As in Definition \ref{def:nuhat}, let us write $p_{1} = p+q$ and shift $k$ by one unit. Recognizing the appearance of $\hat{\nu}$ and denoting $\hat{\hat{\nu}}(q ; s_0) \equiv \hat{\nu}(q ; *, (s_0, *), (*))$  its push-forward on the variable $q$ and $s_{0}$ we can write using the same notation $ \mathscr{C}$ for the parking constraints as in Definition \ref{def:nuhat}
 \begin{eqnarray}
\widetilde{W}_{p}^ \circ & =&  \widetilde{W}_{\K} \mathbbm{1}_{p={\K}} + \sum_{q \geq -p } W^{ \circ}_{p+q} y_{ \crit}^{p+q} \cdot x y_{\crit}^{-q} \sum_{k \geq 0} (k+1)  \sum_{ c \geq 0} \sum_{s_0, \ldots, s_k} \sum_{\substack{(p_1, \ldots, p_k)\\ \forall i, p_i \geq s_i }} w_{k,c, (s_{0}, \ldots, s_k)} \prod_{i=1}^k W_{p_i}  \mathbbm{1}_{ p+q \geq s_{0}} \mathbbm{1}_{  \mathscr{C}}\nonumber \\
&=& \widetilde{W}_{\K} \mathbbm{1}_{p={\K}} + \sum_{q \geq -p } \widetilde{W}^{ \circ}_{p+q} \sum_{k \geq 0}  \sum_{ c \geq 0} \sum_{s_0, \ldots, s_k} \sum_{p_1, \ldots, p_k}  \hat{\nu}\big( q; c, (s_{0}, ... , s_{k}) , (p_{1}, ... , p_{k})\big) \mathbbm{1}_{p+q \geq s_{0}} \nonumber \\
&=& \widetilde{W}_{\K} \mathbbm{1}_{p=\K} + \sum_{q, s_{0} } \widetilde{W}^{ \circ}_{p+q} \cdot \hat{\hat{\nu}}(q; s_{0}) \mathbbm{1}_{p+q \geq s_{0}}.\label{eq:htransfos1} 
 \end{eqnarray}
When $ p+q \geq \K$ the inequality $p+q \geq s_{0}$ is automatically satisfied. In particular, we have 
 \begin{eqnarray}
 \sum_{q \geq \K -p } \widetilde{W}^{ \circ}_{p+q}\cdot \nu (q) \leq \widetilde{W}_{p}^ \circ \leq \widetilde{W}_{\K} \mathbbm{1}_{p={\K}} + \sum_{q \geq -p } \widetilde{W}^{ \circ}_{p+q}\cdot \nu (q).\label{eq:htransfo2}
 \end{eqnarray}
 In view of these inequalities, we shall compare $\widetilde{W}^{ \circ}_{p}$ with the functions $$ {h}_{ \circ}(p) := h_{ \mathrm{pre}}( p- \K) = \mathbb{P}_{p}^\RW( S_{\tau_\K} = \K) =  \mathbb{E}_p^{ \RW} \left[ \sum_{i = 0}^{\tau_ {  \K}} \mathbbm{1}_{ S_i = \K} \right ] ,$$ and the larger function $h^{\circ}$ defined by 
 $$ h^{\circ}(p) :=  \mathbb{E}_p^{ \RW} \left[ \sum_{i = 0}^{\tau_ {-1}} \mathbbm{1}_{ S_i = \K} \right ].$$ 
The function ${h}_{ \circ}$ is equal to $1$ at $ \K$ and null below, whereas $h^\circ \geq  h_\circ$ is null on $(-\infty,0)$. 
More precisely we will show  that   \begin{eqnarray}
\label{eq:goalsandwhich}\widetilde{W}_{\K} {h}_{ \circ} (p) \leq \widetilde{W}_p^ \circ  \leq \widetilde{W}_{\K} h^ \circ (p).  \end{eqnarray} 
To prove it,  for $ m \geq 0$,  let us introduce the notion of volume cut at height $m$, i.e.\ 
  \begin{eqnarray*}W_{p, m}^{\circ} := \sum_{n \geq 1} \sum_{ \begin{subarray}{c}\mathbbm{t} \in \mathrm{FPT}_{p}^{n} \\ v \in \mathbf{t}^{\circ} \\ \mathrm{height} (v) \leq m \end{subarray}}  \mathrm{w}( \mathbbm{t}) \cdot x^{n}. \end{eqnarray*}
so that for all $p,m \geq 0$, the function $\widetilde{W}_{p,m}^ \circ$ is finite and converges towards $\widetilde{W}_{p}^ \circ$ as $ m$ goes to $\infty$.  Similarly, we define the analogue for the functions ${h}_{ \circ}$ and $ h^ \circ$ by
  \begin{eqnarray*}h_{\circ} (p,m) := \mathbb{E}_p^{ \RW} \left[ \sum_{i = 0}^{m \wedge \tau_ {  \K}} \mathbbm{1}_{ S_i = \K} \right ], \quad \mbox{and} \quad h^ \circ (p,m) := \mathbb{E}_p^{ \RW} \left[ \sum_{i = 0}^{m \wedge \tau_ {-1}} \mathbbm{1}_{ S_i = \K} \right ], \end{eqnarray*}
so that for all $ p \geq 0$, by monotone convergence we have $h^\circ(p,m) \to h^\circ(p)$ as $m \to \infty$ and similarly for $h_\circ(p,m)$.
For $m=0$, we have  $$h_{\circ} (p,0) =  h^ \circ(p,0) = \widetilde{W}_{\K} \mathbbm{1}_{p = \K} = W_{p, 0}^{\circ},$$
and the same computation as above shows that for all $m \geq 0$, for all $ p > \K$, 
 \begin{eqnarray*} \sum_{q \geq \K-p } \widetilde{W}^{ \circ}_{p+q, m}\cdot \nu (q) \leq  \widetilde{W}_{p, m+1}^ \circ \leq   \widetilde{W}_{\K} \mathbbm{1}_{p=\K}+ \sum_{q \geq -p } \widetilde{W}^{ \circ}_{p+q, m}\cdot \nu (q), 
\end{eqnarray*}
where the inequality on the right-hand side is even valid for all $ p \in \mathbb{Z}$.
On the other hand, by decomposing the random walk according to its first step, for all $ m \geq 0$ and all $p \in \mathbb{Z}$, we have
\begin{eqnarray*}&\mbox{for all } p \geq 0, \quad h^ \circ (p,m+1) &=\mathbbm{1}_{p={\K}} + \sum_{ q \geq -p} h^ \circ(p+q,m) \nu (q) \quad \mbox{and}\\
&\mbox{for all } p > \K, \quad h_{\circ} (p,m+1) &=\sum_{ q \geq \K - p} h_\circ(p+q,m) \nu (q).
\end{eqnarray*} Thus, a straightforward recursion on $m$ shows that for all $m \geq 0$ and all $ p \in \mathbb{Z}$ we have $ \widetilde{W}_{\K} h_ \circ (p,m) \leq \widetilde{W}_{p,m}^ \circ \leq \widetilde{W}_{\K} h^ \circ (p,m), $ which by taking the limit as $m$ goes to infinity, proves  \eqref{eq:goalsandwhich}. We now use \eqref{eq:hpre} to conclude: On the one hand we have $h_\circ(p) \approx p^{-1/2} $ and on the other hand, denoting $\tau_{\{ \K\}} = \inf\{ i \geq 0: S_{i} = \K\}$ and $\tau_{\{ \K\}}^{+}= \inf\{ i \geq 1: S_{i} = \K\}$, by the Markov property, we have 
\begin{eqnarray*} 
h^ \circ (p) &=& \mathbb{E}_p^{\RW} \left[ \sum_{i = 0}^{ \tau_{ -1}} \mathbbm{1}_{S_i = \K} \right ]
=  \mathbb{P}_p^{\RW} \left( \tau_{ \{\K\}} < \tau_{  -1}  \right )\cdot  \underbrace{\mathbb{E}_{\K}^{\RW} \left[ \sum_{i = 0}^{ \tau_{ -1}} \mathbbm{1}_{S_i = \K} \right ]}_{  \mbox{finite  by aperiodicity}}. 
\end{eqnarray*}
Also we have $ h_{ \mathrm{pre}}(p) = \mathbb{P}_{p}^{\RW}(\tau_{\{0\}} = \tau_{0}) \geq \mathbb{P}^{\RW}_{p}(\tau_{ \{\K\}} < \tau_{  -1}) \cdot \mathrm{h}_{ \mathrm{pre}}(\K)$, and $ \mathrm{h}_{ \mathrm{pre}}(\K) \in (0, \infty)$ by aperiodicity again. We deduce that  $h^ \circ (p)$ is bounded above by a constant times $h_{ \mathrm{pre}}(p)$. Using \eqref{eq:hpre}, this proves the lemma.
\end{proof}

The reader may find surprising that the polynomial exponent of the tail of $ \widetilde{W}^\circ_p$ does not depend on the value of $\beta \in \{3/2, 5/2\}$ (a similar phenomenon has already  been observed  and used in the realm of planar maps by Timothy Budd \cite{Bud15,BuddPeelingLectureNotes}). However, this makes a great difference once this information is turned into probabilistic estimate: Once we know that $W_{p}^{\circ}$ is finite (and non-zero by our connectivity assumption) one can normalize the measure $ \mathrm{w}(  \mathrm{d}\mathfrak{t})$ to define a law $ \mathbb{P}_{p}^{\circ}$ on labeled trees starting with value $p\geq 0$ and given with a distinguished pointed vertex $v^{\circ} \in \mathbf{t}^{\circ}$ of label $\K$. Formally we put for any labeled tree $(t, \phi)$ starting with root label $p$ and with a vertex $v^{\circ} \in \mathbf{t}^{\circ}$ of label $\K$:
\begin{equation}\label{eq:deflawp} \mathbb{P}_{p}^{\circ}(( \mathrm{t}, \phi,v^{\circ})) = \frac{\widetilde{W}_{p}}{\widetilde{W}_{p}^{\circ}} \cdot \mathbb{P}_{p}(( \mathrm{t}, \phi)),\end{equation} so that in particular for any positive function $f$   \begin{eqnarray} \label{eq:biasing} \mathbb{E}^\circ_p\big[ f(\mathrm{Vol}^\circ)\big] = \frac{ \widetilde{W}_p}{\widetilde{W}_p^\circ} \mathbb{E}_p\left[ f(\mathrm{Vol}^\circ) \cdot \mathrm{Vol}^\circ\right].  \end{eqnarray}
When $f\equiv 1$ we deduce that 
  $$ \mathbb{E}_p[ \mathrm{Vol}^\circ] = \frac{ \widetilde{W}_p}{\widetilde{W}_p^\circ} \ \underset{\substack{ \mathrm{Eq.}\ \eqref{eq:tailnu}\\ \&\ \mathrm{Lem.}\ \ref{lem:wtildepoint}}}{\threesim} \ p^{\beta-1/2}, \quad \mbox{ as }p \to \infty.$$

\subsection{Pointed exploration}
 If $ ( \mathbf{t}, v^{\circ})$ is a labeled tree with $ v^{\circ} \in \mathbf{t}^{\circ}$, instead of using the locally largest exploration, one can use the branch towards $v^{\circ}$ to define a decoration-reproduction process:  we denote by $\varnothing = v_0 \preceq v_1 \preceq v_2 \preceq ... \preceq v_\tau = v^{\circ}$ the branch towards $v^{\circ}$. In particular, $\tau \geq 0$ is the length of the distinguished branch and the canonical decoration-reproduction process is defined analogously to what was done for the locally largest exploration: we write $S_i = \phi(v_i)$ for $0 \leq i \leq \tau$ for the labels along the distinguished branch and $Y^{i}_{j}$ for $1 \leq j \leq k_{v_{i-1}}(  \mathrm{t})-1$ for the labels of the siblings of $v_{i}$, i.e.~those vertices $v\ne v_{i}$ having $v_{i-1}$ for parent, ranked from left-to-right. Again,  this collection may be empty and has always fewer than $\K-1$ numbers by our standing assumption and we still have \eqref{eq:sommedelta}. We can thus consider, under the probability measure $ \mathbb{P}^{\circ}_{p}$ (with expectation $ \mathbb{E}^{\circ}_{p}$) the decoration process $S= (S_{i})_{0 \leq i \leq \tau}$ and the reproduction process 
 $ \eta = \sum_{i \geq 1} \sum_{j \geq 1} \delta_{i , Y^{{i}}_{j}}$. Again, by the Markov property of the pointed measure, it should be clear that the labeled trees dangling from the distinguished branch are, conditionally on their labels, independent and of law $ \mathbb{P}_{Y^{i}_{j}}$. We also have the exact analog of Proposition \ref{prop:llrw} (without requiring a locally largest condition this time):

\begin{proposition}[Pointed key-formula] \label{prop:llprw}For any positive functional $f$ we have 
 \begin{eqnarray*} \mathbb{E}^{\circ}_{p}\left[ f\Big(S|_{[0,t]}, \eta|_{[0,t]}\Big) \mathbbm{1}_{\tau \geq t} \right] & \leq & \mathbb{E}_{p}^{ \RW}\left[ f\Big((S|_{[0,t]}, \eta|_{[0,t]}  \Big)    \cdot    \frac{ \widetilde{W}_{S_t}^ \circ }{\widetilde{W}_{S_0}^ \circ }    \right],  \end{eqnarray*}
 furthermore the inequality is an equality if we further add $ \mathbbm{1}_{ S|_{[0,t]} > \K}$  in both expectations.
 \end{proposition}

 \begin{proof} As for Proposition \ref{prop:llrw}, by the Markov property, it suffices to show the proposition for $t=1$. 
 Let us fix $ p \geq 0$, $ q  \in \mathbb{Z}$ as well as  $ p_1, \dots, p_k \geq 0$ 
 and let us compute $ \mathbb{P}_p^ \circ ( S_1 = p +q  \mbox{ and } Y^{1}_{j} = p_j, \ j \geq 1)$.  Recall that the increments are exchangeable  so that we can exchange the vertex from the branch towards the pointed vertex label with the first one and keep the same local weight. Following the same steps that yielded to \eqref{eq:htransfos1} and using the same notation as in Definition \ref{def:nuhat} it gives
 \begin{eqnarray*} \hspace{-0.5cm} &&\mathbb{P}_p^ \circ ( S_1 = p+q \mbox{ and } Y^{1}_{j} = p_j, \ j \geq 1) \\
 &=&  x (k+1) \sum_{ c \geq 0} \sum_{\substack{s_0, \ldots, s_k\\  p_j \geq s_j, \forall j \geq 1}} w_{k+1,c, (s_0, \ldots, s_k)}  \frac{W_{p +q}^ \circ}{W_p^ \circ}  \mathbbm{1}_{p+q \geq s_{0}} \prod_{i =1}^k W_{p_i} \mathbbm{1}_{ \mathscr{C}}\nonumber \\
&=& \frac{W_{p +q}^ \circ}{W_p^ \circ} y_{\crit}^{q}\cdot  \sum_{c, s_{0}, ... , s_{k}}\hat{\nu} \big(q ; c, (s_{0},s_{1}, ... , s_{k}), (p_1, p_2, ... , p_k)\big) \mathbbm{1}_{p+q \geq s_{0}} \\
&\leq& 
   \frac{ \widetilde{W}_{p+q}^ \circ }{\widetilde{W}_p^ \circ }  \cdot  \mathbb{P}_p^\RW ( S_1 = p+q \mbox{ and } Y^{1}_{j} = p_j, \ j \geq 1).   
\end{eqnarray*}
Notice also that when $p+q \geq \K$ the indicator $\mathbbm{1}_{p+q \geq s_{0}}$ is automatically satisfied and the previous inequality is in fact an equality as desired.  \end{proof}

 Notice that the $h$-transformation of the $\nu$-random walk by the function $h_{ \mathrm{pre}}$ is usually called the $\nu$-random walk conditioned to die continuously at $0$. It has been proven by Caravena \& Chaumont \cite{CC08} that the scaling limit of this conditioned walk is given by the analogous $\gamma$-stable L\'evy process conditioned to die continuously at $0$. Although we believe that the same scaling limits result holds for the process $(S)$ under $ \mathbb{P}_{p}^{\circ}$ we settle for  rough estimates that will be useful in the next section:
 \begin{lemma}[Absorption and height estimates for the pointed exploration]\label{lem:absorption_pointed} Recall that $\tau$ denotes the absorption time of the process $(S)$ under $ \mathbb{P}_{p}^{\circ}$. There exists some $ \mathrm{c}>0$ such that  for all $A>1$ sufficiently large 
 \begin{itemize}
 \item In the case when $\beta = 3/2$
  \begin{eqnarray*}  
  \mathbb{P}_{p}^{\circ}\big( \tau \geq A p^{1/2}\big) \leq \mathrm{e}^{- \mathrm{c} A^{\mathrm{c}}}, \qquad \forall p \geq 1.  \end{eqnarray*}
  \item In the case when $\beta = 5/2$
  \begin{eqnarray*} \mathbb{P}_{p}^{\circ}\big( \sup_{0 \leq i \leq \tau} S_{i} \geq A p\big) \leq   \mathrm{c} A^{-\mathrm{c}} \quad \mbox{ and } \quad \mathbb{P}_{p}^{\circ}\big( \tau \geq A p^{3/2}\big) \leq  \mathrm{c} A^{- \mathrm{c}}, \qquad \forall p \geq 1. 
 \end{eqnarray*}
 \end{itemize}
   \end{lemma}
   \begin{proof} Recall that under $ \mathbb{P}^{\RW}_{p}$ the process $(S)$ is a random walk with i.i.d.\ increments bounded from above and in the domain of attraction of the spectrally negative $\gamma$-stable  law. Let us start with the estimate on $\mathbb{P}_{p}^{\circ}\big( \sup_{0 \leq i \leq \tau} S_{i} \geq A p\big)= \mathbb{P}_{p}^{\circ}\left( \tau_{\geq Ap} < \infty\right)$ in the case $\beta=5/2$.  
   By the pointed Key formula we have 
  \begin{eqnarray*} \mathbb{P}_{p}^{\circ}\left( \tau_{\geq Ap} < \infty\right) &\underset{ \mathrm{Prop.\ } \ref{prop:llprw}}{\leq}& \mathbb{E}_{p}^{\RW} \left[ \mathbbm{1}_{\tau_{\geq Ap} < \infty}  \cdot \frac{ \widetilde{W}^ \circ_{S_{\tau_{\geq Ap}}}}{ \widetilde{W}^ \circ_{S_{0}}}\right]  \underset{  \mathrm{Lem.}\  \ref{lem:wtildepoint}}{ \lesssim } \mathbb{P}_{p}^{\RW} \left(\tau_{\geq Ap} < \infty  \right)  \cdot A^{-1/2}.  \end{eqnarray*} We can then upper bound the probability by $1$ which gives the desired estimate. 
  \\
For the second estimate on the tail of the absorption time,    we start by estimating the time it takes for the process to drop below level $p/2$ when starting from $p$.  We have
   \begin{eqnarray*} \mathbb{P}_{p}^{\circ}\left( \tau_{  p/2} \geq A p^{\gamma}\right) &\underset{ \mathrm{Prop.\ } \ref{prop:llprw}}{\leq}& \mathbb{E}_{p}^{\RW} \left[ \mathbbm{1}_{\tau_{  p/2} \geq A p^{\gamma}}  \cdot \frac{ \widetilde{W}^{\circ}_{S_{Ap^{\gamma}}}}{ \widetilde{W}^{\circ}_{S_{0}}}\right]  \underset{  \mathrm{Lem.}\  \ref{lem:wtildepoint}}{ \lesssim }  \mathbb{P}_{p}^{\RW} \left(\tau_{ p/2} \geq A p^{\gamma} \right).  \end{eqnarray*}
In the heavy-tail case $\beta = 3/2$, recalling from \eqref{eq:sommedelta} that $ S_i \leq p + \K i$ under $ \mathbb{P}_p^\RW$ we have crudely
  \begin{eqnarray} \label{eq:tailexpo32} \mathbb{P}_{p}^{\RW} (\tau_{ p/2} \geq A p^{1/2} ) \leq \prod_{i=0}^{A p^{1/2}} \big(1- \nu (-p - \K i)\big) \lesssim  \left( 1- \frac{ \mathrm{Cst}}{ p^{1/2}} \right)^{A p^{1/2}} \leq \mathrm{e}^{- \mathrm{c} A}  \end{eqnarray} for some $ \mathrm{Cst}, \mathrm{c}>0$. However, in the case $ \beta = 5/2$ we have  that $\mathbb{P}_{p}^{\RW} \left(\tau_{ p/2} \geq A p^{3/2} \right) \leq \mathrm{c} A^{-  \mathrm{c}}$ for some $ 1>\mathrm{c}>0$ by the Lemma \ref{lem:tailhitting}  proved below.  In the case $ \beta = 5/2$, this implies that 
   $$ \mathbb{E}_{p}^{\circ}\big[ \big(\tau_{ p/2}\big)^{ \mathrm{c}/2} \big] \lesssim (p^{3/2})^{c/2}.$$
To deduce from these estimates an upper bound on the tail of $\tau_{\K}$, consider the sequence of stopping times defined by $\theta_{0}=0,\theta_{1} = \tau_{p/2}$ and recursively $\theta_{j} = \inf\{ i \geq \theta_{1} : S_{i} \leq S_{\theta_{j-1}}/2\}$. By applying the previous inequality at those stopping times we deduce that, still in the case $\beta = 5/2$ we have 
   $$ \mathbb{E}_{p}^{\circ}\big[ (\tau_{\K})^{ \mathrm{c}/2} \big] \underset{c/2 < 1}{\leq} \sum_{j \geq 0 } \mathbb{E}_{p}^{\circ}[(\theta_{j+1}-\theta_{j})^{c/2}] \lesssim p^{ \frac{3}{2}\cdot \frac{c}{2}} + (p/2)^{\frac{3}{2}\cdot \frac{c}{2}} + (p/4)^{\frac{3}{2}\cdot \frac{c}{2}}+ ... \ \preceq \   p^{\frac{3}{2}\cdot \frac{c}{2}}.$$ To pass from the stopping time $\tau_{\K}$ to the absorption time $\tau \geq \tau_{\K}$, notice that there exists $c>0$ and $N\geq 1$ such that once $(S)$ is below level $\K$, it has probability at least $c>0$ of being absorbed within the next $N$ units of time, and that if $(S)$ decides to overcome level $\K$ again, then we have $ S_{\tau_{\geq \K}} \leq 2 \K$ by \eqref{eq:sommedelta}. Hence $\tau- \tau_{\K}$ can be stochastically upper bounded by 
  \begin{eqnarray} \label{eq:stotau} \tau - \tau_{\K} \leq N \sum_{i=1}^{ \mathrm{Geo}(c)}(1 + \theta^{{(i)}}),  \end{eqnarray} where $ \mathrm{Geo}(c)$ is a geometric random variable with success parameter $c>0$ and  the $\theta^{{(i)}}$ are i.i.d.\ random variables independent of $ \mathrm{Geo}(c)$ and whose law dominates that of $ \tau_{\K}$ under $ \mathbb{P}_{p}^{\circ}$ for all $ \K \leq p \leq 2 \K$. In particular, we can suppose that $ \mathbb{E}[(\theta^{{(1)}})^{{c/2}}] < \infty$ and we deduce that
   $$ \mathbb{E}_{p}^{\circ}[\tau^{{c/2}}] \underset{c/2 < 1}{\leq} \mathbb{E}_{p}^{\circ}\big[\big(\tau_{\K}\big)^{{c/2}}\big] + N\mathbb{E}[ \mathrm{Geo}(c)]  \cdot \mathbb{E}\big[ \big( \theta^{(1)}\big)^{{c/2}}\big] \ \ \lesssim \ \ p^{\frac{3}{2}\cdot \frac{c}{2}}.$$
   The statement of the lemma then follows from Markov's inequality since $$ \mathbb{P}_{p}^{\circ}( \tau \geq A p^{3/2}) \leq \mathbb{E}_{p}^{\circ}[\tau^{c/2}]/ (Ap^{{3/2}})^{{c/2}}.$$
   
   In the case $\beta=3/2$ we can write similarly for some small $\delta>0$
   $$ \mathbb{E}_{p}^{\circ}\left[ \exp\left( \frac{\delta}{ \sqrt{p}} \tau_{\K}\right)\right] = \mathbb{E}_{p}^{\circ}\left[ \prod_{j \geq 0} \exp\left(   \frac{1}{ \sqrt{2^j}} \frac{\delta}{ \sqrt{2^{-j}p}} (\theta_{j+1}-\theta_j) \right)\right] \underset{ \eqref{eq:tailexpo32}}{\lesssim}  \prod_{j \geq 0} \left( 1 + \mathrm{cst}  \sqrt{2^{-j}} \right)  < \infty.$$
Using again \eqref{eq:stotau} with this time $\theta^{{(i)}}$ having exponential moments, we deduce that $$\sup_{p \geq 1} \mathbb{E}_{p}^{\circ}\left[ \exp\left( \frac{\delta}{ \sqrt{p}} \tau\right)\right] < \infty$$ and derive the tail estimate via Markov's inequality as well.
     \end{proof}
   
   \begin{lemma}[Tails for dropping below $p/2$ when $\beta = 5/2$] \label{lem:tailhitting} In the case $\beta= 5/2$ there exists $ \mathrm{c}>0$ such that for all $A \geq 1$ large enough and all $p\geq 1$ we have 
   $$ \mathbb{P}^{\RW}_{p}( \tau_{p/2} \geq A p^{3/2}) \leq \mathrm{c} A^{- \mathrm{c}}.$$
\end{lemma}
\begin{proof}
Recalling that $\tau_{\geq x} = \inf\{i\geq 0: S_{i} \geq x\}$ we have 
 \begin{eqnarray*}
 \mathbb{P}^{\RW}_{p}( \tau_{p/2} \geq A p^{3/2}) \leq  \mathbb{P}^{\RW}_{p}( \tau_{ \geq \sqrt{A} p} < \tau_{p/2}) +  \mathbb{P}^{\RW}_{p}( \tau_{p/2} \geq A p^{3/2} \mbox{ and }  \tau_{ \geq \sqrt{A} p} > \tau_{p/2})  \end{eqnarray*}
 For the second term we argue that if $ p/2 \leq q \leq \sqrt{A}p$, then by Lemma \ref{lem:scalinglimit} we have the existence of some $c >0$ such that 
 $$ \inf_{ p/2 \leq q \leq \sqrt{A}p}\mathbb{P}_{q}^{\RW}\Big( \tau_{0} \leq \big(\sqrt{A}p\big)^{3/2} \Big) \geq \inf_{ p/2 \leq q \leq \sqrt{A}p}\mathbb{P}_{q}^{\RW}\Big( \tau_{0} \leq q^{3/2} \Big) \geq  c.$$
In words, if the walk stays below $ \sqrt{A}p$, in each time interval of length $\big(\sqrt{A}p\big)^{3/2} $ it has a positive probability to drop below $p/2$. By successive applications of Markov property at time $i \cdot \big(\sqrt{A}p\big)^{3/2}$ for $i=0,1,2,...$ we deduce that the second term in the above display is bounded above by $c^{ A/( \sqrt{A})^{3/2}} = c^{A^{{1/4}}}$. 

The first term is treated similarly. Again by the scaling limit result Lemma \ref{lem:scalinglimit} we have the existence of some constant $c>0$ such that 
$$ \liminf_{p \to \infty}\mathbb{P}_{p}^{\RW}(\tau_{\geq 2p} < \tau_{0}) >c>0.$$
Again in words, once the walk has reached level $p$, it has some chance to drop below $0$ before reaching the level $2p$. Applying again the Markov property at the successive hitting times $ \tau_{\geq 2p}<\tau_{\geq 4p} < \tau_{\geq 8p}...$ (notice that those stopping times are distinct by \eqref{eq:sommedelta}) we deduce that the first term in the penultimate display is bounded above by $c^{ \log_{2}( \sqrt{A})-1}$. The lemma follows.
\end{proof}

Let us deduce two useful estimates on $ \mathrm{Vol}^ \circ$ using the preceding lemma. 
\begin{lemma}[Uniform integrability] \label{lem:ui} In both cases $ \beta \in \{3/2, 5/2\}$  the family of laws 
$$ \left( \frac{ \mathrm{Vol}^\circ}{p^{\beta- \frac{1}{2}}}  \mbox{ under }\mathbb{P}_p: p \geq 0 \right) \mbox{ and } \left( \frac{ \mathrm{Vol}^\bullet}{p^{\beta- \frac{1}{2}}}  \mbox{ under }\mathbb{P}_p: p \geq 0 \right) \mbox{ are uniformly integrable}.$$
\end{lemma}

\begin{proof} Let us start with the case of $\circ$. Since by \eqref{eq:biasing} the law $ \mathbb{P}^\circ_p$ is obtained by biaising $ \mathbb{P}_p$ with the variable $ \mathrm{Vol}^\circ$ it suffices to show that the sequence $ \left( { \mathrm{Vol}^\circ} \cdot {p^{-\beta+ \frac{1}{2}}}  \mbox{ under }\mathbb{P}_p^\circ : p \geq 0 \right)$ is tight. Let us fix $A> 1$ and let $ \mathcal{E}_p$ be the event $ \{ \tau < A p^{\beta-1}\} \cap \{ \sup_{i \geq 0} S_i \leq Ap \}$. Recall that by definition, we have
\begin{eqnarray*} \mathbb{E}_p^ \circ \left[ \mathrm{Vol}^ \circ \mathbbm{1}_{ \mathcal{E}_p} \right] &\leq& \mathbb{E}_p^ \circ \left[ \mathbbm{1}_{ \mathcal{E}_p}\left( \tau + \sum_{i =1}^{ \tau}\sum_{j \geq 1} \mathbb{E}_{Y_j^i} \left[ \mathrm{Vol}^ \circ \right ] \right)\right ] \\
&\leq&   A p^{\beta-1} \left(1 +  \K \cdot  \sup_{r \leq A p} \mathbb{E}_r^ \circ \left[\mathbb{E}_{Y_1^1} \left[ \mathrm{Vol}^ \circ \right ] \right ]\right).\\
& \underset{  \mathrm{below\ }\eqref{eq:biasing}}{\lesssim} & A p^{\beta-1} \left(1 +  \K \cdot  \sup_{r \leq A p} \mathbb{E}_r^ \circ \left[ (Y_1^1)^{\beta-1/2} \right ] \right)\\
& \underset{ }{\lesssim} & A p^{\beta-1} \left(1 +  \K \cdot  \sup_{r \leq A p} \mathbb{E}_r^ \circ \left[ \big( \sum_{j \geq 1}Y_j^i\big)^{\beta-1/2} \right ] \right).
\end{eqnarray*}
Recall from \eqref{eq:sommedelta} that $ |(S_{1}-S_{0}) + \sum_{j \geq 1} Y_{1}^{j}| \leq \K^{2}$ under $ \mathbb{P}^\circ_p$. It follows then from the definition of the pointed exploration that the law of $ \mathbb{Y} = (\sum_j Y_1^j)$ satisfies the inequality  $$\mathbb{P}_p^\circ ( \mathbb{Y} = y)  \ \ \lesssim \ \  \mathbb{P}_p^\circ( \Delta S = -y) \underset{  \eqref{eq:tailnu} \ \& \ \mathrm{Lem. \ } \eqref{lem:wtildepoint}}{\lesssim} y^{-\beta} \frac{ \sqrt{p}}{ \sqrt{p-y}}\cdot \mathbbm{1}_{1 \leq y \leq p-1}.$$
In particular we have 
$$\mathbb{E}_p^ \circ \left[ \big( \sum_{j \geq 1}Y_1^j\big)^{\beta-1/2} \right ] \ \ \lesssim \ \ \sum_{\ell=1}^{p-1} \ell^{-\beta} \frac{ \sqrt{p}}{ \sqrt{p-\ell}} \ell^{\beta-1/2} \ \ \lesssim \ \ \sqrt{p}.$$
Plugging back into the penultimate display we find that 
$$  \mathbb{E}_p^ \circ \left[ \mathrm{Vol}^ \circ \mathbbm{1}_{ \mathcal{E}_p} \right] \ \ \lesssim \ \  A p^{\beta-1} \sqrt{Ap} = A^{ \frac{3}{2}}p^{\beta- \frac{1}{2}}.$$
This proves tightness since we have for $B > A >0$
  \begin{eqnarray*} \mathbb{P}_p^{\circ}( \mathrm{Vol}^\circ \geq B p^{\beta- \frac{1}{2}}) &\underset{ \mathrm{Markov}}{\leq}&  \mathbb{P}_p^\circ (\mathcal{E}_p^c) + \frac{\mathbb{E}_p^ \circ \left[ \mathrm{Vol}^ \circ \mathbbm{1}_{ \mathcal{E}_p} \right]}{B p^{\beta- \frac{1}{2}}},  \end{eqnarray*}
where $\mathcal{E}_p^c$ denotes the complementary event of $ \mathcal{E}_p$.  The first probability can be made arbitrarily small by taking $A$ large enough by Lemma \ref{lem:absorption_pointed},  while the second one can made arbitrarily small by adjusting $B$ large enough (once $A$ has been fixed).  

To get that case of $\bullet$ notice that by Lemma \ref{lem:volcompa}, we have for all large $A>0$ and $p$'s
$$ \mathbb{P}_{p}( \mathrm{Vol}^{\bullet} \geq A p^{\beta-1/2}) \leq \mathbb{P}_{p}( \mathrm{Vol}^{\circ} \geq \frac{A}{\delta} p^{\beta-1/2}) + \mathrm{e}^{-A p^{\beta-1/2}}.$$Uniform integrability of $(\mathrm{Vol}^\circ p^{-\beta+ \frac{1}{2}})$ under $ \mathbb{P}_{p}$ is the fact that 
$$\lim_{B\to \infty} \sup_{p \geq 1}\sum_{A=B}^{\infty} \mathbb{P}_{p}(\mathrm{Vol}^{\circ} \geq A p^{\beta-1/2}) =0,$$ and the similar statement is deduce for $ \mathrm{Vol}^{\bullet}$ using the previous two displays. 
\end{proof}
  
  Using a similar ideas  we prove that in the case $\beta = 3/2$  the $\circ$-volume has all its moments (the estimate is very rough since we in fact expect that $ p^{-1}\cdot\mathrm{Vol}^{\circ}$ has exponential moments under $ \mathbb{P}_{p}$). This contrasts with the case $\beta=5/2$ where already the second moment blows up.   This will enable us to settle the dichotomy $\beta \in \{3/2,5/2\}$ as being the same as $x < x_{\crit}$ or $x= x_{\crit}$ in Proposition \ref{prop:dichotomyxxc}.
  
  \begin{lemma}[Dichotomy (III), with volumes] \label{lem:polyvol} \label{dicho3} If $\beta =  \frac{3}{2}$ then we have 
  $$ \forall p, k \geq 0, \qquad \mathbb{E}_{p}\big[ \big(\mathrm{Vol}^{\bullet}\big)^{k}\big] < \infty \mbox{ and } \mathbb{E}_{p}\big[ \big(\mathrm{Vol}^{\circ}\big)^{k}\big] < \infty$$
  whereas if $\beta= \frac{5}{2}$ then 
  $$ \forall p \geq 0, \qquad \mathbb{E}_{p}\big[ \big(\mathrm{Vol}^{\bullet}\big)^{2}\big] = \infty \mbox{ and } \mathbb{E}_{p}\big[ \big(\mathrm{Vol}^{\circ}\big)^{2}\big] = \infty. $$
  \end{lemma}

  \begin{proof} Recall that by Lemma \ref{lem:volcompa}, $ \mathrm{Vol}^ \circ$ and $ \mathrm{Vol}^ \bullet$ have the same finite moments. Let us start with the first item and suppose $\beta=3/2$. We will actually show by induction on $k \geq 1$ that for all $ p \geq 0$,
  \begin{eqnarray} \label{eq:HR} \mathbb{E}_{p}[ (\mathrm{Vol}^{\circ})^{k}] \lesssim p^{d_k},  \end{eqnarray}
for some degree $d_k\geq 1$.  For $k=1$, this is granted by \eqref{eq:biasing}. Let us now fix $ k \geq1$ and suppose \eqref{eq:HR} up to power $k$. By \eqref{eq:biasing} again we have
\begin{eqnarray*} 
\mathbb{E}_p \left[ ( \mathrm{Vol}^ \circ)^{k+1} \right ] &=&\mathbb{E}_p^ \circ \left[ ( \mathrm{Vol}^ \circ)^{k} \right ]  \frac{ \widetilde{W}_p^ \circ}{ \widetilde{W}_p} \\ &\lesssim& p \cdot \mathbb{E}_p^ \circ \left[ ( \mathrm{Vol}^ \circ)^{k} \right ]\\
& \lesssim & p \cdot  \mathbb{E}_p^ \circ \left[ \left( \tau +  \sum_{i =1}^{\tau} \sum_{j \geq1}  \mathrm{Vol}(Y_j^i)\right)^{k} \right ]  
\end{eqnarray*}
where $ \mathrm{Vol}(Y_{j}^{i})$ is the $\circ$-volume of the tree starting from the particle of label $Y_{j}^{i}$. Once we condition on $\tau, (Y_{i}^{j})_{1 \leq i \leq \tau, j \geq 0}$,  one can expand the sum and get at most $(\tau ( \K+1))^{k}$ term, each of them being of the form 
$$   \tau^{\ell_{1}} \big( \mathrm{Vol}(Y_{j_{2}}^{i_{2}})\big)^{\ell_{2}} \cdots \big( \mathrm{Vol}(Y_{j_{k}}^{i_{k}})\big)^{\ell_{k}},$$ with exponents $\ell_{\cdot} \leq k$ and $Y_{j}^{i} \leq \max S + \K \leq p + \K \tau + \K$ so that taking expectation and using the induction hypothesis the penultimate display is (crudely) bounded above by 
$$ \lesssim\ \  p \cdot \mathbb{E}_{p}^{\circ}\left[ \tau^{2k}  \left( (p + \K \tau ) ^{ d_{k}}\right)^{k}\right].$$
By the first item of Lemma \ref{lem:absorption_pointed} we know that $\tau$ has exponential tail in the scale $ \sqrt{p}$ so that the above display is bounded above by $ \lesssim p^{k + p\cdot {d}_{k}\cdot k +1}$ as desired.

In the case $\beta = 5/2$, consider a $\nu$-random walk started from $\K+1$ and stop it when it drops below $\K$. By the cyclic lemma and the local limit theorem \cite[Theorem 4.2.1]{IL71} we have 
  \begin{eqnarray}  \mathbb{P}_{\K+1}^{\RW}( \tau_{\K} = n \mbox{ and } S_{\tau_{\K}} =\K) &\underset{ \mathrm{cyclic}}{\geq}&  \frac{1}{n} \mathbb{P}_{0}^{\RW}(S_{n}=-1)\\
  & \underset{ \mathrm{LCLT}}{ \gtrsim} & \frac{n^{{-2/3}}}{n}  = n^{-5/3}.  \end{eqnarray}
Notice that on this event we obviously have $ \mathrm{Vol}^{\bullet} \geq n$.  
Using the pointed Key formula, we thus deduce that $$ \mathbb{E}_{\K+1}[ \mathrm{Vol}^{\bullet} \mathrm{Vol}^{\circ}] =  \mathbb{E}_{\K+1}^{\circ}[ \mathrm{Vol}^{\bullet}] \gtrsim \sum_{n} n^{-5/3} \cdot n = \infty,$$ which implies by \eqref{eq:biasing} that $\mathbb{E}_{\K+1}[ (\mathrm{Vol}^{\bullet})^2] = \infty$.  
\end{proof}

\section{Proofs of the main results}
We can finally reap the rewards of our work. First of all, we will show that the dichotomy $ \beta  \in \{3/2 , 5/2\}$ is nothing more than the dichotomy $ x< x_{\crit}$ or $x= x_{\crit}$. We then verify the assumptions of the powerful invariance principle of \cite{bertoin2024self} to prove Theorem \ref{thm:scaling}. The convergence of the volume is then used to ultimately prove Theorem \ref{thm:5/2}. Let us start with a discussion on the profound algebraicity result of Bousquet--M\'elou and Jehanne \cite{BMJ06} which we have not yet used. \medskip 

The influential result \cite[Theorem 3]{BMJ06} of Bousquet--M\'elou and Jehanne shows that the solution $F(x,y)$ to  \eqref{eq:1} is algebraic. In particular, our partition function $x \mapsto W_{0}^{x}$ is also algebraic so that it is amenable to standard singularity analysis \cite[Chapter VII]{FS09}: Under the last item of our standing assumptions $(*)$, it is a unique dominant singularity so that we have 
  \begin{eqnarray} [x^{n}] W_{0}^{x} \approx (x_{\crit})^{-n} \cdot n^{-\alpha},   \label{eq:BMJtail}\end{eqnarray}
for some $\alpha >0$ which we will  pin down in this section. As a first remark, notice that $W_{0}^{x}$ blows up polynomially as $x=x_{\crit}$ and combined with Lemma \ref{lem:polyvol} this settle the dichotomy that we proved along Propositions   \ref{prop:nodrift}, \ref{dicho2}  and \ref{dicho3} (recall that in the previous sections the dependence in $x$ was implicit)

\begin{proposition}[Dichotomy (IV), via $x$] \label{prop:dichotomyxxc} The dichotomy $\beta \in \{3/2, 5/2\}$ corresponds to 
 $$  \mathbb{E}[\nu]= \infty \iff \beta = \frac{3}{2} \iff x < x_\crit \quad \mbox{ and } \quad  \mathbb{E}[\nu] =0 \iff \beta = \frac{5}{2} \iff x=x_\crit.$$
\end{proposition}
\begin{proof} Notice that when $x < x_{\crit}$ then $[y^{p}]F(x \mathrm{e}^{ \varepsilon},p) < \infty$ for some $ \varepsilon>0$, which implies that \linebreak[4]
$ \mathbb{E}_{p}\big[ \exp( \varepsilon \mathrm{Vol}^{\bullet})\big]< \infty$ hence $ \mathbb{E}_{p}\big[ (\mathrm{Vol}^{\bullet})^{2}\big]< \infty$, so that by Lemma \ref{lem:polyvol}, we are in the $\beta = 3/2$ case. Conversely, if we are in the $\beta=3/2$ case then $\mathrm{Vol}^{\bullet}$ must have all its moments.  This implies that $x$ cannot be $x_{\crit}$ since by \eqref{eq:BMJtail} the variable $ \mathrm{Vol}^{\bullet}$ has no moment of order larger than $\alpha$ at $x= x_{\crit}$. \end{proof} 

Another consequence of algebraicity is the following sharpening of Lemma \ref{lem:wtildepoint}:
\begin{proposition} \label{prop:equiv!}Regardless of $\beta \in \{3/2, 5/2\}$, the pointed partition functions  satisfy
\begin{eqnarray*} \widetilde{W}_{p}^{\circ} \ \ \approx \ \  p^{-1/2} \quad \mbox{ and } \quad \widetilde{W}_{p}^{\bullet} \ \ \approx \ \  p^{-1/2} \quad \mbox{ as } p \to \infty,  \end{eqnarray*} where we recall that $a_{n}\approx b_{n}$ means that $(a_{n}/b_{n})$ converges towards a positive constant as $n \to \infty$.
\end{proposition}
\begin{proof} Recall that $ \sum_{p\geq 0} W_{p}^{\bullet} y^{p} = x \partial_{x} F(x,y) =: F^{\bullet}(y)$. Proceed as in the proof of Lemma \ref{lem:asympW}, start from \eqref{eq:1}, multiply by a large power of $y$, differentiate with respect to $x$ and then set $x$ to a fixed value yields a polynomial equation $ P_{1}(y, F(y),F^{\bullet}(y))=0$. Combining this equation with the equation $P_{2}(y,F(y))=0$ already obtained in Lemma \ref{lem:asympW} yields an equation $P_{3}(y, F^{\bullet}(y))=0$ proving algebraicity of $F^{\bullet}$. As in the proof of Lemma \ref{lem:asympW} we deduce that $\widetilde{W}_{p}^{\bullet}$ satisfy the asymptotic 
$$ \widetilde{W}_{p}^{\bullet} \underset{\substack{p \to + \infty\\ p \equiv j\, \mathrm{mod}\, b}}{\sim} C_j \cdot  p^{- 1/2},$$ for some constants $C_{0}, ... , C_{b-1}$ and a period $b \in \{1,2, ... \}$. We now follow the same proof as in Proposition \ref{prop:nu_proba} and argue that we have 
$$ C_j \geq \sum_{-j \leq q \leq b -1- j} C_{j+q} \nu(b \mathbb{Z} + (j+q)),$$
which is sufficient to imply the equality of those constants by the maximum principle. The case of $ \widetilde{W}^{\circ}_{p}$ is similar and left to the reader.\end{proof}

\subsection{Scaling limits of fully parked trees}
Armed with the estimates gathered in the preceding sections, we can now embark in the proof of our main results Theorem \ref{thm:5/2} and Theorem \ref{thm:scaling}. We shall actually first prove Theorem \ref{thm:scaling} since it mainly consists in checking the assumptions of  the general invariance principle established in \cite{bertoin2024self}.  
Let us recall it here for the reader's convenience: Suppose $(*)$ and denote by $\mu_{ \mathbf{t}}$ either the counting measure or the counting measure on vertices of label $\K$ in $ \mathbf{t}$. Then there exists some constants $ s^{x},v^{x}>0$ such that 
\begin{itemize}
\item When $x < x_{\crit}$ 
$$  \left(    s^{x} \cdot \frac{ \mathrm{t}}{ p^{1/2}}, \frac{\phi}{p}, v^{x} \cdot \frac{\mu_ \mathbf{t}}{p}\right) \mbox{ under } \mathbb{P}_p \xrightarrow[p\to\infty]{(d)}  \boldsymbol{ \mathcal{T}}_{1/2},$$
\item When $x=x_{\crit}$
$$  \left(   s^{x_{\crit}} \cdot \frac{ \mathrm{t}}{ p^{3/2}}, \frac{\phi}{p}, v^{x_{\crit}}\cdot \frac{\mu_ \mathbf{t}}{ p^{2}}\right) \mbox{ under } \mathbb{P}_p \xrightarrow[p\to\infty]{(d)}  \boldsymbol{ \mathcal{T}}_{3/2},$$
\end{itemize} 
the above convergence holds for the Gromov--Hausdorff--Prokhorov hypograph convergence developed in \cite{bertoin2024self}. We also recall that self-similar Markov trees are random decorated trees $ \boldsymbol{ \mathcal{T}}= ( \mathcal{T}, \mathrm{d}, \rho, g, \mu)$ where $( \mathcal{T}, \mathrm{d},\rho)$ is a rooted real tree, $g : \mathcal{T} \to \mathbb{R}_+$ is its decoration and $\mu$ is a probability measure on it. In the case of the Brownian CRT $ \boldsymbol{\mathcal{T}}_{1/2}$ we have $g_{1/2}(x) = \mu ( \mathcal{T}_{1/2}[x])$ where $ \mathcal{T}_{1/2}[x]$ is the subtree above the point $x \in \mathcal{T}_{1/2}$. We refer the reader to \cite[Examples 3.6, 3.9 + remark after it]{bertoin2024self} for details of the construction of those objects and for the precise notion of convergence. In the rest of this work, we shall only use the down-to-earth convergence of the total mass $$  v^{x_\crit} \cdot  \frac{ \mathrm{Vol}^\circ}{ p^{2}}  \mbox{ under } \mathbb{P}_p  \xrightarrow[p\to\infty]{(d)}  \mu( \mathcal{T}_{3/2}),$$
where the last law is a size-biased transform of a 1/2-stable law (in particular for us with positive density over $ \mathbb{R}_+$).

\begin{proof}[Proof of Theorem \ref{thm:scaling}]
We need to check the Assumptions of Theorem 6.9 of \cite{bertoin2024self}. \\
The Assumption 6.8 there is directly given by Lemma \ref{lem:ui} and Proposition \ref{prop:equiv!}. To check Assumption 6.4 we need to exhibit several super-harmonic functions with the correct growth property. 
 In the case $x = x_ \crit$, recalling that the limit should be the ssMt $ \boldsymbol{ \mathcal{T}}_{3/2}$, the two roots of the cumulant function are $\omega_- = 2$ and $\omega_+ = 3$ (see Example 3.9 in \cite{bertoin2024self}). In the case $x < x_ \crit$, for the Brownian CRT $ \boldsymbol{ \mathcal{T}}_{1/2}$, we have $\omega_- = 1$ and $\omega_+ = + \infty$ (see Example 3.9 in \cite{bertoin2024self}). 
Thus in both cases, we shall verify Assumption 6.4 of \cite{bertoin2024self}   
 with $ \gamma_0 = \beta^{ x}$ and $\gamma_1 = \beta^{ x} - \frac{1}{2}$ so that we indeed have $ \omega_- = \gamma_1 < \gamma_0 < \omega_+$. \\
First, the harmonic function $\phi_1$ satisfying $\phi_1(n) \approx n^{\gamma_1}$ there is provided by the function 
$$  \phi_1(p) :=  \frac{\widetilde{W}^{\circ,x}_{p}}{ \widetilde{W}_{p}^x} = \mathbb{E}_p[ \mathrm{Vol}^\circ].$$
On the one hand, by \eqref{eq:biasing} one has $  \phi_1(p) \threesim p^{\beta^{ x} - \frac{1}{2}}$. On the other hand, the Markov property of the label tree and the fact that $\phi_1(p)=\mathbb{E}_p[ \mathrm{Vol}^\circ]$ makes it clear that this function is super-harmonic for the multitype Bienaym\'e--Galton--Watson process (it is even in fact harmonic at all points except $p = \K$). Second, for $\phi_0(p)$ we shall take 
$$ \phi_0(p) := \frac{1}{\widetilde{W}_{p}^x}.$$
By \eqref{eq:tailnu} we have $\phi_0(p) \approx p^{\beta^x}$. Let us check now that $\phi_0$ is super-harmonic for the Bienaym\'e--Galton--Watson process using the same calculations as in the proof of Proposition \ref{prop:llprw}: For $p \geq 0$ we have
  \begin{eqnarray*}  &&\sum_{k\geq 0}\sum_{p_1, ..., p_k} \pi_{p}^x(p_{1}, ... ,p_{k}) \sum_{j=0}^k \frac{1}{ \widetilde{W}_{p_j}^x} \\
  &\underset{ \mathrm{exch.}}{=}& \sum_{k\geq 0} (k+1) \sum_{q, p_1, ..., p_k} \pi_{p}^x(p +q, ... ,p_{k}) \frac{1}{ \widetilde{W}_{p_0}^x} \\
    &\underset{ \mathrm{Prop.\  }\ref{prop:GW}}{=}& x \cdot \sum_{\substack{c, k , q \geq 0} } (k+1) \sum_{ \substack{ (s_0, \dots, s_k) \\ \forall i , 0 \leq s_i \leq p_i \\  \sum_{i = 1}^{k} p_i - \sum_{i = 0}^{k}  s_i+ c = -q}} w_{c, k+1, (s_0, \ldots, s_k)} y^{-\sum_i s_i} y^c  \cdot \frac{  \widetilde{W}_{p_1}^x \ldots \ldots \widetilde{W}_{p_k}^x}{ \widetilde{W}_p^x}\\
    &\underset{ \mathrm{Def. \ } \ref{def:nuhat}}{=}&  \frac{1}{ \widetilde{W}_p^x} \left(\sum_{\substack{c, k , q \geq 0} }  \sum_{ \substack{ (s_0, \dots, s_k) \\ (p_1, ... , p_k)}} \hat{\nu}^x(q ; (s_0, ... , s_k), (p_1, ... , p_k))  \mathbf{1}_{p+q \geq s_0} \right). \end{eqnarray*} Since the parenthesis is less than one, we deduce that $\phi_0$ is super-harmonic. More precisely, since $s_0 \leq \K$, the parenthesis is sandwiched between $ 1-\nu^x((-\infty,-p])$ and $ 1-\nu^x([-\infty,-p+ 2 \K])$, both being of order $ 1- \mathrm{Cst}\cdot p^{-(\beta^{ x} -1)}$ by  \eqref{eq:tailnu}. We deduce that the condition on $\phi_0$  required by \cite[Assumption 6.4]{bertoin2024self} is granted. 
Finally,  Assumption 6.3 (with the locally largest exploration as selection rule) is enforced by the arguments developed in Section \ref{sec:lampertistable}. More precisely, Lemma \ref{lem:scalingLLbertoin} shows that the decoration-reproduction process, stopped at the first time $S$ drops below level $\delta p$ converge towards the continuous decoration-reproduction process used to build $  \boldsymbol{\mathcal{T}}_{\beta^x - 1}$. The stronger Assumption 6.3 is then implied by this convergence and the Assumption 6.4 using  \cite[Lemma 6.23]{bertoin2024self}. 
\end{proof}

\subsection{$\alpha= 5/2$ : Proof of Theorem \ref{thm:5/2}}
Recall \eqref{eq:BMJtail} and let us first explain how one should interpret the exponent $\alpha$ probabilistically in order to compute it using our random walk estimates. The exponent $\alpha$ can also be written as:
$$ \mathbb{P}_0^{x_ \crit} \left( \mathrm{Vol}^ \bullet = n\right) = \frac{1}{ W_0^{x_ \crit}} [x^n] W_0^x = \frac{1}{ \widetilde{W}_0^{x_ \crit}} [x^n] \widetilde{W}_0^x \ \ \approx \ \  n^{ - \alpha}$$
For this reason, we fix $x = x_{ \crit}$ in the rest of the paper. We shall use the easier tail estimate
$$ \mathbb{P}_{0}( \mathrm{Vol}^{\bullet} \geq n) \approx n^{-\alpha+1}.$$
We will then transfert this estimate back on $ \mathrm{Vol}^\circ$ to be able to use our pointed measure $ \mathbb{P}^\circ$. To do that, we use Lemma \ref{lem:volcompa} to deduce that the exponent $\alpha$ is also recoverable from the asymptotics $ \mathbb{P}_{0}( \mathrm{Vol}^{\circ} \geq n) \threesim n^{-\alpha+1}$ or equivalently after biasing by $ \mathrm{Vol}^\circ$ (see \eqref{eq:biasing}) that 
   \begin{eqnarray}   \label{eq:usefulproba}\mathbb{P}_{0}^\circ( \mathrm{Vol}^{\circ} \geq n) \threesim n^{-\alpha+2}.  \end{eqnarray}
The idea is that to get $ \mathrm{Vol}^\circ \geq n$ under $ \mathbb{P}^\circ_0$, the most likely scenario is that the process $(S)$ along the pointed branch survives for time $n^{3/4}$ so that it reaches values of order $ \sqrt{n}$ and it is likely to perform a negative jump of order $ \sqrt{n}$. When doing so, the (essentially unique) tree attached to this jump starts from a label of order $ \sqrt{n}$ and so has volume larger than $n$ by Theorem \ref{thm:scaling}, see Figure \ref{fig:scenario}. Let us proceed. \medskip  
\begin{figure}[!h]
 \begin{center}
 \includegraphics[width=14cm]{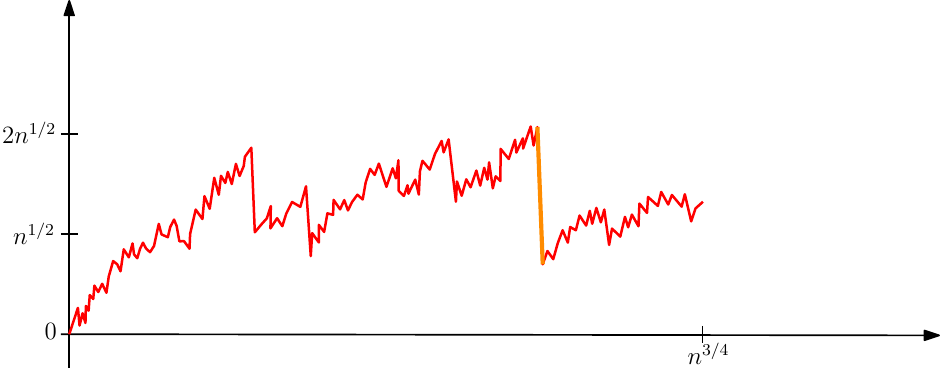}
 \caption{\label{fig:scenario} Illustration of the event $ \mathcal{E}_n$ in \eqref{def:Ep} giving the most likely scenario to reach $ \mathrm{Vol}^\circ \geq n$ under the law $ \mathbb{P}_0^\circ$: a large negative jump of size $ \sqrt{n}$ (in orange) yields a dangling tree of size $\geq n$ with positive probability.}
 \end{center}
 \end{figure}

\textsc{Lower bound.} Suppose that $n^{3/4}$ is an integer to simplify notation and let us consider the event   \begin{eqnarray} \mathcal{E}_n := \left\{S_i \geq \K, \forall 0 \leq i \leq n^{3/4} \ \& \ S_n \in [ \sqrt{n}, 2\sqrt{n}] \ \& \  \exists j \leq n^{3/4}, \ \Delta S_j \leq - \sqrt{n}\right\}.   \label{def:Ep}\end{eqnarray}
First, recall by \eqref{eq:VatVach} that $ \mathbb{P}_\K^{\RW}( S_i \geq \K, \forall 0 \leq i \leq n^{3/4}) \approx n^{-1/4}$. The conditional law of $(S)$ on this event is studied in  \cite{doney1985conditional} and proved to converge in the scaling limit towards the $3/2$-stable L\'evy process conditioned to survive. In particular, it holds that $ \mathbb{P}_\K^\RW( \mathcal{E}_n \mid S_i \geq \K, \forall 0 \leq i \leq n^{3/4}) > c >0$ for some constant $c >0$. Second, conditionally on $ \mathcal{E}_n$ and on the process $(S)$, every large negative jump of order $ \sqrt{n}$ of $(S)$ yields to at least one particle $Y_i^j\geq \sqrt{n}/ (2 \K)$ for $n$ large enough. By the convergence of the measure part\footnote{The reader afraid of our Theorem \ref{thm:scaling} may use instead Lemma \ref{lem:ui} and change $n$ for $c n$ for some small $c>0$.} in Theorem \ref{thm:scaling} we deduce that the $\circ$-volume of the tree started from label $ Y_i^j \geq \sqrt{n}/(2 \K)$ is larger than $n$ with a probability bounded away from $0$, in particular $ \mathbb{P}_\K^\circ( \mathrm{Vol}^\circ \geq n) \gtrsim \mathbb{P}_\K^\circ( \mathcal{E}_n)$. We can then use our pointed Key formula to deduce that 
$$ \mathbb{P}_\K^\circ( \mathrm{Vol}^\circ \geq n) \ \ \gtrsim \ \  \mathbb{P}_\K^\circ \left( \mathcal{E}_n\right) \underset{\mathrm{Prop.\ \ref{prop:llprw}}}{=} \mathbb{E}_\K^{\RW}\left[  \frac{ \widetilde{W}_{S_n}^\circ}{\widetilde{W}_{\K}^\circ} \mathbbm{1}_{ \mathcal{E}_n}\right] \underset{ \mathrm{Lem.\  \ref{lem:wtildepoint}} }{\gtrsim}  {(\sqrt{n})^{-1/2}}\cdot  n^{-1/4} \ \ \gtrsim n^{-1/2}. $$ By transferring the estimate to $ \mathbb{P}_0^\circ$ using aperiodicity, this proves the bound $\alpha \leq \frac{5}{2}$ in \eqref{eq:usefulproba}.\bigskip

\textsc{Upper bound.} To prove the upper bound, one shall establish that the above scenario is the most likely to reach $ \mathrm{Vol}^\circ \geq n$. More precisely, we will show that the event $  \mathrm{Vol}^\circ \geq n$ is unlikely if $(S)$ does not reach values of order $ \sqrt{n}$. We shall first establish that 
  \begin{eqnarray} \label{eq:allerplushaut}\mathbb{P}_\K^{ \circ} \left(\tau_{\geq \sqrt{n}} < \infty\right) \ \ \lesssim \ \  n^{-1/2}.  \end{eqnarray}
To see this, notice first that by the scaling limit result of Lemma \ref{lem:scalinglimit} we have for some $c>0$
$$ \inf_{p \geq \sqrt{n}} \mathbb{P}_{p}^\RW(\tau_{\K} \geq (\sqrt{n})^{3/2}) > c > 0,$$ so that applying Markov property at $\tau_{\geq  \sqrt{n}}$ we have 
$$ \mathbb{P}_\K^{ \RW}( \tau_{\geq \sqrt{n}}< \tau_{\K}) \leq c \cdot \mathbb{P}_\K^ \RW(\tau_{\K} \geq n^{3/4}) \underset{ \eqref{eq:VatVach}}{\approx} n^{-1/4}.$$
By the pointed Key formula (Proposition \ref{prop:llprw}) we deduce that 
$$ \mathbb{P}_\K^\circ( \tau_{ \geq \sqrt{n}} < \tau_ \K) \lesssim (\sqrt{n})^{-1/2} \cdot n^{-1/4}  = n^{-1/2}.$$ 
By aperiodicity, the same estimate is valid even under $ \mathbb{P}_r^\circ$ for $ r \in [\K, 2\K]$. Recall then from \eqref{eq:stotau} that $\tau$ is reached before a geometric number of returns to $[0, \K]$ so that we have the same tail 
$$  \mathbb{P}_{\K}^{\circ}(\tau_{\geq  \sqrt{n}} < \infty) =  \mathbb{P}_{\K}^{\circ}(\tau_{\geq  \sqrt{n}} < \tau) \ \ \lesssim n^{-1/2}.$$
 Second, we fix $ \delta>0$ and shall estimate the expectation of $ \mathrm{Vol}^\circ$ on the event $\tau_{\geq \delta \sqrt{n}} = \infty$. To do this, we proceed as in the proof of Lemma \ref{lem:ui} and get for $ q > \K$ (large) that
  \begin{eqnarray} \mathbb{E}_\K^{ \circ}\left[ \mathbbm{1}_{\tau_{\geq q}= \infty } \mathrm{Vol}^\circ \right] &\lesssim& \mathbb{E}_\K^ \circ \left[ \mathbbm{1}_{\tau_{\geq q} = \infty} \left( \tau + \sum_{i =1}^{ \tau}\sum_{j \geq 1} \mathbb{E}_{Y_j^i} \left[ \mathrm{Vol}^ \circ \right] \right)\right]  \nonumber \\
  & \underset{  \mathrm{below\ }\eqref{eq:biasing}}{\lesssim}  &  \mathbb{E}_\K^\circ\left[ \mathbbm{1}_{\tau_{\geq q} = \infty} \cdot \tau \left(1 +  \K \cdot  \sup_{r \leq q + \K} \mathbb{E}_r^\circ \left[ \big( \sum_{j \geq 1}Y_1^j\big)^{2} \right] \right) \right] \nonumber \\
  & \lesssim & \mathbb{E}_\K^\circ\left[\mathbbm{1}_{\tau_{\geq q} = \infty}\cdot \tau\right] \cdot \sum_{\ell=1}^q \frac{\sqrt{q}}{ \sqrt{q-\ell}} \ell^{-5/2} \cdot \ell^2 \\ &\approx&  \mathbb{E}_\K^\circ\left[\mathbbm{1}_{\tau_{\geq q} = \infty}\cdot \tau\right] \cdot  \sqrt{q}. \label{eq:evaluexpect}
 \end{eqnarray}
To evaluate  $\mathbb{E}_\K^\circ[\mathbbm{1}_{\tau_{\geq q} = \infty}\cdot \tau]$ we split into the possible scale reached by $\sup (S)$ and write $ \mathcal{S}_i$ the event $\{\tau_{\geq 2^{i+1}} = \infty \} \cap \{ \tau_{\geq 2^i} < \infty\}$. We then evaluate  for $A \geq 2$ 
$$\mathbb{P}_\K^\circ\left( \mathcal{S}_i \mbox{ and } \tau \geq 2A (2^i)^{3/2}\right) \leq \mathbb{P}_\K^\circ\left( \mathcal{S}_i \mbox{ and } \left( \tau_{\geq 2^{i}} \geq A( 2^i)^{3/2}  \mbox{ or }\tau - \tau_{\geq 2^{i}} \geq A( 2^i)^{3/2} \right) \right).$$
In the first case, on the event $ \mathcal{S}_i$ and $\tau - \tau_{\geq 2^{i}} \geq A( 2^i)^{3/2}$ applying the Markov property at time $\tau_{\geq 2^i}$ we deduce that 
$$ \mathbb{P}_\K^\circ\left( \mathcal{S}_i  \mbox{ and }\tau - \tau_{\geq 2^{i}} \geq A( 2^i)^{3/2} \right) \leq \mathbb{E}_\K^\circ \left[  \mathbbm{1}_{\tau_{\geq 2^{i}} < \infty}\cdot  \mathbb{P}_{S_{\tau_{\geq 2^{i}}}}^\circ\left(\tau \geq A( 2^i)^{3/2} \mbox{ and } \tau_{\geq 2^{i+1}} = \infty\right)\right].$$
By an argument similar to that used in the proof of Lemma \ref{lem:tailhitting}, on the event $\tau_{\geq 2^{i+1}} = \infty $ the probability that $ \tau \geq A \cdot (2^{i})^{3/2}$ is less than $ \mathrm{e}^{-c\cdot A}$ for some $c>0$ uniformly in $i \geq 0$ since for each time scale of length $(2^{i})^{3/2}$ the process has a probability bounded away from $0$ to be absorbed by Lemma \ref{lem:absorption_pointed} (and there are $A$ such time scales). We deduce that 
$$\mathbb{P}_\K^\circ\left( \mathcal{S}_i  \mbox{ and }\tau - \tau_{\geq 2^{i}} \geq A( 2^i)^{3/2} \right) \leq  \mathbb{P}_\K^\circ(\tau_{\geq 2^{i}} < \infty) \cdot \mathrm{e}^{-c \cdot A} \underset{ \eqref{eq:allerplushaut}}{\lesssim}  \frac{1}{ 2^i} \mathrm{e}^{- c \cdot A}. $$
In the second case, on the event $ \mathcal{S}_i$ and $\tau_{\geq 2^{i}} \geq k( 2^i)^{3/2}$, we use the pointed Key formula  and come back to a random walk estimate under $ \mathbb{P}_\K^\RW$: 
 \begin{eqnarray*}\mathbb{P}_\K^\circ\left( \mathcal{S}_i  \mbox{ and }  \tau_{\geq 2^{i}} \geq A( 2^i)^{3/2}  \right) &\lesssim &  \frac{1}{ \sqrt{2^i}} \mathbb{P}_\K^\RW \left( \mathcal{S}_i  \mbox{ and } A(2^i)^{3/2} \leq \tau_{\geq 2^{i}} < \tau_{-1} \right)\\
 & \leq & \frac{1}{ \sqrt{2^i}} \cdot \mathbb{P}_\K^\RW(\tau_{-1} \geq (2^{i})^{3/2} ) \cdot \sup_{r \leq 2^i} \mathbb{P}_r^\RW \left( (A-1)( 2^i)^{3/2}  \leq  \tau_{\geq 2^{i}} \leq \tau_{-1}\right)\\
 & \underset{  \eqref{eq:VatVach}}{\leq} & \frac{1}{ \sqrt{2^i}} \cdot \frac{1}{ \sqrt{2 ^i}}\cdot  \mathrm{e}^{-c (A-1)},  \end{eqnarray*} by the same reasoning as above. We deduce that $ \mathbb{E}_\K^\circ[\tau \mathbbm{1}_ {\mathcal{S}_i}] \lesssim (2^i)^{3/2-1} = \sqrt{2^i}$, and finally by summing over all $0 \leq i \leq \log_2(q)$ that $ \mathbb{E}_\K^\circ[ \tau  \mathbbm{1}_{\tau_{\geq q} = \infty}] \lesssim \sqrt{q}.$  Coming back into \eqref{eq:evaluexpect} we finally deduce that 
 $$ \mathbb{E}_\K^{ \circ}\left[ \mathbbm{1}_{\tau_{\geq q} = \infty} \mathrm{Vol}^\circ \right] \lesssim q,$$
 so that Markov's inequality yields 
 $$ \mathbb{P}_\K^{ \circ}\left(  \mathrm{Vol}^\circ \geq n \mbox{ and } \tau_{\geq \delta  \sqrt{n}} = \infty \right)Ê\ \  \lesssim  \frac{\delta  \sqrt{n}}{n} = \frac{\delta}{ \sqrt{n}}.$$
 Combined with \eqref{eq:allerplushaut} this gives the matching lower bound on \eqref{eq:usefulproba} and the matching bound $\alpha \geq \frac{5}{2}$ (once transferred to $ \mathbb{P}_0^\circ$ by aperiodicity). \qed

\section{Comments} \label{sec:comments}
We end this work with some comments and possible generalizations.\medskip

The core of our approach is clearly the link with the $\nu$-random walk, in particular via the Key formula. In the case when $\K=1$, the $\nu$-random walk is skip-free ascending. Fluctuation theory for skip-free walks is much simpler than for general walks and this may again explain the particular role played by $\K=1$ in the theory. Similarly, in absence of branching, that is for linear equation \eqref{eq:1} then the step distribution $\nu$ has a finite support making the analysis much less puzzling, see \cite{lalley1995return,lalley2004algebraic}. The link with random walk persists in more challenging cases: one could try to remove the bounded degree assumptions and allow analytic functions in \eqref{eq:1} or even allow heavy tail and classify all possible universality classes. We also hope that the link with random walk may yield new information both in the discrete and in the continuous setting, for example shedding light on the beautiful result of Chapuy \cite{chapuy2024scaling} obtained via analytic methods.

\bibliographystyle{siam}
\bibliography{bibli}

\end{document}